\documentclass[11pt,a4paper]{amsart}

\usepackage{amsfonts,amssymb,amsmath,amsthm}
\usepackage{mathtools,esint}
\usepackage{lmodern}
\usepackage[noadjust]{cite}
\usepackage{enumitem}
\usepackage{needspace}
\usepackage{xcolor}
\usepackage[plainpages=false,pdfpagelabels,backref=page]{hyperref}
\hypersetup{
 hypertexnames=false,
 colorlinks=true,
 linkcolor={cyan!90!black},
 citecolor={magenta},
 urlcolor={green!40!black},
 pdftitle={Fokker-Planck analysis and invariant laws for a continuous-time stochastic model inspired by Adam},
 pdfauthor={Kaj Nystr\"om},
 pdfsubject={Stochastic differential equations and invariant measures}
}

\newtheorem{thm}{Theorem}[section]
\newtheorem{lem}[thm]{Lemma}
\newtheorem{prop}[thm]{Proposition}
\newtheorem{cor}[thm]{Corollary}

\theoremstyle{definition}

\newtheorem{rem}[thm]{Remark}
\numberwithin{equation}{section}

\usepackage{tikz}
\usetikzlibrary{arrows.meta,calc,patterns,decorations.pathreplacing}
\definecolor{KPBlue}{RGB}{35,82,111}
\definecolor{KPGreen}{RGB}{45,100,83}
\definecolor{KPGold}{RGB}{163,115,43}
\tikzset{
  kp/past/.style={draw=KPBlue,densely dashed,line width=.8pt,fill=KPBlue!7},
  kp/future/.style={draw=KPGreen,line width=.85pt,fill=KPGreen!6},
  kp/outer/.style={draw=black!55,dash pattern=on 4pt off 2pt,line width=.65pt},
  kp/arrow/.style={-{Stealth[length=2mm]},draw=black!65,line width=.7pt},
  kp/step/.style={-{Stealth[length=2.2mm]},draw=KPBlue,line width=1pt},
  kp/vertex/.style={draw=black!70,rounded corners=1.5pt,fill=white,
    minimum width=2.05cm,minimum height=1.02cm,inner sep=3pt,
    align=center,font=\small},
  kp/terminal/.style={kp/vertex,double,double distance=1.2pt},
  kp/small/.style={font=\footnotesize,align=center},
  kp/heading/.style={font=\small\bfseries,anchor=west},
  kp/leader/.style={draw=black!65,line width=.5pt},
  kp/axis/.style={-{Stealth[length=1.7mm]},line width=.55pt,draw=black!75}
}

\newcommand{\R}{\mathbb R}

\usepackage{mathrsfs}

\DeclareMathOperator{\Diag}{Diag}

\allowdisplaybreaks
\makeatletter
\renewcommand{\l@section}{%
  \@tocline{1}{3pt}{0pt}{}{\normalfont\small}}
\makeatother

\begin{document}

\title[Invariant laws for an Adam-inspired stochastic model]
{Fokker-Planck analysis and invariant laws for a continuous-time stochastic model inspired by Adam}
\author{Kaj Nystr\"om}
\address{Department of Mathematics, Uppsala University, Box 480\\
SE-751 06 Uppsala, Sweden}
\email{kaj.nystrom@math.uu.se}
\thanks{The author was partially supported by grant 2022-03106 from the Swedish Research Council (VR)}
\renewcommand{\subjclassname}{\textup{2020} Mathematics Subject Classification}

\begin{abstract}
We study a degenerate stochastic system inspired by bias-corrected
Adam. The system arises as the continuous-time limit of an explicitly
modified recursion in which the squared-gradient update is replaced
by a prescribed second-moment closure. We construct an entrance
solution at the singular initial time and prove convergence of the
modified recursion on finite time intervals away from zero. Bias
correction enters through time-dependent coefficients which converge
to constants as time tends to infinity. For a smooth objective with globally Lipschitz gradient and radial
dissipativity at infinity, we prove existence and uniqueness of an
invariant probability measure for the autonomous system under either
$b<4a$ or
$\gamma L_f\le a(\sigma/\sqrt2+\varepsilon)/4$.
Here $L_f$ is the Lipschitz constant of the gradient, $a,b$ are the
relaxation rates, $\gamma$ is the learning-rate parameter, and
$\sigma,\varepsilon$ are the noise and regularization parameters.
The proof combines an exponential Lyapunov function with compact-set
minorization obtained from explicit control paths and a regular
endpoint map. Under an exponential Lyapunov moment condition, we obtain exponential
convergence in weighted total variation and in the squared
$2$-Wasserstein distance of the complete state. The invariant law has
finite moments of every order, and qualitative total variation
convergence holds for every initial probability law. The
time-inhomogeneous process converges weakly to the same invariant law.
We formulate the associated Fokker-Planck equation for probability
measures and establish local hypoellipticity and smoothness of the
invariant density where $\Diag(\nabla f)D^2f$ is invertible.
A quadratic example shows that a uniform exponential convergence
estimate requires a prefactor growing exponentially in the initial
position along a specified family of states.
\end{abstract}
\maketitle


\section{Introduction and statement of main results}\label{Section_0}\label{Invar}

In this paper we study a stochastic system in which a momentum variable
drives the position through a coordinatewise normalization, while a
further variable records the recent squared gradient. Let
$f\colon\R^m\to\R$ be an objective function, write $g=\nabla f$, and
fix $a,b,\gamma,\sigma,\varepsilon>0$. The autonomous system is
\begin{equation}\label{eq:main-system}
\begin{aligned}
 \mathrm d x_t^i&=-\gamma{z_t^i}/{(\sqrt{y_t^i}+\varepsilon)}\,\mathrm d t,\\
 \mathrm d z_t^i&=a\bigl(\partial_i f(x_t)-z_t^i\bigr)\,\mathrm d t
                         +a\sigma\,\mathrm d B_t^i,\\
 \mathrm d y_t^i&=b\bigl((\partial_i f(x_t))^2+\sigma^2-y_t^i\bigr)\,\mathrm d t,
 \qquad i=1,\ldots,m.
\end{aligned}
\end{equation}
Here $B$ is an $m$-dimensional standard Brownian motion. The momentum
variable $z$ relaxes towards the gradient at rate $a$ and is subject
to additive noise. The second-moment variable $y$ relaxes at rate $b$
towards the prescribed quantity $g(x)^{\odot2}+\sigma^2\mathbf1$, where
$g^{\odot2}=(g_1^2,\ldots,g_m^2)$ and $\mathbf1=(1,\ldots,1)$.
Through the denominator in the position equation, $y$ determines how
strongly each component of the momentum influences the motion.
Noise acts directly only in $z$, and randomness is transmitted to
$x$ and $y$ through the coupling.

Our principal questions concern the existence and uniqueness of an
invariant law, convergence towards that law, and the associated
Fokker-Planck equation. The adaptive normalization plays a central
role in this analysis. Large values of $y$ reduce the response of
$x$ to the momentum, while the evolution of $y$ depends in turn on
the squared gradient at the current position. This feedback must be
taken into account when controlling the return from large states.
At the same time, the degeneracy of the noise requires a separate
analysis of mixing in the complete state. The assumptions imposed
below allow flat regions and other nonconvex behavior in bounded
sets, so transmission of noise cannot be inferred from a
nondegeneracy condition holding everywhere.

This combination of momentum and adaptive normalization is inspired
by bias-corrected Adam, which uses exponential averages of stochastic
gradients and their coordinatewise squares. To make the connection
precise, Section~\ref{sec:model} starts from the discrete updates and
examines their scaling. The noise scaling used here to retain a
nonzero momentum diffusion is incompatible with a finite limit for
the unscaled second-moment tracker in the exact recursion. We therefore
construct the model from a modified recursion in which the
squared-gradient update is replaced by a prescribed second-moment
closure.

The modified recursion retains bias correction, whose effect persists
in the continuous-time limit through time-dependent relaxation
coefficients. With
\begin{equation}\label{eq:bias-coefficients}
 a_t={a}/{(1-e^{-at})},\qquad
 b_t={b}/{(1-e^{-bt})},\qquad t>0,
\end{equation}
the resulting time-inhomogeneous system takes the form
\begin{align}\label{eq:entrance-sde-repair}
 \mathrm d x_t^i&=-\gamma{z_t^i}/
                        {(\sqrt{y_t^i}+\varepsilon)}\,\mathrm d t,\notag\\
 \mathrm d z_t&=a_t\big(g(x_t)-z_t\big)\,\mathrm d t
                                  +a_t\sigma\,\mathrm d B_t,\notag\\
 \mathrm d y_t&=b_t\big(-y_t+g(x_t)^{\odot2}
                                    +\sigma^2\mathbf1\big)\,\mathrm d t.
\end{align}
Both coefficients are asymptotic to $1/t$ as $t\downarrow0$, so the
initial time requires a separate entrance construction.
Theorem~\ref{thm:entrance-repair} constructs the entrance solution
with a prescribed initial position and proves convergence of the
modified recursion to this solution on fixed finite time intervals
away from zero.

The same coefficients also connect this construction to the
autonomous system that is the main object of our long-time analysis.
As $t\to\infty$, they converge to $a$ and $b$, respectively, and
\eqref{eq:entrance-sde-repair} approaches \eqref{eq:main-system}
at the level of its coefficients. Under the ergodic assumptions
stated below, we prove that the autonomous system has a unique
invariant law and that the time-inhomogeneous process converges
weakly to this law. Establishing this conclusion requires tightness
and comparison of the two processes, in addition to convergence
of the coefficients.

This long-time comparison concerns the two continuous-time systems.
The approximation of the modified recursion is established on
finite time intervals and gives neither a uniform-in-time
approximation nor an interchange of the limits $h\downarrow0$
and $t\to\infty$, where $h$ is the discretization step.
Accordingly, the invariant law studied here is that of the
continuous-time model. Its identification with a stationary law
of exact Adam, or of the modified recursion at a fixed positive
step size, is not part of the results.

\subsection{Assumptions and main results}\label{Subcond}

For the long-time analysis we impose the following assumptions:
\begin{equation}\tag{A1}\label{ass:smooth}
 f\in C^\infty(\R^m),\qquad
 |\nabla f(x)-\nabla f(\widehat x)|\le L_f|x-\widehat x|
 \quad (x,\widehat x\in\R^m).
\end{equation}
\begin{equation}\tag{A2}\label{ass:dissipative}
 x\cdot\nabla f(x)\ge m_0|x|^2-c_0
 \quad (x\in\R^m),\qquad m_0>0,\quad c_0\ge0.
\end{equation}
The Lipschitz condition in \eqref{ass:smooth} is equivalent to
$\|D^2f(x)\|_{\mathrm{op}}\le L_f$. The smoothness assumption is used
in the hypoellipticity statements but the estimates and control
constructions require less regularity, as indicated in their proofs.
Condition \eqref{ass:dissipative} can equivalently be imposed outside a
fixed ball, since increasing $c_0$ gives its global formulation.
Strong convexity implies this condition but is not required. In
particular, saddle points, flat regions, and multiple minima in bounded
regions are allowed. The growth consequences of these assumptions are
recorded in Section~\ref{subsec:prelim-structural}. The finite-time
entrance and approximation theorem in Section~\ref{sec:model} requires
only a globally Lipschitz gradient.

We use the state spaces
\[
 \overline E=\R^m\times\R^m\times[0,\infty)^m,
 \qquad E=\R^m\times\R^m\times(0,\infty)^m,
\]
and the closed strip
\begin{equation}\label{eq:forward-strip}
 S:=\R^m\times\R^m\times[h_*,\infty)^m,\qquad h_*:={\sigma^2}/{2},
 \qquad t_*:={\log2}/{b}.
\end{equation}
Proposition~\ref{prop:repair-wellposed} below gives a unique global strong solution of \eqref{eq:main-system} from each initial point in $\overline E$. The corresponding Markov semigroup is denoted by $(P_t)_{t\ge0}$. The second-moment equation gives
\[
 y_t^i\ge e^{-bt}y_0^i+\sigma^2(1-e^{-bt}).
\]
Consequently, $E$ and $S$ are forward invariant, and every trajectory from $\overline E$ belongs to $S$ for $t\ge t_*$. The use of a closed strip is convenient because bounded closed sublevel sets of proper functions are then compact in the state space.

Writing $u=(x,z,y)$ for the complete state, we use its Euclidean norm $|u|$.
For an initial probability law $\mu_0$ on $\overline E$, the law at time $t$ is $\mu_0P_t$. For a measure $\mu$ and a nonnegative measurable function $\Omega$, we write
\[
\mu\Omega=\int\Omega\,\mathrm d\mu\quad\mbox{and}\quad \mu_0P_t\Omega=\int_{\overline E}P_t\Omega(u)\,\mu_0(\mathrm d u).
\]
For a signed measure $\nu$ and a weight $\Omega\ge1$, the weighted total variation norm is
\begin{equation}\label{eq:weighted-norm-definition}
 \|\nu\|_\Omega
 :=\sup_{|\varphi|\le\Omega}|\nu(\varphi)|
 =\int\Omega\,\mathrm d|\nu|.
\end{equation}
We normalize the ordinary total variation distance of probabilities by $\|\mu-\nu\|_{\mathrm{TV}}=\sup_A|\mu(A)-\nu(A)|$. The Wasserstein distance $W_2$ is defined using the Euclidean norm of the complete state $(x,z,y)$, as recalled in Section~\ref{Prelim}.

Our first main result concerns the autonomous system. Its quantitative
conclusion uses an exponential Lyapunov weight constructed in
Section~\ref{sec:lya}.

\begin{thm}\label{thm:main-revised}
Assume \eqref{ass:smooth}-\eqref{ass:dissipative} and at least one of the conditions
\begin{equation}\label{eq:parameter-regimes}
 \textnormal{(R1)}\quad b<4a,
 \qquad\textnormal{or}\qquad
 \textnormal{(R2)}\quad
 \gamma L_f\le a\bigl(\sigma/{\sqrt2}+\varepsilon\bigr)/4.
\end{equation}
Then the process \eqref{eq:main-system} has a unique invariant probability measure $\pi_\infty$ on $\overline E$, and
\[
 \pi_\infty\bigl(\R^m\times\R^m\times[\sigma^2,\infty)^m\bigr)=1.
\]
Furthermore, there is a proper continuous weight $\Omega\colon S\to[1,\infty)$ such that
\[
 \pi_\infty\Omega<\infty,
 \qquad 1+|u|^p\le C_p\Omega(u)\quad(p>0,\ u\in S),
\]
and constants $C,\rho>0$ such that, for every probability measure $\mu$ on $S$ with $\mu\Omega<\infty$,
\begin{equation}\label{eq:main-weighted-rate}
 \|\mu P_t-\pi_\infty\|_\Omega\le Ce^{-\rho t}\mu\Omega,
 \qquad t\ge0.
\end{equation}
In particular,
\begin{equation}\label{eq:main-TV-W2-rate}
 \|\mu P_t-\pi_\infty\|_{\mathrm{TV}}+(W_2(\mu P_t,\pi_\infty))^2
 \le Ce^{-\rho t}\mu\Omega.
\end{equation}
For every initial probability measure $\mu_0$ on $\overline E$, without a moment assumption,
\begin{equation}\label{eq:main-qualitative}
 \|\mu_0P_t-\pi_\infty\|_{\mathrm{TV}}\longrightarrow0
 \qquad(t\to\infty).
\end{equation}
If $\mu_0P_{t_*}\Omega<\infty$, the quantitative estimates hold for $t\ge t_*$ with right-hand side
$Ce^{-\rho(t-t_*)}\mu_0P_{t_*}\Omega$.
\end{thm}

\begin{rem} Under \textnormal{(R1)}, the weight
$\Omega=\exp(\alpha\sqrt W)$ is built from the corrected energy
$W$ in \eqref{repair:W}. To describe its growth, write
\[
 K=\sum_i{z_i^2}/{(\sqrt{y_i}+\varepsilon)},
 \qquad Y=\sum_i y_i.
\]
The energy estimates imply that, on the forward invariant strip $S$,
\[
 \sqrt W\asymp 1+|x|+\sqrt K+Y^{1/4},
\]
where the comparison constants depend on the fixed parameters
and the choices made in the construction. Thus the logarithm of
the weight measures the position, the normalized momentum, and
the fourth root of the total tracker size. The momentum contribution
$K$ is discounted when the adaptive denominator is large, while
the tracker correction ensures coercivity in the $y$-variables.
The cross term in $W$ transfers momentum dissipation to the
position. The square root in the exponent reflects the available drift
estimate. The proof gives
\[
 \mathcal LW\le-c\sqrt W+C,
 \qquad \Gamma(W)\le CW,
\]
where $\mathcal L$ is the infinitesimal generator of the autonomous
system \eqref{eq:main-system} and $\Gamma$ is the associated
carr\'e du champ operator. Consequently, $\sqrt W$ has a strictly negative drift outside a
compact set and bounded quadratic-variation density. Exponentiating with
a sufficiently small $\alpha>0$ then gives the geometric drift
inequality for $\Omega$. The restriction $b<4a$ enters through the dissipation of the
normalized kinetic energy. Its coefficient is
$-\gamma(1-b/(4a))$. It places no upper bound on $\gamma$, although
the weight and the convergence constants depend on $\gamma$.
This restriction is sufficient for the construction and is not
claimed to be necessary for recurrence or convergence.
\end{rem}

\begin{rem}
Under \textnormal{(R2)}, the weight in
\eqref{eq:repair-sg-weight} is instead constructed from the
tracking-error energy
\[
 V=1+f(x)-f_*+\lambda|z-g(x)|^2
                  +\eta\sum_i\sqrt{y_i},
 \qquad \Omega=\exp(\theta\sqrt V),
\]
with the positive coefficients specified there. On $S$, the
growth assumptions on $f$ give
\[
 \sqrt V
 \asymp 1+|x|+|z-g(x)|+Y^{1/4}
 \asymp 1+|x|+|z|+Y^{1/4}.
\]
The logarithm of this weight measures the position, the
unnormalized tracking error, and the fourth root of the tracker
size. Large tracker values do not discount the tracking-error
contribution. Condition \textnormal{(R2)} allows momentum relaxation to absorb
the coupling caused by the motion of $g(x)=\nabla f(x)$ in the equation for
$z-g(x)$. Choosing $\eta$ proportional to $b^{-1}$ balances the
tracker contribution to the drift for every $b>0$. Thus $b$ is
unrestricted, although the weight and convergence constants may
depend on it. As under \textnormal{(R1)}, the drift estimate of order $-\sqrt V$
and the quadratic-variation bound explain the square-root
exponential weight. Condition \textnormal{(R2)} is sufficient
and is not claimed to be necessary. Stability outside the two
parameter regimes remains unresolved in this analysis.
\end{rem}

\begin{rem}\label{rem:main-moments-and-rates}
The Wasserstein estimate in \eqref{eq:main-TV-W2-rate} concerns the complete state $(x,z,y)$ and uses the assumption $\mu\Omega<\infty$, or its stated counterpart at time $t_*$. In particular, these assumptions ensure a finite second moment of $y$. The moment-free assertion \eqref{eq:main-qualitative} concerns total variation and gives no Wasserstein conclusion. The distinction between first and second moments of the tracker is discussed in Remark~\ref{rem:no-moment-creation}. The rate $\rho$ in \eqref{eq:main-TV-W2-rate} is the rate for $W_2^2$. The corresponding estimate for $W_2$ has rate $\rho/2$. Moreover, the exponential initial moment in the quantitative bound cannot be replaced in general by a fixed polynomial moment. Proposition~\ref{prop:polynomial-obstruction} proves this for the one-dimensional quadratic objective. The qualitative assertion \eqref{eq:main-qualitative} is therefore stated separately from the quantitative estimate.
\end{rem}

The time-inhomogeneous system has an evolution family rather than a
one-parameter semigroup. Our second main result identifies its limiting
law.

\begin{thm}\label{thm:inhom-main}
Assume the hypotheses of Theorem~\ref{thm:main-revised}. Let the solution of \eqref{eq:entrance-sde-repair} start at time $\delta>0$ from any probability law on $\overline E$. Then
\[
 \operatorname{Law}(x_t,z_t,y_t)\Rightarrow\pi_\infty
 \qquad(t\to\infty),
\]
where $\pi_\infty$ is as in Theorem \ref{thm:main-revised}. The same conclusion holds for the entrance solution of Theorem~\ref{thm:entrance-repair}.
\end{thm}

The conclusion of Theorem~\ref{thm:inhom-main} is weak convergence.
It gives convergence of probabilities for Borel sets whose boundary has
zero $\pi_\infty$-measure, without asserting convergence on every Borel
set. The proof requires tightness and a comparison on fixed time
intervals and convergence of the coefficients alone does not establish
the conclusion.

\subsection{Analytical difficulties and proof strategy}

The proof of Theorem~\ref{thm:main-revised} combines control of large
states with quantitative mixing on compact sets. The first is obtained
through Lyapunov estimates, while the second follows from an explicit
control construction that gives a common minorizing measure for
transition laws starting in each compact set. A weighted Harris theorem
then yields uniqueness of the invariant law and convergence to
equilibrium.

The Lyapunov analysis must account for the interaction between momentum
and the adaptive denominator. In the first construction, the terms
involving $\nabla f$ cancel in the drift of the natural energy
\[
 f-f_*+\frac\gamma{2a}\sum_i
          \frac{z_i^2}{\sqrt{y_i}+\varepsilon},
 \qquad f_*:=\min_{\R^m}f.
\]
The evolution of $y$ creates an additional term whose possible positive
contribution is controlled when $b<4a$. A normalized cross term
transfers dissipation to the position variable, and a further term
controls the tracker. The second construction instead uses the
momentum error $z-\nabla f(x)$. Its equation contains a coupling
through the Hessian of $f$, which can be controlled when $\gamma$
is sufficiently small. In both constructions, the corrected energy
has a drift bounded above, outside a compact set, by a negative
multiple of its square root. Together with a compatible
quadratic-variation estimate, this allows us to exponentiate the
square root and obtain a geometric Lyapunov inequality.

The resulting weight also determines the dependence of the convergence
estimate on the initial state. Coordinatewise normalization can slow
the return from large states. Even for a quadratic objective,
Proposition~\ref{prop:polynomial-obstruction} exhibits a family of
initial states whose positions remain of order $R$ for times
proportional to $R$, with probability tending to one as $R\to\infty$.
Any uniform exponential convergence estimate therefore requires a
prefactor growing at least exponentially in $R$ along this family.
For either Lyapunov weight, under its corresponding parameter
condition, the logarithm of the weight is of order $R$ along the
same family. Qualitative total variation convergence requires no
corresponding moment assumption and follows separately by truncating
the initial law.

To complement these recurrence estimates, we must show that the noise,
which acts directly only in the momentum variable, produces sufficient
mixing in the complete state. The coupling to the second-moment tracker
is described by the map
\begin{equation}\label{eq:main-nondegenerate-map}
 q(x)=\bigl((\partial_i f(x))^2+\sigma^2\bigr)_{i=1}^m,
 \qquad Dq(x)=2\Diag(\nabla f(x))D^2f(x).
\end{equation}
Here $\Diag(v)$ denotes the diagonal matrix with diagonal $v$.
Dissipativity implies that $\nabla f$ is onto. A degree argument
combined with Sard's theorem then gives a point $x_*$ at which
$Dq(x_*)$ is invertible, as shown in Lemma~\ref{lem:repair-degree}.
The resulting local inverse allows us to adjust the tracker by
varying a controlled position path near $x_*$.

This observation is the starting point for the control construction.
The position path moves from its initial value to a holding point
near $x_*$, remains there for a sufficiently long interval, and then
follows a short terminal arc. The holding stage reduces the influence
of the initial tracker, while the terminal arc adjusts the final
position and momentum. To match the final tracker as well, we solve
the complete endpoint equation, including the contribution of the
terminal arc. This produces a differentiable family of controls
indexed by the target state and an onto endpoint derivative.

The regularity of this control family is what permits the passage
from accessibility to quantitative minorization. The support theorem
relates controlled paths to the support of the diffusion, but does
not itself supply the required uniform lower bound. We obtain that
bound by decomposing Brownian motion into finitely many Gaussian
coordinates and an independent residual path, and applying the inverse
function theorem in those coordinates. A finite cover then yields
the small-set property for every compact subset of $E$.

The same local nondegeneracy also enters the analysis of the
Fokker-Planck equation. The matrix
\[
 A(x)=\Diag(\nabla f(x))D^2f(x)
\]
gives a sufficient parabolic H\"ormander condition on the open set
where it is invertible, and hence local regularity there. Failure
of invertibility does not imply failure of the full bracket condition,
since higher brackets may recover missing directions. The ergodic
argument relies on the regular control family and therefore does
not require a smooth transition density on the whole state space.

Once recurrence and mixing have been established for the autonomous
system, they provide the basis for the comparison with the
time-inhomogeneous dynamics in Theorem~\ref{thm:inhom-main}.
Time-dependent Lyapunov functions give tightness, while comparison
with the autonomous process on fixed time windows, together with
convergence to equilibrium uniformly over compact sets, identifies
the limiting law. The singular entrance at zero is treated separately
through normalized integral formulas in Appendix~\ref{aaapp}.

\subsection{Relation to the literature}

Coordinatewise adaptation has a long history in stochastic
optimization. Adagrad \cite{Duchi2011} uses accumulated squared
gradients, while Adam \cite{KingmaBa2015} combines exponentially
weighted first and second moments with bias correction. The examples
and modifications in \cite{Reddi2018} show that convergence of the
discrete algorithm depends on the assumptions and parameter choices.
Our analysis concerns the stochastic model defined above and the precise
relation to the discrete updates is developed in Section~\ref{sec:model}.

Deterministic continuous-time models of Adam and related adaptive
methods are studied in
\cite{BarakatBianchi2021,BelottoGazeau2020,BarakatEtAl2021}.
These works analyze energy dissipation, convergence to critical
points, and the singular coefficients arising from bias correction.
In particular, the normalized momentum energy and the condition
$b\le4a$ appear in \cite{BarakatBianchi2021}, while related energy
criteria and corrections are developed in
\cite{BelottoGazeau2020,BarakatEtAl2021}. The cancellation mechanism
in our first Lyapunov construction has these deterministic
antecedents. Persistent momentum noise requires additional estimates
to obtain recurrence and quantitative convergence to an invariant law.

Stochastic modified equations provide finite-time weak approximations
of stochastic gradient algorithms, see
\cite{li2017stochastic,li2019stochastic}. For adaptive methods,
\cite{malladi2022sdes} develops SDE approximations and scaling rules,
and \cite{CompagnoniEtAl2025} studies adaptive dynamics through
stochastic differential equations with particular attention to the
role of noise. These approaches involve distinct choices of scaling,
state variables, and approximation order. Section~\ref{sec:model}
identifies the modified recursion used here, and
Section~\ref{Conc} discusses alternative SDE models. In particular,
finite-time approximation and identification of invariant laws are
separate questions.

The ergodic framework used in this paper is classical
\cite{Harris1956,MeynTweedie2009,HairerMattingly2011b,HairerMattinglyScheutzow2011}.
Related approaches to degenerate stochastic systems and positivity
of transition probabilities appear in
\cite{MattinglyStuartHigham2002,HerzogMattingly2015}. The contribution
here lies in the coercive corrections and quadratic-variation
estimates that yield exponential Lyapunov bounds in the two parameter
regimes, together with a regular family of controls compatible with
the adaptive tracker. Local hypoellipticity and support arguments
are used with their respective scopes, while quantitative minorization
is proved directly from the endpoint map.

\subsection{Organization of the paper}

Section~\ref{sec:model} describes Adam, explains the scaling and
second-moment closure, and states the entrance and continuous-time
limit theorem for the modified recursion. Section~\ref{Prelim}
collects the finite-time estimates and structural properties used
later. It also treats the Fokker-Planck equation, local
hypoellipticity, the support theorem, and small sets.
Section~\ref{sec:lya} constructs the two Lyapunov functions and proves
existence of an invariant measure. Section~\ref{sec:ergodicity}
applies the weighted Harris argument to obtain uniqueness and
convergence, using the compact-set minorization proved in
Section~\ref{sec:proof-CorC}. Section~\ref{Conc} discusses the scope
of the results and further questions.

Appendix~\ref{aaapp} proves the entrance and approximation theorem.
Appendix~\ref{LLD} establishes convergence of the time-inhomogeneous
system. Appendix~\ref{HMI} contains the moment and integrability
consequences and the obstruction to a polynomial prefactor in a
uniform exponential convergence estimate.

\section{Connection with Adam and construction of the model}
\label{sec:adam}\label{heur}\label{sec:model}

Consider the finite-sum optimization problem
\begin{equation}\label{minima}
 x^\star\in\operatorname*{arg\,min}_{x\in\R^m}
 \left\{f(x):=\frac1n\sum_{i=1}^n f_i(x)\right\}.
\end{equation}
Here $x\in\R^m$ is the parameter vector and $f_i$ is the loss
associated with the $i$-th data point. The objective may be nonconvex and
its finite-sum representation does not imply the analytical
assumptions used in our long-time results. Write $g=\nabla f$.
When evaluating the full gradient is expensive, a minibatch
estimator takes its place. We write
\[
 \nabla\widetilde f_k(x)=g(x)+\xi_k(x),
 \qquad
 \xi_k(x)=\nabla\widetilde f_k(x)-\nabla f(x).
\]
At the current iterate $x_k$, the gradient noise is
$\xi_k=\xi_k(x_k)$. An unbiased estimator satisfies
$\mathbb E[\xi_k\mid\mathcal F_k]=0$, where $\mathcal F_k$ records
the information available before the $k$th stochastic gradient is
sampled. Its covariance may depend on the state, sampling rule,
and minibatch size. The isotropic noise introduced below is a
separate modeling assumption.

The Adam algorithm combines a first-moment tracker with a
coordinatewise second-moment tracker \cite{KingmaBa2015}. Let
$\widetilde z_k,\widetilde y_k$ denote the uncorrected exponential
moving averages. With the standard zero initialization, they satisfy
$\widetilde z_0=\widetilde y_0=0$, and
\begin{align}\label{eq:adam-uncorrected}
\widetilde z_{k+1}&=\alpha\widetilde z_k
                   +(1-\alpha)(g(x_k)+\xi_k),\notag\\
 \widetilde y_{k+1}^i&=\beta\widetilde y_k^i
                   +(1-\beta)|g_i(x_k)+\xi_k^i|^2,
 \qquad i=1,\ldots,m.
\end{align}
Here $\alpha,\beta\in(0,1)$ are the two decay parameters. For $k\ge1$
the bias-corrected trackers are
\begin{equation}\label{eq:adam-bias-correction}
 z_k={\widetilde z_k}/{(1-\alpha^k)},\qquad
 y_k={\widetilde y_k}/{(1-\beta^k)}.
\end{equation}
Starting from a prescribed parameter vector $x_0$, the update is
\begin{equation}\label{eq:adam-parameter-update}
 x_{k+1}^i=x_k^i-\eta{z_{k+1}^i}/
                              {(\sqrt{y_{k+1}^i}+\varepsilon)},
 \qquad i=1,\ldots,m,
\end{equation}
where $\eta>0$ is the learning rate and $\varepsilon>0$ prevents
division by zero. All squares, square roots, and divisions involving
the tracker $y$ are understood coordinatewise.

Eliminating the uncorrected variables gives the equivalent recursion
\begin{align}\label{eqn:sgd_main+}
 x_{k+1}^i&=x_k^i-\eta{z_{k+1}^i}/
                              {(\sqrt{y_{k+1}^i}+\varepsilon)},\notag\\
 z_{k+1}&=\frac{\alpha(1-\alpha^k)}{1-\alpha^{k+1}}z_k
       +\frac{1-\alpha}{1-\alpha^{k+1}}(g(x_k)+\xi_k),\notag\\
 y_{k+1}^i&=\frac{\beta(1-\beta^k)}{1-\beta^{k+1}}y_k^i
       +\frac{1-\beta}{1-\beta^{k+1}}|g_i(x_k)+\xi_k^i|^2.
\end{align}
The formula also applies at $k=0$ if arbitrary finite values are
assigned to $z_0,y_0$, because their coefficients at the first update
are zero. Those assigned values do not represent nonzero initial
values of the uncorrected averages in \eqref{eq:adam-uncorrected}.
In particular, changing them does not change any subsequent iterate. The first tracker in \eqref{eqn:sgd_main+} smooths the update direction, while the second
records recent coordinatewise squared gradients and changes the
relative step sizes. Since a second moment includes the square of
the mean, the denominator measures gradient magnitude rather than
the standard deviation alone. These two averages supply memory and
adaptive preconditioning.

The factors in \eqref{eq:adam-bias-correction} compensate for zero
initialization. To see their precise role, suppose that the gradient
observations have a constant mean $\overline g$. For deterministic
$\widetilde z_0$,
\[
 \mathbb E\widetilde z_k
 =\alpha^k\widetilde z_0+(1-\alpha^k)\overline g,
 \qquad
 \mathbb E\widetilde z_k-\overline g
 =\alpha^k(\widetilde z_0-\overline g).
\]
For $\widetilde z_0=0$, division by $1-\alpha^k$ removes this
initialization bias exactly. If the uncorrected recursion is instead
started from a nonzero $\widetilde z_0$, its normalized version for
$k\ge1$ has mean
\[
 \mathbb E z_k
 =\overline g+\frac{\alpha^k}{1-\alpha^k}\widetilde z_0.
\]
The additional term vanishes as $k\to\infty$, but it is not removed
by the usual zero-initialization correction. Such a recursion must
be initialized through its uncorrected variable and is distinct
from assigning a finite $z_0$ in \eqref{eqn:sgd_main+}. The same calculation applies to the second-moment tracker. If the
coordinatewise second moment of the gradient observations is the
constant vector $\overline s$, then zero initialization gives
\[
 \mathbb E\widetilde y_k=(1-\beta^k)\overline s,
 \qquad \mathbb E y_k=\overline s.
\]
These identities concern constant first and second moments. Along an optimization
trajectory the distribution generally changes, and the corrected
trackers remain weighted averages of past moments. They need not
be unbiased at the current point: bias correction removes the
normalization loss due to initialization, but the averaging lag
remains.

\begin{rem} The default choices $\alpha=0.9$ and $\beta=0.999$ in
\cite{KingmaBa2015} give characteristic memory lengths of order
$(1-\alpha)^{-1}$ and $(1-\beta)^{-1}$ iterations. Values closer
to one retain information longer, while smaller values respond more
quickly. In the scaling below, these iteration counts grow as the
time step decreases, retaining positive memory on the continuous
time scale.
\end{rem}

\subsection{Scaling and the second-moment closure}

The increment form of \eqref{eqn:sgd_main+} is
\begin{align}\label{eqn:sgd_main++}
 x_{k+1}^i-x_k^i&=-\eta{z_{k+1}^i}/
                                  {(\sqrt{y_{k+1}^i}+\varepsilon)},\notag\\
 z_{k+1}-z_k&=\frac{1-\alpha}{1-\alpha^{k+1}}
                         \big(-z_k+g(x_k)+\xi_k\big),\notag\\
 y_{k+1}^i-y_k^i&=\frac{1-\beta}{1-\beta^{k+1}}
                         \big(-y_k^i+|g_i(x_k)+\xi_k^i|^2\big).
\end{align}
The continuous-time description depends jointly on the update size,
memory coefficients, and noise. We discuss these choices before
defining the modified recursion to which the limit theorem applies.

Let $h>0$ and $t_k=kh$. For the iterates corresponding to this time
step, define the piecewise constant interpolations
\begin{equation}\label{eq:adam-interpolations}
 X^h(t)=x_k^h,\qquad Z^h(t)=z_k^h,\qquad Y^h(t)=y_k^h,
 \qquad t\in[t_k,t_{k+1}).
\end{equation}
We also use the continuous, piecewise linear interpolation
$\bar X^h$ of the parameter iterates. The superscript $h$ will
occasionally be suppressed when the step size is fixed.

To retain a finite parameter drift when the normalized update
direction has order one, we set
\[
 \eta=\gamma h,\qquad \gamma>0.
\]
Thus $\gamma$ is the learning-rate coefficient on the time scale
$t=kh$. This specifies the refinement regime studied here without
excluding other time or state rescalings.

If $\alpha,\beta$ were fixed while $h\downarrow0$, the physical
memory lengths $h/(1-\alpha)$ and $h/(1-\beta)$ would tend to zero.
We instead choose
\begin{equation}\label{scalla}
 \alpha_h=1-ah,\qquad \beta_h=1-bh,
 \qquad a,b>0,
\end{equation}
with $h<\min\{a^{-1},b^{-1}\}$. The physical memory lengths then
have order $a^{-1}$ and $b^{-1}$. If $kh\to t>0$, then
\begin{equation}\label{scalla+}
 \alpha_h^k\longrightarrow e^{-at},\qquad
 \beta_h^k\longrightarrow e^{-bt}.
\end{equation}
The limiting bias-correction coefficients are $a_t,b_t$ defined in
\eqref{eq:bias-coefficients}. More precisely, for every $0<\delta<T$,
uniformly over grid points
$t_{k+1}\in[\delta,T]$,
\begin{equation}\label{scalla++}
 \frac{ah}{1-(1-ah)^{k+1}}=h a_{t_{k+1}}+O_{\delta,T}(h^2),
 \qquad
 \frac{bh}{1-(1-bh)^{k+1}}=h b_{t_{k+1}}+O_{\delta,T}(h^2).
\end{equation}
The restriction away from zero is essential for these error bounds.
At the first update each discrete bias-correction factor equals one,
whereas $a_t,b_t$ are asymptotic to $1/t$ near zero. The initial
part of the limiting process must therefore be handled through the
normalized averages themselves, rather than through a uniform
Euler expansion down to time zero.

We next consider the noise scaling. Let $\zeta_k$ be independent
standard Gaussian vectors, independent of the initial state.
The first-moment increment contains, away from
zero, a noise contribution of the form $h a_t\xi_k$. Consequently,
noise of order one and noise of size $h^{-1/2}$ lead to different regimes.
The second-moment equation must be examined at the same time, because
it contains both a term linear in $\xi_k$ and its coordinatewise square.

\smallskip
\noindent\emph{Vanishing or order-one noise.}
If $\xi_k=\sqrt h\,\sigma\zeta_k$, then the random contribution
to the first-moment increment has order $h^{3/2}$. The corresponding
averaged equations are formally
\begin{align}\label{eq:cts-xode}
 \dot x_t^i&=-\gamma{z_t^i}/{(\sqrt{y_t^i}+\varepsilon)},\quad\dot z_t=a_t(g(x_t)-z_t),\quad\dot y_t^i=b_t\big(g_i(x_t)^2-y_t^i\big).
\end{align}
Here the variance of the gradient observations tends to zero.
For noise with variance independent of $h$, such as $\xi_k=\sigma\zeta_k$,
the accumulated first-moment fluctuations still vanish on this
time scale, but the averaged second-moment equation instead contains
$g_i(x_t)^2+\sigma^2$. These are law-of-large-numbers descriptions.
A variance contribution can therefore survive in the averaged
second-moment equation without a limiting first-moment diffusion.
No convergence theorem for these auxiliary regimes is used below.

\smallskip
\noindent\emph{Diffusive first-moment scaling.}
To obtain nonvanishing first-moment fluctuations on the scale $t=kh$,
we use
\begin{equation}\label{klara3}
 \xi_k={\sigma}\zeta_k/\sqrt h.
\end{equation}
The resulting random first-moment increment has order $\sqrt h$,
as required for a Brownian limit. Its formal limiting equation is
\begin{equation}\label{eq:cts-zapa}
 \mathrm dz_t=a_t(g(x_t)-z_t)\,\mathrm dt
                         +a_t\sigma\,\mathrm dB_t.
\end{equation}
However, this scaling does not by itself give a finite limit for the
original, unrescaled second-moment tracker. Indeed, the elementary decomposition
\begin{equation}\label{Klara2}
 |g_i(x_k)+\xi_k^i|^2
 =g_i(x_k)^2+2g_i(x_k)\xi_k^i+(\xi_k^i)^2
\end{equation}
gives, under \eqref{klara3},
\begin{equation}\label{eq:scaled-square-decomposition}
 h|g_i(x_k)+\xi_k^i|^2
 =h g_i(x_k)^2+2\sigma\sqrt h\,g_i(x_k)\zeta_k^i
                              +\sigma^2(\zeta_k^i)^2.
\end{equation}
Conditional on the past,
\[
 \mathbb E\big[|g_i(x_k)+\xi_k^i|^2\mid\mathcal F_k\big]
 =g_i(x_k)^2+{\sigma^2}/{h}.
\]
The quadratic-noise term therefore has order one per update in
\eqref{eq:scaled-square-decomposition}. Its centered part
$\sigma^2((\zeta_k^i)^2-1)$ has variance $2\sigma^4$ per step.
Summing order $h^{-1}$ such independent terms gives variance of order
$h^{-1}$, and cannot be justified as a fluctuation that vanishes
without an additional normalization.

The incompatibility is already explicit for the constant objective
$g\equiv0$. For the exact bias-corrected second-moment recursion,
\[
 y_k^i=\frac{b\sigma^2}{1-\beta_h^k}
                   \sum_{j=0}^{k-1}\beta_h^{k-1-j}(\zeta_j^i)^2.
\]
Thus $\mathbb E y_k^i=\sigma^2/h$. If $kh\to t>0$, then
\[
 \mathbb E(hy_k^i)=\sigma^2,
 \qquad
 \operatorname{Var}(hy_k^i)
 =\frac{2b^2\sigma^4h^2}{(1-\beta_h^k)^2}
                         \sum_{j=0}^{k-1}\beta_h^{2j}
 =O(h).
\]
Consequently $hy_k^i\to\sigma^2$ in $\mathrm L^2$, and the unrescaled tracker
diverges in probability. In particular, the diffusion limit proved
below cannot be a limit of the complete exact Adam recursion under
\eqref{klara3} with its second-moment variable unchanged.

We retain the diffusive first-moment scaling and specify an effective
second-moment closure by replacing the squared-gradient update with
\begin{equation}\label{scale+}
 |g_i(x_k)+\xi_k^i|^2
 \quad\leadsto\quad g_i(x_k)^2+\sigma^2.
\end{equation}
This is a modeling choice, not the conditional expectation of the
squared observation under \eqref{klara3}. It provides a positive
variance floor and tracks the squared deterministic gradient while
retaining first-moment noise, coordinatewise preconditioning,
finite memory, and bias correction. The closure is part of the
definition of the discrete model whose limit is proved below.

\subsection{The modified recursion}

Collecting the parameter choices, we use
\begin{equation}\label{scale}
 \eta=\gamma h,\qquad \alpha_h=1-ah,\qquad
 \beta_h=1-bh,\qquad \xi_k=\sigma h^{-1/2}\zeta_k,
\end{equation}
together with \eqref{scale+}. Put
\[
 A_k^h=1-\alpha_h^k,\qquad B_k^h=1-\beta_h^k.
\]
It is convenient to define the modified scheme through its
uncorrected trackers $m_k^h,q_k^h$. For a prescribed $x_0^h=x_0$,
\begin{align}\label{eq:closed-repair}
 m_0^h&=q_0^h=0,\notag\\
 m_{k+1}^h&=\alpha_hm_k^h+ah\,g(x_k^h)
                                +a\sigma\sqrt h\,\zeta_k,\quad q_{k+1}^h=\beta_hq_k^h
                 +bh\big(g(x_k^h)^{\odot2}+\sigma^2\mathbf1\big),\notag\\
 z_{k+1}^h&=\frac{m_{k+1}^h}{A_{k+1}^h},\quad y_{k+1}^h=\frac{q_{k+1}^h}{B_{k+1}^h},\quad x_{k+1}^{h,i}=x_k^{h,i}
       -\gamma h{z_{k+1}^{h,i}}/
                         {(\sqrt{y_{k+1}^{h,i}}+\varepsilon)}.
\end{align}
We use $g^{\odot2}=(g_1^2,\ldots,g_m^2)$ and
$\mathbf1=(1,\ldots,1)$. These equations are equivalent to
\eqref{eqn:sgd_main+} after the replacement \eqref{scale+}.
The explicit tracker formulas are
\begin{align}\label{eq:closed-explicit-averages}
 z_k^h&=\frac{ah}{A_k^h}\sum_{j=0}^{k-1}
              \alpha_h^{k-1-j}g(x_j^h)
       +\frac{a\sigma\sqrt h}{A_k^h}
                         \sum_{j=0}^{k-1}\alpha_h^{k-1-j}\zeta_j,
 \notag\\
 y_k^h&=\frac{bh}{B_k^h}\sum_{j=0}^{k-1}
        \beta_h^{k-1-j}\big(g(x_j^h)^{\odot2}+\sigma^2\mathbf1\big),
 \qquad k\ge1.
\end{align}
The deterministic averaging weights sum to one. In particular,
\begin{equation}\label{eq:closed-positive-repair}
 y_k^{h,i}\ge\sigma^2,\qquad k\ge1,
\end{equation}
and the modified recursion is well defined for every realization of
the Gaussian sequence. Any assigned finite values $z_0^h,y_0^h$
are again irrelevant after the first update.

\begin{rem}
The Gaussian scaling \eqref{klara3} is used to define a diffusion
refinement of the modified recursion. It is not asserted to describe
the dependence of a fixed minibatch estimator on an artificial time
step. In particular, the parameter $\sigma$ is an effective noise
parameter in this model. State-dependent covariance, correlated
gradient observations, or different closures would lead to different
limiting equations and require separate analysis.
\end{rem}

\subsection{Continuous-time limit and entrance solution}

The limit of the modified recursion includes a distinguished entrance
at the singular initial time. The proof, given in
Appendix~\ref{aaapp}, requires only global Lipschitz continuity of
$g$. Radial dissipativity and the stronger smoothness assumptions
of the ergodic analysis are not needed here.

Let $A(t)=1-e^{-at}$ and $B(t)=1-e^{-bt}$. For a continuous path $v\colon[0,T]\to\mathbb R^m$, define
\begin{align}\label{eq:averaging-repair}
 D[v](t)&=\frac{a}{A(t)}\int_0^t
                        e^{-a(t-r)}g(v(r))\,\mathrm dr,\notag\\
 S[v](t)&=\frac{b}{B(t)}\int_0^t e^{-b(t-r)}
                   \big(g(v(r))^{\odot2}+\sigma^2\mathbf1\big)\,\mathrm dr,
 \qquad t>0.
\end{align}
The continuous extensions at zero are
$D[v](0)=g(v(0))$ and
$S[v](0)=g(v(0))^{\odot2}+\sigma^2\mathbf1$.
For the following statement, $\mathbb D([\delta,T];\mathbb R^{3m})$
denotes the space of right-continuous paths with left limits,
equipped with the Skorokhod $J_1$ topology.

\begin{thm}\label{thm:entrance-repair}
Let $g=\nabla f$ be globally Lipschitz and let $x_0\in\mathbb R^m$
be deterministic. Fix $a,b,\gamma,\sigma,\varepsilon>0$.
On a space carrying an $m$-dimensional Brownian motion there is a
unique adapted continuous process $x$ on $[0,\infty)$ satisfying
\begin{equation}\label{eq:entrance-volterra-repair}
 x_t^i=x_0^i-\gamma\int_0^t
       \frac{D_i[x](r)+N_r^i/A(r)}
            {\sqrt{S_i[x](r)}+\varepsilon}\,\mathrm dr,
 \qquad
 N_t=a\sigma\int_0^t e^{-a(t-r)}\,\mathrm dB_r.
\end{equation}
Define $z_t=D[x](t)+N_t/A(t)$ and $y_t=S[x](t)$ for $t>0$.
Then $(x,z,y)$ solves the system \eqref{eq:entrance-sde-repair}.
Moreover,
\[
 y_t^i\ge\sigma^2,\qquad
 x_t\longrightarrow x_0,\qquad
 y_t\longrightarrow g(x_0)^{\odot2}+\sigma^2\mathbf1
 \quad(t\downarrow0),
\]
and $z\in \mathrm L^1((0,T);\mathbb R^m)$ almost surely for every $T<\infty$.
Uniqueness is within the entrance class specified by
\eqref{eq:entrance-volterra-repair}. Equivalently, this class is
specified by the SDE and the conditions
$x_t\to x_0$, $A(t)z_t\to0$, and $B(t)y_t\to0$ as $t\downarrow0$.
For every $\delta>0$, its law at that time is therefore the uniquely
determined entrance law
\[
 \mu_\delta=\operatorname{Law}\bigl(
 x_\delta,\ D[x](\delta)+{N_\delta}/{A(\delta)},\ S[x](\delta)
 \bigr).
\]
It is determined by $x_0$ and the dynamics on $[0,\delta]$. Furthermore, if we couple the modified recursion \eqref{eq:closed-repair} to the same
Brownian motion by $\sqrt h\,\zeta_k=B_{(k+1)h}-B_{kh}$,  then, with the interpolations defined above, for every $0<\delta<T$,
\begin{align}\label{eq:entrance-convergence-repair}
 \sup_{0\le t\le T}|\bar X^h(t)-x_t|&\longrightarrow0,
 \notag\\
 \sup_{\delta\le t\le T}
 \big(|X^h(t)-x_t|+|Z^h(t)-z_t|+|Y^h(t)-y_t|\big)
 &\longrightarrow0
\end{align}
in probability. In particular, the interpolated modified recursion
converges in law in $\mathbb D([\delta,T];\mathbb R^{3m})$ to this
entrance solution.
\end{thm}

\begin{rem}
The initial parameter $x_0$ survives in the limit, but arbitrary
assigned corrected tracker values $z_0^h,y_0^h$ do not. The entrance
values of the uncorrected quantities are zero, and the entrance
value of the corrected second-moment tracker is determined by
$g(x_0)^{\odot2}+\sigma^2\mathbf1$. This distinction is essential
because ordinary finite initial values cannot be prescribed freely
for all three corrected variables at time zero.
\end{rem}

\begin{rem}
The coefficients $a_t,b_t$ are asymptotic to $1/t$ as $t\downarrow0$.
The stochastic first-moment singularity is integrable in time but
is not removable by assigning a finite $z_0$. More precisely,
\[
 \sqrt t\,z_t\ \Rightarrow\ N(0,\sigma^2I_m)
 \qquad(t\downarrow0).
\]
Thus $z_t$ cannot converge in probability to a finite random vector
at zero. Nevertheless, its contribution to the parameter equation
is integrable, and $x$ has the continuous initial value $x_0$.
The processes on different intervals $[\delta,T]$ are restrictions
of the single entrance process constructed in the theorem.
\end{rem}

\begin{rem}
If $x_0$ is random and independent of the Brownian motion, the same
construction and convergence statement hold. One first applies the
estimates on $\{|x_0|\le R\}$ and then lets $R\to\infty$.
No moment assumption on $x_0$ is required for convergence in
probability. The moment bounds in the proof retain their stated
moment assumptions when the initial state is random.
\end{rem}

\begin{rem}
The tracker $y$ has finite variation, but it is generally random
because it depends on the random path $x$. Positivity follows from
its explicit average formula, not from an assertion that its
trajectory is deterministic. In the large-time regime,
$a_t\to a$ and $b_t\to b$, giving the autonomous system
\eqref{eq:main-system}, whose long-time behavior is described by
Theorem~\ref{thm:main-revised}. The transfer of its equilibrium law to
the bias-corrected process requires tightness and mixing, and is
proved in Appendix~\ref{LLD}.
\end{rem}

\section{Preliminaries}\label{Prelim}

In this section we collect the notation and probabilistic facts used in
the analysis of the homogeneous system \eqref{eq:main-system}.  We first
establish well-posedness and continuous dependence on the initial state
and on the additive driving path.  This also gives a direct proof of the
support theorem in the present setting.  We then discuss the generator,
the Fokker-Planck equation, and the local nondegeneracy supplied by the
structural assumptions on $f$.

Throughout, $|\cdot|$ denotes the Euclidean norm, and
$\|\cdot\|_p$ denotes the $\ell^p$-norm when a distinction is needed.
We write $B(x,r)$ for an open Euclidean ball.  For a continuous path $w$
on $[0,T]$, its uniform norm is
$\|w\|_{\infty,T}=\sup_{0\le t\le T}|w(t)|$.
The notation $C_b(X)$ refers to bounded continuous functions on a
topological space $X$, and $C_0^\infty(O)$ refers to smooth functions
with compact support in an open set $O$.

\subsection{State spaces and distances between laws}

We use the state spaces $\overline E$, $E$, and $S$ introduced in
Section~\ref{Subcond}.
Thus $\overline E$ allows zero components of the second-moment variable,
$E$ is its interior, and $S$ is the closed strip with
$y_i\ge h_*:=\sigma^2/2$.  In particular, $\overline E$ and $S$ are
closed subsets of $\mathbb R^{3m}$ and are Polish spaces with their
Euclidean metrics.  The distinction between open and closed strips
will be relevant when applying compactness arguments.  An intersection
of a bounded closed sublevel set with $S$ is compact, whereas the same
assertion need not hold with $y_i>h_*$ in place of $y_i\ge h_*$.

For a probability measure $\mu$ and a Markov kernel $P_t$, the notation
$\mu P_t$ denotes the probability measure
\[
 (\mu P_t)(A)=\int P_t(u,A)\,\mu(\mathrm d u).
\]
For a measurable function $\varphi$ for which the integral is defined,
we also write $\mu\varphi=\int\varphi\,\mathrm d\mu$ and
$P_t\varphi(u)=\int\varphi(v)P_t(u,\mathrm d v)$.
The total variation distance between two probability measures on the
same measurable space is
\[
 \|\mu-\nu\|_{\mathrm{TV}}
 =\sup_A|\mu(A)-\nu(A)|
 =\frac12\sup_{|\varphi|\le1}
       \left|\int\varphi\,\mathrm d\mu-\int\varphi\,\mathrm d\nu\right|.
\]
With this normalization mutually singular probability measures have
distance one.  The weighted total variation norm used in the main
theorem has no factor $1/2$.  If $\Omega\ge1$, then
\[
 2\|\mu-\nu\|_{\mathrm{TV}}
 \le\|\mu-\nu\|_\Omega
 =\int\Omega\,\mathrm d|\mu-\nu|.
\]

If $\mu$ and $\nu$ are probability measures on $\mathbb R^{3m}$ with
finite second moments, their $2$-Wasserstein distance is
\begin{equation}\label{eq:prelim-Wasserstein}
 W_2(\mu,\nu)
 =\inf_{\kappa\in\Pi(\mu,\nu)}
   \left(\iint|u-v|^2\,\kappa(\mathrm d u,\mathrm d v)\right)^{1/2}.
\end{equation}
Here $\Pi(\mu,\nu)$ is the set of couplings, namely probability
measures on $\mathbb R^{3m}\times\mathbb R^{3m}$ whose first and second
marginals are $\mu$ and $\nu$, respectively.  This distance concerns
the entire state $(x,z,y)$.  In particular, a first moment of $y$ does
not suffice to make \eqref{eq:prelim-Wasserstein} finite.  The moment
assumptions in the quantitative convergence statements will therefore
be expressed using the Lyapunov weight rather than inferred from a
quadratic moment of $x$ and $z$ alone.

\subsection{Consequences of the structural assumptions}
\label{subsec:prelim-structural}

The assumptions on $f$ are those stated in
\eqref{ass:smooth}-\eqref{ass:dissipative}.  We write $g=\nabla f$ and
$H_f=D^2f$.  Global Lipschitz continuity of $g$ implies
\begin{equation}\label{eq:prelim-linear-gradient}
 |g(x)|\le |g(0)|+L_f|x|,
 \qquad \|H_f(x)\|_{\mathrm{op}}\le L_f.
\end{equation}
Thus the curvature is uniformly bounded, although the objective may
be nonconvex.  In particular,
\[
 f(x)\le f(0)+g(0)\cdot x+\frac{L_f}{2}|x|^2.
\]
The dissipativity assumption gives a corresponding quadratic lower
bound, with possibly different constants.  To see this directly, fix
$R\ge1$, write $x=r\theta$ with $|\theta|=1$, and integrate along the
ray.  For $r\ge R$,
\begin{align*}
 f(r\theta)-f(R\theta)
 &=\int_R^r\frac{g(s\theta)\cdot s\theta}{s}\,\mathrm d s\ge\frac{m_0}{2}(r^2-R^2)-c_0\log(r/R).
\end{align*}
The minimum of $f$ on the sphere of radius $R$ is finite, and the
logarithmic term can be absorbed into a smaller positive multiple of
$r^2$.  It follows that $f$ is coercive, attains its minimum, and satisfies
the two-sided bounds
\begin{equation}\label{eq:initial-growth}
 c|x|^2-C\le f(x)-f_*\le C(1+|x|^2),
 \qquad f_*:=\min_{\R^m}f.
\end{equation}
The same dissipativity inequality and Cauchy-Schwarz also give
\begin{equation}\label{eq:prelim-gradient-coercivity}
 |g(x)|^2\ge c|x|^2-C
\end{equation}
for suitable constants $c>0$ and $C<\infty$.

The stated hypotheses are weaker than global strong convexity.  For
comparison, if $H_f\succeq m_f I$ for some $m_f>0$, then
\[
 f(x)\ge f(0)+g(0)\cdot x+\frac{m_f}{2}|x|^2,
\]
and hence $f(x)\ge(m_f/4)|x|^2-C$ after absorbing the linear term.
Strong convexity also implies radial dissipativity after changing the
constants.  It is not assumed in this paper.  Flat regions, saddle
points, and several local minima in a bounded region remain possible.
An eventual radial dissipativity inequality is equivalent to its global
version in \eqref{ass:dissipative}, since the constant can be enlarged
on the remaining compact ball.

\subsection{The additive-noise solution map}
\label{subsec:prelim-solution-map}

The diffusion coefficient in \eqref{eq:main-system} is constant.  It is
therefore useful to construct the dynamics first for a deterministic
continuous driving path.  Let
\[
 C_0([0,T];\mathbb R^m)
 =\{w\in C([0,T];\mathbb R^m)\colon w(0)=0\}.
\]
For $u=(x_0,z_0,y_0)\in\overline E$ and
$w\in C_0([0,T];\mathbb R^m)$, consider the integral equations
\begin{equation}\label{eq:prelim-driven-system}
 \begin{aligned}
 x_i(t)&=x_{0,i}-\gamma\int_0^t
            \frac{z_i(s)}{\sqrt{y_i(s)}+\varepsilon}\,\mathrm d s,\\
 z_i(t)&=z_{0,i}+a\int_0^t(g_i(x(s))-z_i(s))\,\mathrm d s
            +a\sigma w_i(t),\\
 y_i(t)&=e^{-bt}y_{0,i}+\sigma^2(1-e^{-bt})
       +b\int_0^t e^{-b(t-s)}g_i(x(s))^2\,\mathrm d s.
 \end{aligned}
\end{equation}
The last identity is the variation-of-constants formula for the
$y$-equation.  In particular, every solution satisfies
\begin{equation}\label{eq:prelim-path-lower-bound}
 y_i(t)\ge e^{-bt}y_{0,i}+\sigma^2(1-e^{-bt}).
\end{equation}

\begin{lem}\label{lem:prelim-solution-map}
Let $g$ be globally Lipschitz.  For every $T<\infty$, initial point
$u\in\overline E$, and path $w\in C_0([0,T];\mathbb R^m)$,
\eqref{eq:prelim-driven-system} has a unique continuous solution on
$[0,T]$.  Write this solution as
\[
 \Phi_T(u,w)\in C([0,T];\overline E),
 \qquad F_T(u,w)=\Phi_T(u,w)(T).
\]
The map $\Phi_T$ is continuous for the Euclidean topology on the
initial state and the uniform topologies on the path spaces.
More precisely, on every bounded set of initial points and driving
paths it satisfies a Lipschitz estimate, with a constant depending
on that set and on $T$.
\end{lem}

\begin{proof}
Initially assume that $y_{0,i}>0$ for every $i$.  Set
$v=z-a\sigma w$.  Then $(x,v,y)$ solves a system of ordinary
differential equations with coefficients continuous in time and
locally Lipschitz in the state as long as $y_i>0$.  The lower bound
\eqref{eq:prelim-path-lower-bound} prevents a component from reaching
zero, so a unique solution exists up to its possible explosion time.

We next derive bounds that rule out explosion.  Since
$\sqrt{y_i}+\varepsilon\ge\varepsilon$, the first two equations and
\eqref{eq:prelim-linear-gradient} give, on any interval of existence,
\[
 \sup_{0\le s\le t}(|x(s)|+|z(s)|)
 \le C\bigl(1+|x_0|+|z_0|+\|w\|_{\infty,T}\bigr)
       +C\int_0^t\sup_{0\le r\le s}(|x(r)|+|z(r)|)\,\mathrm d s.
\]
Gr\"onwall's inequality therefore yields
\begin{equation}\label{eq:prelim-driven-xz-bound}
 \sup_{0\le t\le T}(|x(t)|+|z(t)|)
 \le C_T\bigl(1+|x_0|+|z_0|+\|w\|_{\infty,T}\bigr).
\end{equation}
The explicit formula for $y$ then gives
\begin{equation}\label{eq:prelim-driven-y-bound}
 \sup_{0\le t\le T}|y(t)|
 \le |y_0|+C_T\biggl(1+\sup_{0\le t\le T}|x(t)|^2\biggr).
\end{equation}
All three variables remain bounded on every finite interval.

For continuity, compare solutions with initial states $u,\widehat u$
and paths $w,\widehat w$.  On bounded sets of these data,
\eqref{eq:prelim-driven-xz-bound}-\eqref{eq:prelim-driven-y-bound}
place the solutions in a common bounded region.  For $0<t\le T$,
\eqref{eq:prelim-path-lower-bound} gives
\begin{align*}
 \left|\frac1{\sqrt{y_i(t)}+\varepsilon}
       -\frac1{\sqrt{\widehat y_i(t)}+\varepsilon}\right|
 &\le\frac{|y_i(t)-\widehat y_i(t)|}
 {\varepsilon^2(\sqrt{y_i(t)}+\sqrt{\widehat y_i(t)})}\le C_Tt^{-1/2}|y_i(t)-\widehat y_i(t)|.
\end{align*}
The coefficient $t^{-1/2}$ is integrable at zero.  The remaining
coefficients are Lipschitz on the common bounded region.  If
\[
 D(t)=\sup_{0\le s\le t}
       |(x,z,y)(s)-(\widehat x,\widehat z,\widehat y)(s)|,
\]
subtraction of the integral equations gives
\[
 D(t)\le C\bigl(|u-\widehat u|+
                \|w-\widehat w\|_{\infty,T}\bigr)
       +C\int_0^t(1+s^{-1/2})D(s)\,\mathrm d s.
\]
Another application of Gr\"onwall's inequality proves the asserted
Lipschitz estimate.  Equivalently, one may first localize to a bounded
state region and then use the preceding a priori bounds to remove
the localization.  No global Lipschitz property of the full drift is
being assumed.

Finally, allow $y_{0,i}=0$.  Approximate $y_0$ by
$y_0+n^{-1}(1,\ldots,1)$.  The estimates just obtained are uniform in
$n$, including the integrable bound near time zero.  The corresponding
solutions form a Cauchy sequence in the uniform norm.  Its limit
satisfies \eqref{eq:prelim-driven-system}, because its integrands
converge uniformly and the denominators remain at least $\varepsilon$.
The same difference estimate proves uniqueness and continuous
dependence for these boundary initial points.
\end{proof}

\begin{prop}\label{prop:repair-wellposed}
Let $g$ be globally Lipschitz.  The homogeneous system
\eqref{eq:main-system} has a unique global strong solution with
continuous paths for every initial point of $\overline E$.  Its
transition kernels form a time-homogeneous Markov semigroup $(P_t)$,
which is weak Feller on $\overline E$.
For every $q\ge1$ and $T<\infty$,
\begin{equation}\label{eq:repair-finite-time-moments}
 \mathbb E\sup_{0\le s\le T}
 \bigl(|x_s|^{2q}+|z_s|^{2q}+|y_s|^q\bigr)
 \le C_{T,q}\left(1+
 \mathbb E\bigl[|x_0|^{2q}+|z_0|^{2q}+|y_0|^q\bigr]\right)
\end{equation}
whenever the right-hand side is finite.  Moreover,
\begin{equation}\label{eq:repair-y-formula}
 y_i(t)=e^{-bt}y_i(0)+\sigma^2(1-e^{-bt})
       +b\int_0^t e^{-b(t-s)}g_i(x_s)^2\,\mathrm d s.
\end{equation}
The strip $S$ is forward invariant, and every trajectory starting in
$\overline E$ belongs to $S$ for $t\ge t_*=(\log2)/b$.
\end{prop}

\begin{proof}
Apply Lemma~\ref{lem:prelim-solution-map} with $w=B$, where $B$ is
the Brownian motion in \eqref{eq:main-system}.  The construction is
nonanticipating, since the solution up to time $t$ is determined by
the initial state and the path up to time $t$.  It therefore produces
an adapted strong solution.  The integral formulation agrees with the
It\^o system because the diffusion coefficient is constant.  Pathwise
uniqueness follows from the deterministic uniqueness statement.
The Markov property and the semigroup identity follow from uniqueness
and the independent stationary increments of Brownian motion. For initial points $u_n\to u$, use the same Brownian motion for all
solutions.  Continuity of $\Phi_T$ gives almost sure convergence in
the uniform path norm on $[0,T]$.  Thus, for every
$\varphi\in C_b(\overline E)$,
\[
 P_t\varphi(u_n)\longrightarrow P_t\varphi(u)
\]
by dominated convergence.  This is the weak Feller property used
below.  No assertion about global smoothing of bounded measurable
functions is needed here. The Brownian uniform norm on a finite interval has moments of every
order.  Raising \eqref{eq:prelim-driven-xz-bound} to the power $2q$
and taking expectations bounds the $x$- and $z$-terms in
\eqref{eq:repair-finite-time-moments}.  Raising
\eqref{eq:prelim-driven-y-bound} to the power $q$ then bounds the
$y$-term.  These estimates also hold for an initial random variable
with the stated moments.  Formula \eqref{eq:repair-y-formula} is
already part of the integral construction.  If $y_i(0)\ge h_*$, its
first two terms are at least
$e^{-bt}h_*+(1-e^{-bt})\sigma^2\ge h_*$.
For arbitrary $y_i(0)\ge0$, they are at least $h_*$ when
$t\ge(\log2)/b$.  This proves the remaining assertions.
\end{proof}

\begin{rem}
Estimate \eqref{eq:repair-finite-time-moments} is a finite-time
estimate.  Its constant is allowed to grow with $T$.  It does not
give recurrence, an invariant probability measure, or a uniform
moment bound as $T\to\infty$.  Those conclusions require the
Lyapunov estimates proved later.
\end{rem}

\begin{lem}\label{lem:repair-stationary-support}
Every invariant probability measure $\pi_\infty$ of $(P_t)$ on
$\overline E$ satisfies
\[
 \pi_\infty\bigl(\mathbb R^m\times\mathbb R^m
                    \times[\sigma^2,\infty)^m\bigr)=1.
\]
This conclusion does not require a moment assumption.
\end{lem}

\begin{proof}
Fix $c<\sigma^2$ and choose $t>0$ such that
$\sigma^2(1-e^{-bt})>c$.  Formula \eqref{eq:repair-y-formula} implies
$P_t(u,\{y_i\le c\})=0$ for every $u\in\overline E$.
Invariance consequently gives
$\pi_\infty(\{y_i\le c\})=0$.
Let $c$ increase to $\sigma^2$ through a countable sequence and
repeat the argument for each coordinate.
\end{proof}

\subsection{A nondegenerate point and the gradient map}

The matrix
\begin{equation}\label{eq:prelim-A}
 A(x)=\operatorname{Diag}(g(x))H_f(x)=\operatorname{Diag}(\nabla f(x))H_f(x)
\end{equation}
will occur both in the bracket computation and in the construction
of control paths.  We set
\begin{equation}\label{dagga}
 \mathcal D_A^\dagger=\{x\in\mathbb R^m\colon\det A(x)=0\}.
\end{equation}
The following argument gives more than the existence of a point
outside this set.  It shows that the gradient map is onto.

\begin{lem}\label{lem:repair-degree}
Under \eqref{ass:smooth}-\eqref{ass:dissipative},
$\nabla f(\mathbb R^m)=\mathbb R^m$ and
$\mathcal D_A^\dagger\ne\mathbb R^m$.  In particular, there are
$x_*\in\mathbb R^m$ and $r>0$ such that
\begin{equation}\label{eq:prelim-good-ball}
 \det A(x)\ne0\qquad(x\in\overline{B(x_*,2r)}).
\end{equation}
\end{lem}

\begin{proof}
Let $F=\nabla f$ and fix $v\in\mathbb R^m$.  Dissipativity gives
\[
 \langle F(x)-v,x\rangle
 \ge m_0|x|^2-c_0-|v||x|>0
 \qquad (|x|=R)
\]
when $R$ is sufficiently large.  The homotopy
\[
 H_s(x)=(1-s)x+s(F(x)-v),\qquad 0\le s\le1,
\]
has a strictly positive scalar product with $x$ on $\partial B(0,R)$.
It therefore has no zero on that boundary.  Homotopy invariance of
the Brouwer degree yields
\[
 \deg(F-v,B(0,R),0)=\deg(\mathrm{Id},B(0,R),0)=1.
\]
The existence property of the degree gives $F(x)=v$ for at least
one $x\in B(0,R)$.  Since $v$ was arbitrary, $F$ is onto. The map $F$ is $C^1$.  For maps between spaces of the same dimension,
this regularity is sufficient for Sard's theorem, which says that
the set of critical values has Lebesgue measure zero.  Choose a
regular value $v$ with $v_i\ne0$ for every $i$.  Such values exist,
since the union of the coordinate hyperplanes also has measure zero.
Surjectivity gives an $x_*$ with $F(x_*)=v$.  Regularity of the value
implies $\det DF(x_*)\ne0$, and hence
\[
 \det A(x_*)=
 \left(\prod_{i=1}^m v_i\right)\det H_f(x_*)\ne0.
\]
Continuity of $A$ gives \eqref{eq:prelim-good-ball} after reducing $r$.
\end{proof}

\begin{rem}
The degree argument needs only $f\in C^2$ and radial dissipativity.
Global Lipschitz continuity of the gradient is used elsewhere to
control the dynamics, but is not needed for the surjectivity argument.
No assumption about the number of critical points or local minima
has entered the proof.
\end{rem}

\subsection{The generator and the Fokker-Planck equation}\label{FP}
\label{subsec:repaired-fokker-planck}

The infinitesimal generator gives a convenient weak formulation of
the evolution of the law.  Put $d_i(y)=\sqrt{y_i}+\varepsilon$.
For $\varphi\in C_0^\infty(E)$ the generator is
\begin{equation}\label{eq:repaired-generator}
 \begin{aligned}
 \mathcal L\varphi
 &=-\gamma\sum_{i=1}^m\frac{z_i}{d_i(y)}\partial_{x_i}\varphi
   +a\sum_{i=1}^m(g_i(x)-z_i)\partial_{z_i}\varphi+b\sum_{i=1}^m(g_i(x)^2+\sigma^2-y_i)\partial_{y_i}\varphi
   +\frac{a^2\sigma^2}{2}\Delta_z\varphi.
 \end{aligned}
\end{equation}
The diffusion acts only in the momentum variables.  Transport and
the dependence of $g$ on $x$ are responsible for transmitting its
effects to the other coordinates.

Let $\mu_t=\mu_0P_t$ be the law of $U_t=(x_t,z_t,y_t)$.
It\^o's formula gives
\begin{equation}\label{eq:repaired-forward-weak}
 \int_E\varphi\,\mathrm d\mu_t-\int_E\varphi\,\mathrm d\mu_s
 =\int_s^t\int_E\mathcal L\varphi\,\mathrm d\mu_r\,\mathrm d r,
 \qquad 0\le s\le t.
\end{equation}
If $s=0$ and the initial law charges $\partial E$, the test function
is extended by zero to that boundary.  The extension is smooth in a
neighborhood of $\overline E$, since the support of $\varphi$ is
compactly contained in $E$.  Both $\mathcal L\varphi$ and the
martingale integrand $a\sigma\nabla_z\varphi$ are bounded.  Thus
\eqref{eq:repaired-forward-weak} requires no moment assumption on
$\mu_0$.

The formal adjoint acts on smooth densities by
\begin{equation}\label{eq:repaired-adjoint}
 \begin{aligned}
 \mathcal L^*p
 &=\gamma\sum_{i=1}^m\partial_{x_i}
       \left(\frac{z_i}{d_i(y)}p\right)
   -a\sum_{i=1}^m\partial_{z_i}\bigl((g_i(x)-z_i)p\bigr)\\
 &\quad-b\sum_{i=1}^m\partial_{y_i}
       \bigl((g_i(x)^2+\sigma^2-y_i)p\bigr)
   +\frac{a^2\sigma^2}{2}\Delta_zp.
 \end{aligned}
\end{equation}
For compactly supported smooth functions $\varphi,\psi$ in $E$,
integration by parts gives
\[
 \int_E(\mathcal L\varphi)\psi\,\mathrm d u
 =\int_E\varphi(\mathcal L^*\psi)\,\mathrm d u,
 \qquad \mathrm d u=\mathrm d x\,\mathrm d z\,\mathrm d y.
\]
More generally, \eqref{eq:repaired-forward-weak} means that
\begin{equation}\label{eq:repaired-forward-distribution}
 \partial_t\mu_t=\mathcal L^*\mu_t
 \quad\hbox{in }\mathcal D'((0,\infty)\times E).
\end{equation}
This is a measure-valued Fokker-Planck equation.  If a density $p$
is available, it takes the familiar form
$\partial_tp=\mathcal L^*p$.  We do not assume a global density in
order to formulate or study \eqref{eq:repaired-forward-distribution}.

For an invariant probability measure $\pi_\infty$, invariance and
\eqref{eq:repaired-forward-weak} imply
\begin{equation}\label{eq:repaired-stationary-weak}
 \int_E\mathcal L\varphi\,\mathrm d\pi_\infty=0,
 \qquad \varphi\in C_0^\infty(E).
\end{equation}
Thus $\mathcal L^*\pi_\infty=0$ in the sense of distributions.
This is a necessary consequence of invariance.  No converse assertion
for arbitrary distributional solutions is needed in our argument.

\subsection{Local hypoellipticity}\label{Hypoellipticity}

The coefficients of the generator are smooth on $E$ under
\eqref{ass:smooth}.  Write
\[
 \mathcal L=\frac12\sum_{i=1}^mX_i^2+X_0,
 \qquad X_i=a\sigma\partial_{z_i},
\]
where
\begin{equation}\label{eq:prelim-drift-field}
 \begin{aligned}
 X_0&=\sum_{j=1}^m\left[
       -\frac{\gamma z_j}{d_j(y)}\partial_{x_j}
       +a(g_j(x)-z_j)\partial_{z_j}\right]+b\sum_{j=1}^m(g_j(x)^2+\sigma^2-y_j)\partial_{y_j}.
 \end{aligned}
\end{equation}
The divergence of $X_0$ is $-m(a+b)$, so
\[
 \mathcal L^*=\frac12\sum_{i=1}^mX_i^2-X_0+m(a+b).
\]
For evolution equations the relevant bracket condition is the
space-time condition for $X_1,\ldots,X_m,X_0-\partial_t$.
This distinction ensures that a deterministic drift direction is not
mistaken for a direction of stochastic smoothing at a fixed time.

Define the open region
\begin{equation}\label{eq:prelim-good-region}
 G=\{(x,z,y)\in E\colon\det A(x)\ne0\}.
\end{equation}
The condition defining $G$ is sufficient for the bracket argument
below.  Its complement is not asserted to be the complete set of
points where a H\"ormander condition fails.

\begin{lem}\label{lem:repair-brackets}
The parabolic H\"ormander condition holds on $\mathbb R\times G$.
In particular, the forward and backward parabolic operators associated
with $\mathcal L$ and $\mathcal L^*$ are hypoelliptic there.
\end{lem}

\begin{proof}
We use the convention $[X,Y]=XY-YX$.  Set
\[
 \Gamma_i(y)=\frac{\gamma a\sigma}{d_i(y)},\qquad
 B_0=a^2\sigma,\qquad
 \Gamma_i'(y)=-\frac{\gamma a\sigma}
                    {2\sqrt{y_i}\,d_i(y)^2}.
\]
The first commutator is
\begin{equation}\label{eq:prelim-first-bracket}
 Z_i=[X_i,X_0]
 =-\Gamma_i(y)\partial_{x_i}-B_0\partial_{z_i}.
\end{equation}
Since $\Gamma_i>0$, the vectors $X_i,Z_i$, $1\le i\le m$, span all
$x$- and $z$-directions at every point of $E$.
To see which $y$-directions appear next, put
$q_j(x)=g_j(x)^2+\sigma^2$.  For $1\le i,j\le m$, the components
of the second commutator are
\begin{equation}\label{eq:prelim-second-bracket}
 \begin{aligned}
 [Z_i,X_0]^{x_j}
 &=\delta_{ij}\left(\frac{B_0\gamma}{d_i(y)}
              +b(q_i(x)-y_i)\Gamma_i'(y)\right),\\
 [Z_i,X_0]^{z_j}
 &=-a\Gamma_i(y)(H_f(x))_{ji}+aB_0\delta_{ij},\\
 [Z_i,X_0]^{y_j}
 &=-b\Gamma_i(y)\partial_iq_j(x)
   =-2b\Gamma_i(y)g_j(x)(H_f(x))_{ji}.
 \end{aligned}
\end{equation}
Modulo the already available $x$- and $z$-directions, the columns of
these vectors form the matrix
\[
 -2bA(x)\operatorname{Diag}(\Gamma_1(y),\ldots,\Gamma_m(y)).
\]
This matrix has full rank on $G$, and therefore the displayed
commutators span every spatial direction there. Furthermore, the fields are independent of time.  Replacing $X_0$ by
$X_0-\partial_t$ in these commutators consequently gives exactly the
same spatial vectors.  The field $X_0-\partial_t$ supplies the
remaining time direction.  This proves the parabolic bracket condition.
The argument with $-X_0-\partial_t$ is unchanged up to signs and
gives the condition for the adjoint equation.  The smooth zeroth-order
term in $\mathcal L^*$ does not alter the conclusion of
H\"ormander's theorem \cite{Hormander1967}.
\end{proof}

\begin{rem}\label{rem:prelim-rank-condition}
Having a nonzero component in each coordinate does not mean that a
family of vectors spans all coordinate directions.  Accordingly,
the condition that $A$ have no zero row cannot replace the rank
condition in the preceding proof.  For a concrete example, take
$m=2$ and let $f(x)=\tfrac12(x_1+x_2)^2$ near $x=(1,0)$.
At that point,
\[
 A(x)=\begin{pmatrix}1&1\\1&1\end{pmatrix},
\]
which has no zero row but has rank one.  In this neighborhood
$g_1^2=g_2^2$, so the controlled equations give
$\frac{\mathrm d}{\mathrm d t}(y_1-y_2)=-b(y_1-y_2)$.
At states with $x$ in this neighborhood and $y_1=y_2$, all the
diffusion and drift fields, and hence all their iterated commutators,
are tangent to that hypersurface. The full bracket condition
therefore fails at those states.  Multiplying this
local objective by a smooth cutoff and taking $f(x)=|x|^2$ outside
a larger ball gives an example satisfying
\eqref{ass:smooth}-\eqref{ass:dissipative}, with the same local
obstruction. Conversely, the vanishing of
$\det A$ does not imply failure of all higher bracket conditions.
For example, when $m=1$ and $f(x)=x^2/2$, put
$k=a\sigma$ and $h(y)=\gamma/(\sqrt y+\varepsilon)$.  Then
$Z=[X_1,X_0]=-kh\partial_x-ka\partial_z$ satisfies
\[
 [Z,X_0]^y=-2bkhx,\qquad
 [Z,[Z,X_0]]^y=2bk^2h^2>0.
\]
Together with $X_1$ and $Z$, this recovers all three directions even
at $x=z=0$.  Critical points of the objective need not be points of
H\"ormander degeneracy.
\end{rem}

\begin{prop}\label{prop:repaired-local-stationary-density}
Let $\pi_\infty$ be any invariant probability measure of the homogeneous
system.  Its restriction to $G$ has a nonnegative $C^\infty$ density
$p_\infty$, and
\[
 \mathcal L^*p_\infty=0\quad\hbox{on }G.
\]
In particular, this conclusion applies to the measure in
Theorem~\ref{thm:main-revised}.
\end{prop}

\begin{proof}
Equation \eqref{eq:repaired-stationary-weak} says that
$\mathcal L^*\pi_\infty=0$ distributionally on $E$.  Consider the
time-independent distribution $\mathrm d t\otimes\pi_\infty$ on
$\mathbb R\times E$.  It satisfies
\[
 (\partial_t-\mathcal L^*)(\mathrm d t\otimes\pi_\infty)=0.
\]
Lemma~\ref{lem:repair-brackets} and hypoellipticity show that this
distribution is represented by a smooth function on
$\mathbb R\times G$.  Its time derivative is zero, so that function
is independent of time.  It therefore gives the asserted spatial
density of $\pi_\infty|_G$.  Positivity of the measure and continuity
of the density imply $p_\infty\ge0$.  Its stationary equation holds
distributionally and hence classically.
\end{proof}

\begin{rem} Proposition \ref{prop:repaired-local-stationary-density} is a local regularity statement.  It neither asserts
absolute continuity on the whole state space nor strict positivity
at every point of $G$.  It also does not require an irreducibility
assumption.  In particular, smoothness of the invariant measure on
$G$ and global uniqueness of an invariant measure are distinct
questions.
\end{rem}

\begin{rem}
Convergence to an invariant law should also be distinguished from
concentration at minimizers of $f$.  When $\sigma>0$, the invariant
$x$-marginal cannot be a point mass.  Indeed, if it were
$\delta_{x_*}$, a stationary solution would satisfy $x_t=x_*$ at
every rational time almost surely, and hence at every time by
continuity.  The first equation in \eqref{eq:main-system} would then
force $z_t=0$ for all $t$.  This contradicts the quadratic variation
$[z_i]_t=a^2\sigma^2t$.  Quantitative concentration near minimizers,
including concentration in a small-noise limit, requires estimates
beyond convergence of the probability laws.
\end{rem}

\subsection{The support theorem and controlled trajectories}
\label{subsec:prelim-support}

For $h\in \mathrm L^2([0,T];\mathbb R^m)$, let
$H(t)=\int_0^t h(s)\,\mathrm d s$.  Replacing the Brownian path by
$H$ gives the controlled system
\begin{equation}\label{eq:skeleton-components}
 \left\{
 \begin{aligned}
 \dot x_i(t)&=-\gamma\frac{z_i(t)}{\sqrt{y_i(t)}+\varepsilon},\\
 \dot z_i(t)&=a(g_i(x(t))-z_i(t))+a\sigma h_i(t),\\
 \dot y_i(t)&=b(g_i(x(t))^2+\sigma^2-y_i(t)),
 \qquad i=1,\ldots,m.
 \end{aligned}\right.
\end{equation}
For initial state $u$, its solution is $\Phi_T(u,H)$.
The equations hold almost everywhere in time and all trajectories
exist on the full finite interval by
Lemma~\ref{lem:prelim-solution-map}.  Define the reachable set
\begin{equation}\label{eq:prelim-reachable-set}
 \mathcal R_T(u)
 =\left\{F_T(u,H)\colon
      H(t)=\int_0^t h(s)\,\mathrm d s,
      \ h\in \mathrm L^2([0,T];\mathbb R^m)\right\}.
\end{equation}
In this notation the support theorem of Stroock-Varadhan
\cite{StroockVaradhan1972,StroockVaradhan1979} has the following form.
Because the noise is additive, we can prove it
directly from the solution map, without appealing to a general
approximation theorem for stochastic integrals.

\begin{thm}\label{thm:support}
Let $g$ be globally Lipschitz, $T>0$, and $u\in\overline E$.
Then
\begin{equation}\label{eq:prelim-support-equality}
 \operatorname{supp}P_T(u,\cdot)=\overline{\mathcal R_T(u)},
\end{equation}
where the closure is taken in $\overline E$.
More precisely, the support of the law of the full trajectory in
$C([0,T];\overline E)$ is the uniform closure of the controlled
trajectories $\Phi_T(u,H)$ with $\dot H\in \mathrm L^2$ and $H(0)=0$.
\end{thm}

\begin{proof}
We first recall why Brownian motion has full support in
$C_0([0,T];\mathbb R^m)$ for the uniform norm.  Fix a continuous
path $w$ with $w(0)=0$ and a tolerance $\eta>0$.  Let $I_n$ denote
piecewise linear interpolation on the uniform grid with mesh $T/n$.
Continuity gives $\|I_nw-w\|_{\infty,T}\to0$, and continuity of
Brownian paths gives
$\|B-I_nB\|_{\infty,T}\to0$ almost surely.  Choose $n$ large
enough that
\[
 \|I_nw-w\|_{\infty,T}<\eta/3,
 \qquad
 \mathbb P\bigl(\|B-I_nB\|_{\infty,T}<\eta/3\bigr)>0.
\]
The Gaussian vector of Brownian values at the nonzero grid points
has a strictly positive density on $\mathbb R^{mn}$.  Hence
\[
 \mathbb P\bigl(\|I_nB-I_nw\|_{\infty,T}<\eta/3\bigr)>0.
\]
The residual bridges $B-I_nB$ are independent of the vector of grid
values.  This follows, for example, by computing their zero
covariances with the grid values and using joint Gaussianity.
The last two events are therefore independent and have a positive
probability of occurring together.  Their intersection implies
$\|B-w\|_{\infty,T}<\eta$.  This proves full support.

Let
\[
 \mathcal H_T=
 \left\{H\in C_0([0,T];\mathbb R^m)\colon
       H(t)=\int_0^t h(s)\,\mathrm d s,
       \ h\in \mathrm L^2([0,T];\mathbb R^m)\right\}.
\]
Piecewise linear paths belong to $\mathcal H_T$ and are uniformly
dense in $C_0([0,T];\mathbb R^m)$.
By Lemma~\ref{lem:prelim-solution-map}, the map
$w\mapsto\Phi_T(u,w)$ is continuous.  Its image of an arbitrary
continuous path is therefore a limit of controlled trajectories.
Conversely, every open neighborhood of $\Phi_T(u,H)$ has a
preimage that is an open neighborhood of $H$ in Wiener path space.
The full support just proved gives that preimage positive Brownian
probability.  Consequently the support of the trajectory law is
exactly the closure of $\Phi_T(u,\mathcal H_T)$.

The same argument applies to the continuous endpoint map $F_T$.
Every open set containing a reachable endpoint has positive
transition probability, and a set outside the closure of the image
of $F_T$ has zero transition probability.  Density of $\mathcal H_T$
then gives \eqref{eq:prelim-support-equality}.
\end{proof}

\begin{rem}\label{thm:supportrem}
The It\^o and Stratonovich versions of the homogeneous system coincide,
because its diffusion vector fields are constant.  This is why the
controlled equation \eqref{eq:skeleton-components} has the displayed
drift without an additional correction.  The proof above also shows
that the support theorem requires neither a H\"ormander condition nor
global controllability.  It identifies the support with the closure
of the reachable set, whatever that set may be.
\end{rem}

\begin{cor}\label{cor:prelim-uniform-accessibility}
Let $K\subset\overline E$ be compact, let $O\subset\overline E$ be
relatively open, and fix $T>0$.  If
$\mathcal R_T(u)\cap O\ne\varnothing$ for every $u\in K$, then
\[
 \inf_{u\in K}P_T(u,O)>0.
\]
\end{cor}

\begin{proof}
The support theorem gives $P_T(u,O)>0$ at every point of $K$.
The weak Feller property implies lower semicontinuity of
$u\mapsto P_T(u,O)$.  For each $u\in K$, there is consequently a
neighborhood on which this probability is at least half its value
at $u$.  A finite subcover of $K$ gives a common positive lower
bound.
\end{proof}

\begin{rem} Accessibility and positivity of densities must be distinguished.
For example, a smooth nonnegative function can vanish at a point
while having positive integral over every nonempty open set.
Thus the support theorem, even when combined with smoothness of a
density, does not imply that the density is strictly positive at
every accessible endpoint.  The minorization argument in
Section~\ref{sec:proof-CorC} instead constructs a regular family of
controls and obtains a lower bound for transition probabilities
from a finite-dimensional Gaussian component of the noise.
\end{rem}

\subsection{Small sets and minorization}\label{small}

A measurable subset $K$ of the state space is called \emph{small}
for $(P_t)$ if there are $t_0>0$, $\alpha>0$, and a probability
measure $\nu$ such that
\begin{equation}\label{petta-}
 P_{t_0}(u,A)\ge\alpha\nu(A)
 \qquad(u\in K,\ A\text{ measurable}).
\end{equation}
The measure $\nu$ is a minorizing measure, and the time and constant
may depend on $K$.  We apply this definition mainly to compact
subsets of $S$ and regard the kernels as kernels on $S$, which is
permitted by forward invariance.

A useful sufficient form of \eqref{petta-} is
\begin{equation}\label{petta-+}
 P_{t_0}(u,A)\ge
 \eta\frac{\operatorname{Leb}(A\cap W)}{\operatorname{Leb}(W)},
 \qquad u\in K,
\end{equation}
where $W\subset S$ is measurable with finite positive Lebesgue
measure and $\eta>0$.  The right-hand side is a fixed positive
multiple of normalized Lebesgue measure on $W$.  In particular,
it controls probabilities of all measurable subsets of $W$, not
just the probability of reaching $W$.  This is stronger than the
conclusion of Corollary~\ref{cor:prelim-uniform-accessibility}.
No global transition density is part of the definition.

If $K$ is small at time $t_0$, then it is small at every later
time $t_0+s$, with the same constant and minorizing measure
$\nu P_s$.  Indeed, the semigroup identity gives
\[
 P_{t_0+s}(u,A)
 =\int P_s(v,A)P_{t_0}(u,\mathrm d v)
 \ge\alpha(\nu P_s)(A).
\]
This observation allows a minorization time to be enlarged to
match the time step required by a discrete skeleton in Harris'
theorem.

For completeness, \eqref{petta-} also implies the pairwise estimate
\[
 \|P_{t_0}(u,\cdot)-P_{t_0}(v,\cdot)\|_{\mathrm{TV}}
 \le1-\alpha,\qquad u,v\in K.
\]
To see this, subtract the common part $\alpha\nu$ from each kernel.
If $\alpha<1$, the remainders divided by $1-\alpha$ are probability
measures and have total variation distance at most one.  The case
$\alpha=1$ is immediate.  In probabilistic terms, the common part
can be used to couple two transitions so that they agree with
probability at least $\alpha$.  This is the local mixing input
combined with the Lyapunov drift in the proof of the main theorem.

\section{Construction of Lyapunov functions}\label{sec:lya}

In this section we construct the Lyapunov functions used in the proof of
Theorem~\ref{thm:main-revised}. The objective function $f$ satisfies
\textnormal{(A1)} and \textnormal{(A2)}, stated in
\eqref{ass:smooth}-\eqref{ass:dissipative}. We give two constructions.
The first applies when $b<4a$ and uses a cancellation between the objective
and a momentum energy adapted to the denominator in the $x$ equation.
The second applies when $\gamma L_f\le a(\sigma/\sqrt2+\varepsilon)/4$
and uses the error $z-\nabla f(x)$. Both lead to a geometric drift
inequality for an exponential weight. The constants in these estimates
are independent of the initial state.

We use the growth consequences established in
Section~\ref{subsec:prelim-structural}. In particular, after choosing
suitable positive constants,
\begin{equation}\label{repair:f-growth}
 \begin{gathered}
 c_f|x|^2-C_f\le f(x)-f_*\le C_f(1+|x|^2),
 \qquad f_*:=\min_{\mathbb R^m}f,\\
 |\nabla f(x)|^2\le C_g(1+|x|^2),
 \qquad |x|\le C_g(1+|\nabla f(x)|).
 \end{gathered}
\end{equation}
The minimum exists because $f$ is continuous and coercive. For the last
inequality, use $|x|\,|\nabla f(x)|\ge m_0|x|^2-c_0$ when $|x|$ is
large and then enlarge the constant on the remaining ball.

All the calculations below are made on the closed strip
\[
 S=\mathbb R^m\times\mathbb R^m\times[\sigma^2/2,\infty)^m.
\]
By Proposition~\ref{prop:repair-wellposed}, this strip is forward
invariant and every trajectory starting in $\overline E$ belongs to
$S$ for $t\ge t_*=(\log2)/b$. Working on $S$ keeps the denominators
uniformly away from zero and ensures that bounded closed sublevel sets
are compact in the state space.

Write $g=\nabla f(x)$ and set
\[
 r_0=\sigma/{\sqrt2},\qquad d_0=r_0+\varepsilon,
 \qquad d_i=d_i(y)=\sqrt{y_i}+\varepsilon,
 \qquad D=\operatorname{Diag}(d_1,\ldots,d_m).
\]
The generator of the homogeneous system is
\[
 \mathcal L
 =-\gamma D^{-1}z\cdot\nabla_x
   +a(g-z)\cdot\nabla_z
   +b\sum_i(g_i^2+\sigma^2-y_i)\partial_{y_i}
   +\frac{a^2\sigma^2}{2}\Delta_z.
\]
For a smooth function $v$ we use the notation
\[
 \Gamma(v):=\mathcal L(v^2)-2v\mathcal Lv
            =a^2\sigma^2|\nabla_zv|^2.
\]
In particular, the quadratic variation term does not contain derivatives
in $x$ or $y$. This identity refers throughout to the original homogeneous
generator $\mathcal L$. Neither the constructions below nor the moment
arguments in Appendix~\ref{HMI} add diffusion in the other variables.

\subsection{The cancellation energy and its correction}

For the first construction, define
\[
 Y=\sum_i y_i,
 \qquad Q=(1+|x|^2+Y)^{1/2},
 \qquad K=z\cdot D^{-1}z=\sum_i\frac{z_i^2}{d_i}.
\]
Then $d_0\le d_i\le c_dQ$, where $c_d=1+\varepsilon$. The natural
energy $f-f_*+\gamma K/(2a)$ uses the cancellation mechanism from
the deterministic Adam analysis in \cite{BarakatBianchi2021}, with
related energies studied in \cite{BelottoGazeau2020,BarakatEtAl2021}.
For the present diffusion, we add two terms to control the complete
state. The positive term
$\delta Q$ controls the tracker, while the normalized cross term
$-\beta x\cdot z/Q$ turns dissipativity of $f$ into a negative contribution
to the drift. The normalization by $Q$ is used both in the parameter
estimates and in the lower bound for the resulting function.

\begin{prop}\label{repair:lyapunov}
Assume $b<4a$. Then there exist $\beta,\delta,A_0>0$ such that
\begin{equation}\label{repair:W}
 W=f(x)-f_*+\frac{\gamma}{2a}K
       -\beta\frac{x\cdot z}{Q}+\delta Q+A_0
\end{equation}
satisfies $W\ge1$ and, on $S$,
\begin{equation}\label{repair:W-equivalence}
 c(1+|x|^2+K+Q)\le W\le C(1+|x|^2+K+Q),
\end{equation}
\begin{equation}\label{repair:W-drift}
 \mathcal LW\le-c(K+Q)+C\le-c'\sqrt W+C',
 \qquad
 \Gamma(W)\le C''W.
\end{equation}
Here $\Gamma(u)=\mathcal L(u^2)-2u\mathcal Lu
=a^2\sigma^2|\nabla_z u|^2$. Consequently, for some $\alpha,\lambda,B>0$
and $R<\infty$, the proper function
\begin{equation}\label{repair:Omega}
 \Omega=\exp(\alpha\sqrt W)
\end{equation}
satisfies the geometric drift inequality
\begin{equation}\label{repair:Omega-drift}
 \mathcal L\Omega\le-\lambda\Omega+B\mathbf1_{\{W\le R\}}.
\end{equation}
All constants depend only on the fixed equation and the chosen parameters,
and not on the initial state.
\end{prop}

\begin{proof}
We first compute the drift of the natural energy. Set $A=\gamma/(2a)$ and
$\mathcal E=f-f_*+AK$. The $x$ derivative of $f$ cancels the gradient
forcing in the $z$ derivative of $AK$. The remaining terms give the
exact identity
\begin{align}\label{repair:natural-energy}
 \mathcal L\mathcal E
 ={}&-\gamma K
 -\frac{\gamma b}{4a}\sum_i
       \frac{z_i^2(g_i^2+\sigma^2-y_i)}{\sqrt{y_i}\,d_i^2}
 +\frac{\gamma a\sigma^2}{2}\sum_id_i^{-1}.
\end{align}
Since $\sqrt{y_i}/d_i\le1$,
\begin{equation}\label{repair:energy-drift}
 \mathcal L\mathcal E\le-dK+C_E,
 \qquad d=\gamma\bigl(1-{b}/{(4a)}\bigr)>0.
\end{equation}

The estimate above dissipates momentum but does not control large values
of $x$ or $y$ when $z$ is small. We introduce the two correction terms
in \eqref{repair:W} to obtain this additional dissipation. Write
\[
 F=\frac{x\cdot z}{Q},\qquad
 X=\frac{|x|^2}{Q},\qquad H=\frac YQ,\qquad
 T=|g|^2+m\sigma^2-Y.
\]
The preceding bounds imply
\begin{equation}\label{repair:elementary}
 |F|\le\sqrt{c_dXK},\qquad
 |T|\le C_TQ^2,\qquad
 |x\cdot D^{-1}z|\le d_0^{-1/2}|x|\sqrt K.
\end{equation}
We also have
\begin{equation}\label{repair:cross-product}
 \frac{|x\cdot z|\,|x\cdot D^{-1}z|}{Q^3}
 \le\sqrt{\frac{c_d}{d_0}}\frac{|x|^2}{Q^{5/2}}K
 \le\sqrt{\frac{c_d}{d_0}}K,
\end{equation}
because $|x|^2\le Q^2$ and $Q\ge1$.

Since $F$ is linear in $z$, its second $z$ derivatives vanish. The exact
formula is
\[
 \mathcal LF
 =-\frac{\gamma K}{Q}
   +\gamma\frac{(x\cdot z)(x\cdot D^{-1}z)}{Q^3}
   +\frac{a x\cdot g}{Q}-aF-\frac{bFT}{2Q^2}.
\]
Using dissipativity and \eqref{repair:elementary}-\eqref{repair:cross-product},
we obtain constants $C_1,C_2,C_3$ such that
\begin{equation}\label{repair:cross-drift}
 \mathcal L(-\beta F)
 \le-a m_0\beta X+C_1\beta K+C_2\beta\sqrt{XK}+C_3\beta.
\end{equation}
Also,
\[
 \mathcal LQ=-\gamma\frac{x\cdot D^{-1}z}{Q}+\frac b{2Q}T,
\]
and therefore
\begin{equation}\label{repair:Q-drift}
 \mathcal LQ\le-\frac b2H+C_4X+C_5\sqrt{XK}+C_6.
\end{equation}
For the first term we used
$|x|\sqrt K/Q=\sqrt{XK}/\sqrt Q\le\sqrt{XK}$.

We now choose the parameters in an order that makes all the estimates
compatible. The constants $C_1,\ldots,C_6$ above are independent of
$\beta$ and $\delta$. Choose first a fixed $k\in(0,1]$ small enough that
$C_4k\le a m_0/4$, and put $\delta=k\beta$. Combining the last estimates
with \eqref{repair:energy-drift} yields
\begin{align*}
 \mathcal LW\le{}&-(d-C_1\beta)K
 -(a m_0\beta-C_4k\beta)X-\frac{bk\beta}{2}H+(C_2+C_5k)\beta\sqrt{XK}+C.
\end{align*}
Young's inequality gives
\[
 (C_2+C_5k)\beta\sqrt{XK}
 \le\frac d4K+\frac{(C_2+C_5k)^2\beta^2}{d}X.
\]
Choose $\beta>0$ sufficiently small that
\[
 C_1\beta\le d/4,\qquad
 \frac{(C_2+C_5k)^2\beta}{d}\le a m_0/4,
 \qquad \beta\le\frac{Ak}{c_d}.
\]
It follows that
\begin{equation}\label{repair:full-dissipation}
 \mathcal LW\le-\frac d2K-\frac{a m_0\beta}{2}X
                    -\frac{bk\beta}{2}H+C
 \le-c(K+Q)+C',
\end{equation}
where we used $X+H=Q-Q^{-1}\ge Q-1$.

It remains to check that these same choices make $W$ coercive, and that
its square root can be exponentiated. For the first point, observe that
\[
 \beta|F|\le\beta\sqrt{c_dQK}
 \le\frac A2K+\frac{\beta^2c_d}{2A}Q
 \le\frac A2K+\frac\delta2Q.
\]
Thus
\[
 f-f_*+\frac A2K+\frac\delta2Q+A_0
 \le W\le
 f-f_*+\frac{3A}{2}K+\frac{3\delta}{2}Q+A_0.
\]
By \eqref{repair:f-growth}, increasing $A_0$ gives $W\ge1$ and
\eqref{repair:W-equivalence}. Moreover,
\[
 W\le C(Q^2+K),\qquad
 \sqrt W\le C(Q+\sqrt K)\le C'(Q+K).
\]
This proves the second drift inequality in \eqref{repair:W-drift}.

The second point uses the fact that the noise acts only in $z$. Since
\[
 \nabla_zW=\frac\gamma aD^{-1}z-\beta\frac xQ,
\]
we have
\[
 \Gamma(W)\le C\left(\sum_i\frac{z_i^2}{d_i^2}
                           +\frac{|x|^2}{Q^2}\right)
 \le C(K+1)\le C'W.
\]
Let $\Psi=\sqrt W$. The chain rule gives
\[
 \mathcal L\Psi
 =\frac{\mathcal LW}{2\sqrt W}
       -\frac{\Gamma(W)}{8W^{3/2}}
 \le-\frac{c'}2+\frac{C'}{2\sqrt W},
 \qquad
 \Gamma(\Psi)=\frac{\Gamma(W)}{4W}\le C_\Psi.
\]
Consequently there are $r>0$ and $R<\infty$ such that
$\mathcal L\Psi\le-r$ on $\{W>R\}$. For $\alpha>0$ sufficiently small,
\[
 \frac{\mathcal L(e^{\alpha\Psi})}{e^{\alpha\Psi}}
 =\alpha\mathcal L\Psi+\frac{\alpha^2}{2}\Gamma(\Psi)
 \le-\frac{\alpha r}{2}
 \qquad\hbox{on }\{W>R\}.
\]
The set $\{W\le R\}\subset S$ is compact by
\eqref{repair:W-equivalence}. In particular, $Q$ controls $y$, and
$|z|^2\le c_dQK$ controls $z$ on a sublevel set. All generator terms
are bounded on this compact set. This proves \eqref{repair:Omega-drift}.
\end{proof}

\begin{rem}\label{Cond1}
Note that the condition $b<4a$ is used when passing from
\eqref{repair:natural-energy} to \eqref{repair:energy-drift}.
Indeed, the potentially positive contribution from the tracker
equation satisfies
\[
 \frac{\gamma b}{4a}\sum_i
 \frac{z_i^2\sqrt{y_i}}{d_i^2}
 =
 \frac{\gamma b}{4a}\sum_i
 \frac{z_i^2}{d_i}\frac{\sqrt{y_i}}{d_i}
 \le \frac{\gamma b}{4a}K.
\]
Combining this bound with the momentum dissipation $-\gamma K$
gives
\[
 \mathcal L\mathcal E\le-dK+C_E,
 \qquad
 d=\gamma\bigl(1-{b}/{(4a)}\bigr)>0.
\]
The positivity of $d$ is subsequently used in Young's inequality
and in the choice of $\beta$ to absorb the correction terms.
Although $\sqrt{y_i}/d_i<1$ at every finite $y_i$, this ratio tends
to $1$ as $y_i\to\infty$. Thus the estimate above gives no uniform
positive dissipation coefficient at $b=4a$. The restriction
$b<4a$ is therefore sufficient for this Lyapunov construction,
but the argument does not establish instability when $b\ge4a$.
\end{rem}

\subsection{A Lyapunov function based on the momentum error}

As discussed in Remark \ref{Cond1}, the preceding construction uses $b<4a$ to preserve a strictly negative
coefficient of $K$ in the natural energy estimate. A different choice of
energy gives a second sufficient condition, which puts no restriction on
$b$. In this case we measure the error $z-g(x)$ between the momentum and the current
gradient, and take $\gamma$ small enough to control the motion of that
gradient along the trajectory. We retain the constants
$r_0=\sigma/\sqrt2$ and $d_0=r_0+\varepsilon$ and again work on $S$.

\begin{prop}\label{prop:repair-sg-lyapunov}
Suppose \eqref{ass:smooth}-\eqref{ass:dissipative} hold and
\begin{equation}\label{eq:repair-sg-restriction}
 \gamma L_f\le{a d_0}/{4}.
\end{equation}
Then there exist a proper smooth function $\Omega\colon S\longrightarrow[1,\infty)$,
constants $\kappa,K>0$, and a compact set $C\subset S$ such that
\begin{equation}\label{eq:repair-sg-geometric}
 \mathcal L\Omega\le-\kappa \Omega+K\mathbf1_C.
\end{equation}
The weight can be chosen as
\begin{equation}\label{eq:repair-sg-weight}
 \Omega=\exp(\theta\sqrt V),\qquad
 V=1+f(x)-f_\ast
       +\lambda|z-g(x)|^2+\eta\sum_{i=1}^m\sqrt{y_i},
\end{equation}
where
\[
 \lambda={\gamma}/{(a d_0)},\qquad
 \eta={\gamma r_0}/{(4bd_0)},
\]
and $\theta>0$ is sufficiently small. The resulting moment estimates are recorded in Appendix~\ref{HMI}.
\end{prop}

\begin{proof}
The growth estimates in \eqref{repair:f-growth} apply. Put
\[
 e=z-g(x),\qquad d_i=\sqrt{y_i}+\varepsilon,\qquad
 M=\operatorname{Diag}(d_1^{-1},\ldots,d_m^{-1}),
\]
and abbreviate
\[
 G=g(x)^{\mathsf T}Mg(x)=\sum_i\frac{g_i(x)^2}{d_i},
 \qquad Q_y=\sum_i\sqrt{y_i}.
\]
Since $d_i\ge d_0$, we have
\[
 \mathcal L f=-\gamma g^{\mathsf T}M(g+e)
 \le-\frac\gamma2G+\frac\gamma{2d_0}|e|^2.
\]
There is no diffusion in $x$. Consequently, with $H=D^2f(x)$,
\[
 \mathrm d e=(-ae+\gamma HM(g+e))\,\mathrm d t
                 +a\sigma\,\mathrm d B,
\]
and hence
\begin{align*}
 \mathcal L|e|^2
 &=-2a|e|^2+2\gamma e^{\mathsf T}HMe
                     +2\gamma e^{\mathsf T}HMg+a^2\sigma^2m\\
 &\le\left(-\frac{3a}{2}+\frac{2\gamma L_f}{d_0}\right)|e|^2
        +\frac{2\gamma^2L_f^2}{a}|Mg|^2+a^2\sigma^2m\\
 &\le-a|e|^2+\frac{2\gamma^2L_f^2}{a d_0}G+a^2\sigma^2m.
\end{align*}
Here Young's inequality was used with the contribution
$(a/2)|e|^2$, and the last line follows from
\eqref{eq:repair-sg-restriction} and $|Mg|^2\le G/d_0$.
Moreover,
\begin{align*}
 \mathcal LQ_y
 &=\frac b2\sum_i\frac{g_i(x)^2+\sigma^2-y_i}{\sqrt{y_i}}\le\frac{bd_0}{2r_0}G-\frac b2Q_y
                      +\frac{bm\sigma^2}{2r_0},
\end{align*}
because $d_i/\sqrt{y_i}\le d_0/r_0$.
Combining these three estimates with the choices of $\lambda,\eta$ gives
\begin{equation}\label{eq:repair-sg-pre-drift}
 \mathcal LV
 \le-\frac\gamma4G-\frac\gamma{2d_0}|e|^2
       -\frac{\gamma r_0}{8d_0}Q_y+C_0,
 \qquad
 C_0=\lambda a^2\sigma^2m+\frac{\eta bm\sigma^2}{2r_0}.
\end{equation}
Indeed, the positive $G$ contribution from $\lambda\mathcal L|e|^2$
is at most $\gamma G/8$, and the one from $\eta\mathcal LQ_y$ is
exactly $\gamma G/8$.

The function $V$ is proper on $S$. More explicitly, the lower bound on
$f$ controls $|x|^2$, the $\lambda|e|^2$ term controls $|e|^2$, and
$\eta Q_y$ controls $y$. Since $|g(x)|\le C(1+|x|)$, these also control
$z=e+g(x)$. Conversely,
\begin{equation}\label{eq:repair-sg-upper-V}
 V\le C(1+|x|^2+|e|^2+Q_y).
\end{equation}
For each coordinate,
$g_i^2/d_i+d_i\ge2|g_i|$. Therefore
\[
 G+Q_y\ge2\sum_i|g_i|-m\varepsilon\ge2|g|-m\varepsilon.
\]
This estimate uses the tracker term together with the normalized
gradient. It remains valid when some $d_i$ become arbitrarily large and
does not require a positive uniform lower bound for the eigenvalues of
$M$.
Dissipativity also gives $|x|\le C(1+|g(x)|)$. It follows from
\eqref{eq:repair-sg-upper-V} that
\[
 \sqrt V\le C(1+|x|+|e|+\sqrt{Q_y})
            \le C(1+G+|e|^2+Q_y).
\]
Thus \eqref{eq:repair-sg-pre-drift} implies, for some $c_1,C_1>0$,
\begin{equation}\label{eq:repair-sg-square-root-drift}
 \mathcal LV\le-c_1\sqrt V+C_1.
\end{equation}
Only the $z$ variables carry diffusion, and
$\nabla_zV=2\lambda e$. Consequently,
\begin{equation}\label{eq:repair-sg-covariance}
 a^2\sigma^2|\nabla_zV|^2
 =4a^2\sigma^2\lambda^2|e|^2
 \le4a^2\sigma^2\lambda V.
\end{equation}
For $\Omega=\exp(\theta\sqrt V)$, the chain rule gives
\begin{align*}
 \frac{\mathcal L\Omega}{\Omega}
 &=\frac\theta{2\sqrt V}\mathcal LV
   +\left(\frac{\theta^2}{8V}-\frac\theta{8V^{3/2}}\right)
                          a^2\sigma^2|\nabla_zV|^2\le-\frac{\theta c_1}{2}
       +\frac{\theta C_1}{2\sqrt V}
       +\frac{\theta^2a^2\sigma^2\lambda}{2}.
\end{align*}
Choose $0<\theta\le c_1/(4a^2\sigma^2\lambda)$ and then choose
$R$ large enough that $C_1/\sqrt R\le c_1/4$.
For $V\ge R$ this yields
\[
 \mathcal L\Omega\le-\frac{\theta c_1}{4}\Omega.
\]
The level set $C=\{V\le R\}$ is compact in $S$, and the generator
is continuous there. Enlarging $K$ proves
\eqref{eq:repair-sg-geometric} with $\kappa=\theta c_1/4$.
\end{proof}

\begin{rem}\label{rem:repair-sg-consequences}
The condition \eqref{eq:repair-sg-restriction} is sufficient for
this Lyapunov construction and allows arbitrary $b>0$. No claim
of necessity is made. Outside the union of the parameter regimes
$b<4a$ and \eqref{eq:repair-sg-restriction}, the present analysis
gives no conclusion about stability.
\end{rem}

\subsection{Existence of invariant measures}

The Lyapunov estimates also give existence of an invariant probability
measure before the accessibility and minorization arguments are used.
We apply the original semigroup on the invariant strip $S$, so no
elliptic regularization or change of the $y$ equation is needed.

\begin{cor}\label{cor:lyapunov-invariant}
Assume \textnormal{(A1)} and \textnormal{(A2)}, and suppose that at least
one of the parameter conditions \eqref{eq:parameter-regimes} holds.
Then the homogeneous system has an invariant probability measure on
$S$. It can be chosen to have finite expectation of the corresponding
weight $\Omega$ constructed above.
\end{cor}

\begin{proof}
In either construction there are $c,B>0$ and a compact set $C\subset S$
such that
\[
 \mathcal L\Omega\le-c\Omega+B\mathbf1_C.
\]
The localized It\^o formula, followed by Fatou's lemma, gives
\begin{equation}\label{eq:lya-semigroup-weight}
 P_t\Omega(u)\le e^{-ct}\Omega(u)+\frac Bc(1-e^{-ct}),
 \qquad u\in S,\quad t\ge0.
\end{equation}
For a fixed $u\in S$, let
\[
 \nu_T=\frac1T\int_0^T P_t(u,\cdot)\,\mathrm d t.
\]
The estimate above implies
\[
 \nu_T\Omega\le\frac{\Omega(u)}{cT}+\frac Bc.
\]
Since $\Omega$ has compact sublevel sets, the family
$(\nu_T)_{T\ge1}$ is tight. By Prokhorov's theorem, there exist
times $T_j\to\infty$ and a probability measure $\pi_\infty$
such that
\[
 \nu_{T_j}\rightharpoonup\pi_\infty
 \qquad\text{as }j\to\infty.
\]
We now use the weak Feller property from
Proposition~\ref{prop:repair-wellposed} to show that
$\pi_\infty$ is invariant. For
$\varphi\in C_b(S)$ and $s>0$,
\[
 |\nu_{T_j}P_s\varphi-\nu_{T_j}\varphi|
 \le\frac{2s\|\varphi\|_\infty}{T_j}\longrightarrow0.
\]
Weak convergence and the Feller property therefore imply
$\pi_\infty P_s\varphi=\pi_\infty\varphi$, proving invariance. Lower semicontinuity
of $\Omega$ gives $\pi_\infty\Omega\le B/c$.
\end{proof}

\begin{rem}
Corollary \ref{cor:lyapunov-invariant} establishes existence through limits of occupation
measures. It does not, by itself, establish convergence of $P_t(u,\cdot)$
or uniqueness of the invariant probability. Those conclusions use the
minorization argument and are proved in Section~\ref{sec:ergodicity}.
The moment consequences, including the dependence on the initial law,
are treated in Appendix~\ref{HMI}.
\end{rem}

\section{Uniqueness of invariant measures and convergence to equilibrium}
\label{sec:ergodicity}

In Section~\ref{sec:lya} we established the existence of an invariant
probability measure. We now prove uniqueness and convergence to
equilibrium, completing the proof of
Theorem~\ref{thm:main-revised}. The argument combines the Lyapunov
estimates with a minorization of transition probabilities on compact
sets, whose proof is deferred to Section~\ref{sec:proof-CorC}.

We first formulate a weighted Harris criterion and apply it to the
autonomous system. This yields exponential convergence in weighted
total variation and in the squared $2$-Wasserstein distance under
the stated moment assumption, together with qualitative total
variation convergence for every initial probability law. We then
establish convergence uniformly over compact sets of initial
states, as needed for the analysis of the time-inhomogeneous system.

Throughout this section, $(P_t)_{t\ge0}$ denotes the original Markov
semigroup of the autonomous system. We continue to work on the
closed, forward invariant strip
\[
 S=\mathbb R^m\times\mathbb R^m\times[\sigma^2/2,\infty)^m,
 \qquad t_*=(\log2)/b.
\]
Every trajectory starting in $\overline E$ belongs to $S$ at time
$t_*$ and remains there thereafter. The restriction of $(P_t)$ to
$S$ is therefore a Markov semigroup.

For a measurable weight $W\ge1$ on $S$ and a finite signed measure
$\eta$ satisfying $\int_S W\,\mathrm d|\eta|<\infty$, recall that
the weighted total variation norm is
\[
 \|\eta\|_W
 =\sup_{|\varphi|\le W}|\eta(\varphi)|
 =\int_S W(u)\,\mathrm d|\eta|(u),
\]
where the supremum is taken over measurable functions $\varphi$.
For probability measures, we use the convention
\[
 \|\mu-\nu\|_{\mathrm{TV}}
 =\sup_{A\in\mathcal B(S)}|\mu(A)-\nu(A)|.
\]
Since $W\ge1$, it follows that
$\|\mu-\nu\|_{\mathrm{TV}}\le\frac12\|\mu-\nu\|_W$.
In the abstract criterion below, $W$ denotes a general weight.
For the present diffusion, it will be chosen as one of the
exponential weights $\Omega$ constructed in
Section~\ref{sec:lya}.

The remaining ingredient is the following minorization property
on compact sets, which will be established for the present
diffusion in Proposition~\ref{prop:compact-minorization} below.
\begin{equation}\label{eq:repair-compact-smallness}
 \begin{gathered}
 \text{For every compact }K\subset S\text{ there exist }
 s_K>0,\quad \alpha_K\in(0,1],\\
 \text{and a probability measure }\nu_K\text{ on }S\text{ such that}\\
 P_{s_K}(u,A)\ge\alpha_K\nu_K(A)
 \quad\text{for all }u\in K\text{ and }A\in\mathcal B(S).
 \end{gathered}
\end{equation}
Thus the transition laws at time $s_K$ from all initial states in
$K$ have a common part $\alpha_K\nu_K$, of mass $\alpha_K$. The following criterion explains how
this common part, together with the drift estimate, yields
uniqueness and convergence to equilibrium.

\subsection{A weighted Harris theorem for the original semigroup}

\begin{prop}\label{prop:repair-harris}
Assume \eqref{eq:repair-compact-smallness}. Suppose that a continuous
proper function $W\colon S\to[1,\infty)$ satisfies
\begin{equation}\label{eq:repair-valid-drift}
 P_tW(u)\le e^{-ct}W(u)+B,
 \qquad u\in S,\quad t\ge0,
\end{equation}
where $c>0$ and $B<\infty$. Then there is a unique invariant probability
measure $\pi_\infty$ on $\overline E$, and $\pi_\infty$ is supported on $S$ with
$\pi_\infty W<\infty$. Furthermore, there are constants $C,\rho>0$ such that
\begin{equation}\label{eq:repair-harris-on-strip}
 \|\mu P_t-\pi_\infty\|_W\le Ce^{-\rho t}\mu W,
 \qquad t\ge0,
\end{equation}
for every probability measure $\mu$ on $S$ satisfying $\mu W<\infty$.
For a probability measure $\mu_0$ on $\overline E$, the corresponding
estimate is
\begin{equation}\label{eq:repair-harris-rate}
 \|\mu_0P_t-\pi_\infty\|_W
 \le Ce^{-\rho(t-t_*)}\mu_0P_{t_*}W,
 \qquad t\ge t_*,
\end{equation}
provided that $\mu_0P_{t_*}W<\infty$.
For every initial probability measure $\mu_0$ on $\overline E$, without
a moment assumption,
\begin{equation}\label{eq:repair-harris-qualitative}
 \|\mu_0P_t-\pi_\infty\|_{\mathrm{TV}}\longrightarrow0
 \qquad\text{as }t\to\infty.
\end{equation}
If $|u|^2\le C_WW(u)$ on $S$, then
\begin{equation}\label{eq:repair-W2-rate}
 \|\mu_0P_t-\pi_\infty\|_{\mathrm{TV}}+(W_2(\mu_0P_t,\pi_\infty))^2
 \le C'e^{-\rho(t-t_*)}\mu_0P_{t_*}W,
 \qquad t\ge t_*,
\end{equation}
under the integrability assumption in \eqref{eq:repair-harris-rate}.
For initial measures on $S$, the same bound holds from time zero with
right-hand side $C'e^{-\rho t}\mu_0W$.
\end{prop}

\begin{proof}
Choose $R>\max\{4B,1\}$ and set
\[
 K=\{u\in S\colon W(u)\le R\}.
\]
This set is compact because $W$ is proper and $S$ is closed.
Let $s_K,\alpha_K,\nu_K$ be supplied by
\eqref{eq:repair-compact-smallness}. We choose a single time
\[
 \tau\ge\max\{s_K,(\log2)/c\}.
\]
For the Markov kernel $P_\tau$, the semigroup property gives
\begin{equation}\label{eq:repair-skeleton-minorization}
 P_\tau(u,\cdot)
 \ge\alpha_K(\nu_KP_{\tau-s_K})(\cdot),
 \qquad u\in K.
\end{equation}
The drift estimate at this same time is
\begin{equation}\label{eq:repair-skeleton-drift}
 P_\tau W\le qW+B,
 \qquad q=e^{-c\tau}\le 1/2.
\end{equation}
In particular, $R>2B/(1-q)$. These are the drift and minorization
hypotheses of the weighted Harris theorem
\cite[Theorems~1.2, 1.3 and 3.2]{HairerMattingly2011b}.
Applied to the single kernel $P_\tau$, that theorem gives an invariant
probability measure $\pi_\infty$ with $\pi_\infty W<\infty$, constants $C_0<\infty$
and $\vartheta\in(0,1)$, and the estimate
\begin{equation}\label{eq:repair-skeleton-convergence}
 \|P_{n\tau}(u,\cdot)-\pi_\infty\|_W
 \le C_0\vartheta^nW(u),
 \qquad n\ge0,\quad u\in S.
\end{equation}
The theorem is sometimes formulated with a norm weighted by
$1+\beta W$, where $\beta>0$. Since $W\ge1$, that norm and
$\|\cdot\|_W$ are equivalent, so the displayed formulation follows
after changing $C_0$.

Uniqueness does not require an integrability assumption on a competing
invariant probability measure. Indeed, let $\widehat\pi_\infty$ be any
$P_\tau$-invariant probability on $S$. For each $u\in S$,
\eqref{eq:repair-skeleton-convergence} implies ordinary total variation
convergence to $\pi_\infty$. Since total variation distances between
probability measures are bounded by one,
\[
 \|\widehat\pi_\infty-\pi_\infty\|_{\mathrm{TV}}
 \le\int_S
       \|P_{n\tau}(u,\cdot)-\pi_\infty\|_{\mathrm{TV}}
       \,\mathrm d\widehat\pi_\infty(u)\longrightarrow0
\]
by dominated convergence. Thus $\widehat\pi_\infty=\pi_\infty$.
For every $r\ge0$, the measure $\pi_\infty P_r$ is also $P_\tau$-invariant,
because $P_r$ and $P_\tau$ commute. It follows that $\pi_\infty P_r=\pi_\infty$.
Hence $\pi_\infty$ is invariant for the full continuous-time semigroup.

The estimate obtained so far holds at the discrete times $n\tau$.
We now extend it to every $t\ge0$ using the semigroup property
and the Lyapunov estimate over the remaining interval of
length less than $\tau$. For a signed
measure $\eta$ with finite $W$ moment and $0\le r\le\tau$,
\begin{align*}
 \|\eta P_r\|_W
 &\le\int_S P_rW(u)\,\mathrm d|\eta|(u)\le\int_S\bigl(W(u)+B\bigr)\,\mathrm d|\eta|(u)
 \le(1+B)\|\eta\|_W.
\end{align*}
Writing $t=n\tau+r$ in \eqref{eq:repair-skeleton-convergence} gives
\begin{equation}\label{eq:repair-harris-pointwise}
 \|P_t(u,\cdot)-\pi_\infty\|_W
 \le Ce^{-\rho t}W(u),
 \qquad t\ge0,\quad u\in S,
\end{equation}
for some $\rho>0$. Integrating against a probability measure $\mu$ on
$S$ proves \eqref{eq:repair-harris-on-strip}.

For an initial probability measure $\mu_0$ on $\overline E$, let
$\mu_*:=\mu_0P_{t_*}$. This measure is supported on $S$, and
$\mu_0P_t=\mu_*P_{t-t_*}$ for $t\ge t_*$. Applying
\eqref{eq:repair-harris-on-strip} to $\mu_*$ proves
\eqref{eq:repair-harris-rate}. Without assuming $\mu_*W<\infty$, we
still have
\[
 \|\mu_0P_t-\pi_\infty\|_{\mathrm{TV}}
 \le\int_S
       \|P_{t-t_*}(u,\cdot)-\pi_\infty\|_{\mathrm{TV}}
       \,\mathrm d\mu_*(u)\longrightarrow0.
\]
This proves \eqref{eq:repair-harris-qualitative} by dominated
convergence. Any invariant probability measure on $\overline E$ is
supported on $S$, since every trajectory enters $S$ by time $t_*$.
The uniqueness already proved on $S$ therefore gives uniqueness on
$\overline E$.

To prove the Wasserstein assertion, let $\mu_1,\mu_2$ have finite
second moments. Couple their common part identically. Their residual
parts have the same mass $r$ and, when $r>0$, are coupled by $r$ times
the product of the two normalized residual probabilities. The elementary
inequality
$|u-v|^2\le2|u|^2+2|v|^2$ gives
\begin{equation}\label{eq:repair-weighted-transport}
 (W_2(\mu_1,\mu_2))^2
 \le2\int_S|u|^2\,\mathrm d|\mu_1-\mu_2|(u)
 \le2C_W\|\mu_1-\mu_2\|_W.
\end{equation}
The moments needed here are finite under the assumptions of
\eqref{eq:repair-harris-rate}, by \eqref{eq:repair-valid-drift} and
$\pi_\infty W<\infty$. Combining \eqref{eq:repair-weighted-transport} with
the weighted total variation estimate proves
\eqref{eq:repair-W2-rate} and its version from time zero on $S$.
\end{proof}

\begin{rem}
The proof of Proposition \ref{prop:repair-harris} uses the drift and minorization at one compatible skeleton
time. No global transition density or global H\"ormander condition is
required. The deterministic entrance time is handled through the
initial distribution $\mu_0P_{t_*}$ and the semigroup identity.
The exponent in \eqref{eq:repair-W2-rate} is the exponent for the
squared Wasserstein distance. The bound for $W_2$ itself has half
that exponent.
\end{rem}

\subsection{Proof of the autonomous convergence theorem: Theorem~\ref{thm:main-revised}}

We now apply the preceding Harris criterion to the autonomous
system. The argument uses the compact-set minorization
\eqref{eq:repair-compact-smallness}, which will be proved in
Proposition~\ref{prop:compact-minorization}. Together with the
existence result from Section~\ref{sec:lya}, the conclusions
below complete the proof of Theorem~\ref{thm:main-revised}
once this minorization has been established.

Proposition~\ref{prop:repair-wellposed} gives the original weak Feller
semigroup on $\overline E$ and entry into $S$ by time $t_*$.
If $b<4a$, take the exponential weight $\Omega$ of
Proposition~\ref{repair:lyapunov}. Under the alternative
small-learning-rate condition, take the weight from
Proposition~\ref{prop:repair-sg-lyapunov}. In either case, the
localized It\^o formula and Fatou's lemma give constants $c>0$ and
$B<\infty$ such that
\[
 P_t\Omega\le e^{-ct}\Omega+B
 \qquad\text{on }S,\quad t\ge0.
\]
Both weights are proper on $S$.

Lemma~\ref{lem:repair-degree} supplies the nondegenerate point required
by Proposition~\ref{prop:compact-minorization}. The target ball in
that proposition can be chosen inside $S$, since every coordinate of
$q(x_*)$ is strictly larger than $\sigma^2$. Normalizing Lebesgue
measure on that ball gives \eqref{eq:repair-compact-smallness}.
Proposition~\ref{prop:repair-harris} now proves existence, uniqueness,
and weighted convergence. For an initial measure supported on $S$,
the estimate holds from time zero. For a general initial measure on
$\overline E$, it holds from time $t_*$ whenever
$\mu_0P_{t_*}\Omega<\infty$.

The weights constructed in Section~\ref{sec:lya} dominate every fixed
polynomial in the complete state $(x,z,y)$. Since $\pi_\infty\Omega<\infty$,
the invariant probability measure has moments of every order.
In particular, \eqref{eq:repair-weighted-transport} gives the stated
bound for $W_2^2$ with the same exponential rate as weighted total
variation. Lemma~\ref{lem:repair-stationary-support} gives the support
assertion. Finally, \eqref{eq:repair-harris-qualitative} proves ordinary
total variation convergence for every initial probability measure,
without any initial moment assumption. The proof of Theorem~\ref{thm:main-revised} is complete.

\begin{rem}
The quantitative estimate requires the exponential moment stated in
Theorem~\ref{thm:main-revised}, whereas the qualitative conclusion has
no moment assumption. Proposition~\ref{prop:polynomial-obstruction}
explains why a fixed polynomial prefactor cannot replace the
exponential weight in a uniform exponential estimate.
\end{rem}

\subsection{Uniform convergence for initial states in compact sets}

The argument for the bias-corrected process will use convergence that
is uniform over a compact set of initial states. We record this
consequence, including the finite-time integrability needed to apply
Theorem~\ref{thm:main-revised}.

\begin{prop}\label{prop:repair-compact-mixing}
Under the assumptions of Theorem~\ref{thm:main-revised}, let $\Omega$
be its exponential weight. For every compact set $K\subset\overline E$,
\begin{equation}\label{eq:repair-compact-entrance-moment}
 M_K:=\sup_{u\in K}P_{t_*}\Omega(u)<\infty.
\end{equation}
Consequently,
\begin{equation}\label{eq:repair-compact-mixing}
 \sup_{u\in K}\|P_t(u,\cdot)-\pi_\infty\|_\Omega
 \le C M_Ke^{-\rho(t-t_*)},
 \qquad t\ge t_*.
\end{equation}
In particular, the convergence is uniform on $K$ in ordinary total
variation and, for bounded continuous $\varphi$,
\[
 \sup_{u\in K}|P_t\varphi(u)-\pi_\infty\varphi|\longrightarrow0.
\]
\end{prop}

\begin{proof}
Lemma~\ref{lem:Omega-entrance-bound} gives
\[
 P_{t_*}\Omega(x,z,y)
 \le C_1\exp\!\left(C_2\bigl(|x|+|z|+|y|^{1/4}\bigr)\right).
\]
Its proof uses the finite-time pathwise bounds from
Section~\ref{Prelim} and the energy bounds from
Section~\ref{sec:lya}. The right-hand side is bounded on $K$,
which proves \eqref{eq:repair-compact-entrance-moment}.
Apply \eqref{eq:repair-harris-rate} with $\mu_0=\delta_u$ and take
the supremum over $K$ to obtain \eqref{eq:repair-compact-mixing}.
\end{proof}

\section{Mixing and minorization}\label{sec:proof-CorC}

The weighted Harris argument in Section~\ref{sec:ergodicity}
shows that, once the Lyapunov estimates of Section~\ref{sec:lya}
are available, it remains to establish the compact-set minorization
property \eqref{eq:repair-compact-smallness}. This is the purpose
of the present section.

As discussed in Section~\ref{Hypoellipticity}, the noise acts
directly only in the $z$-variables, and the bracket condition need
not hold throughout the state space. We construct a family of
controls that reach a common neighborhood from every initial
state in a given compact set, at a common time, and for which
the endpoint map has surjective derivative. Decomposing Brownian
motion into finitely many Gaussian coordinates and an independent
residual path then allows us to apply a change of variables in
those coordinates. This yields the uniform lower bound for
transition probabilities required by the Harris argument.

Throughout this section, $f\in C^2(\mathbb R^m)$ has globally
Lipschitz gradient and satisfies the radial dissipativity
assumption of the paper. The parameters
$a,b,\gamma,\sigma,\varepsilon$ are arbitrary positive numbers.
The minorization argument requires neither of the parameter
restrictions used in the Lyapunov estimates. Recall the state
space $E$ and the forward invariant strip $S$,
\[
 E=\mathbb R^m\times\mathbb R^m\times(0,\infty)^m,
 \qquad
 S=\mathbb R^m\times\mathbb R^m\times[\sigma^2/2,\infty)^m,
\]
and define
\[
 q(x)=\bigl((\partial_{x_i}f(x))^2+\sigma^2\bigr)_{i=1}^m.
\]
By Lemma~\ref{lem:repair-degree}, there is a point
$x_\ast\in\mathbb R^m$ such that
\begin{equation}\label{eq:control-nondegeneracy}
 \det Dq(x_\ast)
 =\det\bigl(2\operatorname{Diag}(\nabla f(x_\ast))
                       \nabla^2f(x_\ast)\bigr)
 \ne0.
\end{equation}
This is the local nondegeneracy condition underlying the
hypoellipticity result in Section~\ref{Hypoellipticity}.
Here we use it to obtain a local inverse of $q$ and construct
a regular family of controls. The resulting surjectivity of
the endpoint derivative yields minorization through a
finite-dimensional Gaussian change of variables. Thus the
argument does not invoke smoothness of transition densities.

Fix such a point $x_\ast$ and set
\[
 q_\ast=q(x_\ast),\qquad u_\ast=(x_\ast,0,q_\ast).
\]
Every component of $\nabla f(x_\ast)$ is nonzero, since a
vanishing component would give a zero row in $Dq(x_\ast)$.
Hence $q_\ast^i>\sigma^2$ for every $i$, and
$u_\ast\in\operatorname{int}S$. This state will serve as the
common target of the control construction. In the controlled
system obtained by replacing $\mathrm dB_t$ with
$h(t)\,\mathrm dt$, the constant control
$h(t)=-g(x_\ast)/\sigma$ keeps a trajectory starting at
$u_\ast$ fixed there.

\begin{prop}\label{prop:compact-minorization}
For every compact set $K\subset E$ there is $T_K>0$ such that, for
every $T\ge T_K$, there are an open ball $W_{K,T}$ about $u_\ast$
and a constant $\kappa_{K,T}>0$ satisfying
\[
 \overline{W_{K,T}}\subset\operatorname{int}S
\]
and
\begin{equation}\label{eq:compact-minorization-repaired}
 P_T(u,A)\ge\kappa_{K,T}\operatorname{Leb}_{3m}(A\cap W_{K,T})
 \qquad (u\in K)
\end{equation}
for every Borel set $A\subset E$. Consequently every compact subset
of $S$ is small at every sufficiently large prescribed time.
\end{prop}

The proof has two parts. First we construct a differentiable family
of controls reaching every point in a common target neighborhood.
We then prove a local minorization lemma and apply compactness to the
initial states. The construction matches the $y$-coordinates at the
final time, including the contribution of the last connecting arc.

\subsection{A regular family of controls}

Fix a compact set $K\subset E$ and choose an open neighborhood
$K^+$ with $K\subset K^+\Subset E$. Shrinking a neighborhood $O$
of $x_\ast$ if necessary, the inverse function theorem gives a
$C^1$ diffeomorphism $q\colon O\longrightarrow q(O)$. Choose
$r>0$ so that
\begin{equation}\label{eq:control-inverse-ball}
 \overline{B(q_\ast,2r)}
 \subset q(O)\cap(\sigma^2,\infty)^m.
\end{equation}
The set of permitted intermediate control points is
\[
 C_x=q^{-1}\bigl(\overline{B(q_\ast,r)}\bigr)\Subset O.
\]
Here $C_x$ is compact. We shall eventually choose the point $c\in C_x$
by solving an equation for the complete terminal $y$-value.

Choose also a compact product neighborhood $\Lambda\subset
\operatorname{int}S$ of $u_\ast$, with nonempty interior. Write its
points as $(\xi,\zeta,\eta)$. They will be possible terminal values
of $(x,z,y)$. The initial and terminal velocities required by the
$x$-equation are
\begin{equation}\label{eq:control-velocities}
 v_0^i(u)=-{\gamma z_0^i}/{(\sqrt{y_0^i}+\varepsilon)},
 \qquad
 v_\sharp^i(\zeta,\eta)
 =-{\gamma\zeta^i}/{(\sqrt{\eta^i}+\varepsilon)}.
\end{equation}
These velocities, and their first derivatives in the parameters, are
bounded on the compact sets under consideration.

The position path has three stages: an initial arc joins $x_0$
to a holding point $c$, a long interval at $c$ reduces the influence
of the initial tracker, and a short final arc reaches the prescribed
position and velocity. We choose $c$ by solving the full terminal
tracker equation, including the contribution of the final arc.
To construct the two connecting arcs, we use the cubic Hermite polynomials
\begin{equation}\label{eq:control-hermite-basis}
 \begin{aligned}
 h_{00}(s)&=2s^3-3s^2+1,& h_{10}(s)&=s^3-2s^2+s,\\
 h_{01}(s)&=-2s^3+3s^2,& h_{11}(s)&=s^3-s^2,
 \end{aligned}
 \qquad 0\le s\le1.
\end{equation}
They satisfy $h_{00}+h_{01}=1$, and $h_{00},h_{01}$ are
nonnegative on $[0,1]$. Their endpoint values and first derivatives
give the four prescribed position and velocity conditions below.

\begin{lem}\label{lem:initial-hermite-path}
For $u=(x_0,z_0,y_0)\in\overline{K^+}$ and $c\in C_x$, the curve
\begin{equation}\label{eq:control-initial-arc}
 x_1(s)=h_{00}(s)x_0+h_{10}(s)v_0(u)+h_{01}(s)c,
 \qquad 0\le s\le1,
\end{equation}
satisfies
\[
 x_1(0)=x_0,\qquad \dot x_1(0)=v_0(u),\qquad
 x_1(1)=c,\qquad \dot x_1(1)=0.
\]
Let $Y_1(u,c)$ be the terminal value of
\[
 \dot y=b(-y+q(x_1)),\qquad y(0)=y_0.
\]
Then $Y_1$ is $C^1$ on a neighborhood of the selected parameters,
and both $Y_1$ and $D_cY_1$ are uniformly bounded there. The
corresponding paths have a common positive lower bound for all
$y$-coordinates and stay in a fixed compact set in the $(x,y)$
variables.
\end{lem}

\begin{proof}
The endpoint assertions follow directly from
\eqref{eq:control-hermite-basis}. The polynomial formula and
compactness give uniform bounds for $x_1$ and its first two time
derivatives. Moreover,
\[
 Y_1(u,c)=e^{-b}y_0+
 b\int_0^1e^{-b(1-s)}q(x_1(s))\,\mathrm ds,
\]
and
\[
 D_cY_1(u,c)=
 b\int_0^1e^{-b(1-s)}Dq(x_1(s))h_{01}(s)\,\mathrm ds.
\]
The asserted differentiability and bounds follow from $q\in C^1$
and compactness of the range of $x_1$. Since $q_i\ge\sigma^2$,
the variation-of-constants formula also gives
\[
 y^i(s)\ge\min\{y_0^i,\sigma^2\},\qquad 0\le s\le1.
\]
The minimum of the right-hand side over the initial compact set is
strictly positive. The same formula supplies an upper bound.
\end{proof}

After the first arc we hold $x=c$ and $z=0$ for a time $S_0>0$.
The control on this interval is $h=-\nabla f(c)/\sigma$, and its
terminal $y$-value is
\begin{equation}\label{eq:control-dwell-value}
 e^{-bS_0}Y_1(u,c)+(1-e^{-bS_0})q(c).
\end{equation}
We then use a short final arc. The length of the holding interval
will be allowed to increase in order to obtain any sufficiently
large prescribed final time.

\begin{lem}\label{lem:control-final-arc}
For $c\in C_x$, $(\xi,\zeta,\eta)\in\Lambda$, and
$0<\delta\le1$, put
\begin{equation}\label{eq:control-final-arc}
 X_\delta(s)=h_{00}(s)c+h_{01}(s)\xi
                  +\delta h_{11}(s)v_\sharp(\zeta,\eta),
 \qquad 0\le s\le1.
\end{equation}
The physical-time curve $x(t)=X_\delta(t/\delta)$ on $[0,\delta]$
starts at $c$ with velocity zero and ends at $\xi$ with velocity
$v_\sharp(\zeta,\eta)$. Its range is contained in a common compact
subset of $\mathbb R^m$, independently of $\delta\in(0,1]$.
Furthermore,
\begin{equation}\label{eq:control-final-derivative}
 D_cX_\delta(s)=h_{00}(s)I_m.
\end{equation}
\end{lem}

\begin{proof}
The endpoint values and derivatives of the Hermite polynomials give
the prescribed positions and velocities. The range bound follows
from compactness of $C_x$ and $\Lambda$ and boundedness of the
polynomials on $[0,1]$. Differentiation in $c$ gives
\eqref{eq:control-final-derivative}.
\end{proof}

The estimate on the final arc is a bound on its range, not a claim
that $X_\delta$ converges uniformly to $c$ as $\delta\downarrow0$.
Indeed, $X_\delta(1)=\xi$ for every $\delta$. Its contribution to
the terminal $y$-value will be small because it is integrated over
an interval of length $\delta$.

For the next lemma use a holding time $S_0>0$ between the first and
last arcs, and put $T=1+S_0+\delta$. Let
$\mathcal Y(u,c,\xi,\zeta,\eta)$ denote the final $y$-value
obtained by following these three pieces. The last arc is translated
to $[1+S_0,T]$.

\begin{lem}\label{lem:control-terminal-matching}
There are $S_\ast>0$ and a fixed $\delta\in(0,1]$ such that the
following holds for every $S_0\ge S_\ast$. There is an open
neighborhood $\Lambda_0$ of $u_\ast$, with
$\overline{\Lambda_0}\subset\operatorname{int}\Lambda$, such that
for every $u\in K^+$ and $(\xi,\zeta,\eta)\in\Lambda_0$ there is
a control point
\[
 c=c(u,\xi,\zeta,\eta)\in C_x
\]
satisfying
\begin{equation}\label{eq:control-terminal-y}
 \mathcal Y(u,c(u,\xi,\zeta,\eta),\xi,\zeta,\eta)=\eta.
\end{equation}
For each fixed $S_0$, the control point is $C^1$ in the initial state
and target. The same $S_\ast$, $\delta$, and $\Lambda_0$ work for
all the indicated initial states and holding times.
\end{lem}

\begin{proof}
The complete terminal value is
\begin{align}\label{eq:control-full-terminal-value}
 \mathcal Y
 &=e^{-b(S_0+\delta)}Y_1(u,c)
   +e^{-b\delta}(1-e^{-bS_0})q(c)\notag\\
 &\quad+b\delta\int_0^1e^{-b\delta(1-s)}q(X_\delta(s))\,\mathrm ds
 =q(c)+\mathcal R(u,c,\xi,\zeta,\eta),
\end{align}
where
\begin{align}\label{eq:control-terminal-error}
 \mathcal R
 &=e^{-b(S_0+\delta)}(Y_1(u,c)-q(c))\notag+b\delta\int_0^1e^{-b\delta(1-s)}
               (q(X_\delta(s))-q(c))\,\mathrm ds.
\end{align}
All functions in these expressions range over fixed compact sets.
Differentiation in $c$ gives
\begin{align*}
 D_c\mathcal R
 &=e^{-b(S_0+\delta)}(D_cY_1(u,c)-Dq(c))+b\delta\int_0^1e^{-b\delta(1-s)}
       \bigl(Dq(X_\delta(s))h_{00}(s)-Dq(c)\bigr)\,\mathrm ds.
\end{align*}
Consequently there is $C_K<\infty$, independent of $S_0>0$ and
$0<\delta\le1$, such that
\begin{equation}\label{eq:full-endpoint-C1-error}
 \sup\bigl(|\mathcal R|+\|D_c\mathcal R\|\bigr)
 \le C_K(e^{-bS_0}+\delta).
\end{equation}
The supremum is over $u\in\overline{K^+}$, $c\in C_x$, and
targets in $\Lambda$. This estimate accounts for the whole last arc.
It uses its bounded range and its short duration.

Use the coordinates $w=q(c)$ and set
\[
 \widetilde{\mathcal R}(u,w,\xi,\zeta,\eta)
 =\mathcal R(u,q^{-1}(w),\xi,\zeta,\eta).
\]
Since $Dq^{-1}$ is bounded on $\overline{B(q_\ast,r)}$, we may
first choose $S_\ast$ large and then fix $\delta>0$ small enough
that, for every $S_0\ge S_\ast$,
\begin{equation}\label{eq:control-contraction-bounds}
 \sup|\widetilde{\mathcal R}|\le r/4,
 \qquad
 \sup\|D_w\widetilde{\mathcal R}\|\le1/2.
\end{equation}
Shrink $\Lambda_0$ so that $|\eta-q_\ast|<r/4$ on its closure.
For each initial state and target, the map
\begin{equation}\label{eq:control-fixed-point}
 w\longmapsto\eta-
 \widetilde{\mathcal R}(u,w,\xi,\zeta,\eta)
\end{equation}
is a contraction of the complete ball
$\overline{B(q_\ast,r)}$ into $\overline{B(q_\ast,r/2)}$.
Banach's theorem gives a unique fixed point $w$. Taking
$c=q^{-1}(w)$ proves \eqref{eq:control-terminal-y}.

The derivative in $w$ of
$w+\widetilde{\mathcal R}(u,w,\xi,\zeta,\eta)-\eta$ is
invertible by \eqref{eq:control-contraction-bounds}. The implicit
function theorem therefore shows that the fixed point is $C^1$ in
$u$ and the target. The fixed point remains in the interior of the
ball, and uniqueness makes these local parameterizations agree.
This proves the lemma.
\end{proof}

We next recover the remaining variables and the control. The following
lemma also records the regularity needed in the minorization step.
For $h\in \mathrm L^2([0,T];\mathbb R^m)$, let $\mathcal E_T(u,h)$ be the
endpoint of the skeleton system \eqref{eq:skeleton-components}.
Equivalently, this is the continuous-forcing endpoint with forcing
$H(t)=\int_0^t h(s)\,\mathrm ds$.

\begin{lem}\label{lem:control-regular-family}
Set $T_K=1+S_\ast+\delta$, with the choices of
Lemma~\ref{lem:control-terminal-matching}. For every fixed
$T\ge T_K$, there is a $C^1$ map
\[
 (u,\xi,\zeta,\eta)\longmapsto
 h(u,\xi,\zeta,\eta)\in \mathrm L^2([0,T];\mathbb R^m)
\]
on $K^+\times\Lambda_0$ such that
\begin{equation}\label{eq:control-right-inverse}
 \mathcal E_T(u,h(u,\xi,\zeta,\eta))=(\xi,\zeta,\eta).
\end{equation}
The corresponding paths stay in a common compact subset of $E$
when the initial state and target range over smaller compact
neighborhoods of $K$ and $u_\ast$. In particular,
\begin{equation}\label{eq:control-surjectivity}
 D_h\mathcal E_T(u,h(u,u_\ast))
 \colon \mathrm L^2([0,T];\mathbb R^m)\longrightarrow\mathbb R^{3m}
 \quad\hbox{is onto}
\end{equation}
for every $u\in K$.
\end{lem}

\begin{proof}
Take $S_0=T-1-\delta$ and use the control point of
Lemma~\ref{lem:control-terminal-matching}. Concatenate the initial
arc, the holding interval, and the final arc, obtaining a $C^1$,
piecewise smooth $x$-path. Solve
$\dot y=b(-y+q(x))$ along this path, starting at $y_0$, and put
\begin{equation}\label{eq:control-recovery}
 z^i(t)=-\frac{\sqrt{y^i(t)}+\varepsilon}{\gamma}\dot x^i(t),
 \qquad
 h_i(t)=\frac{\dot z^i(t)-a(\partial_{x_i}f(x(t))-z^i(t))}
                   {a\sigma}.
\end{equation}
At all times,
\begin{equation}\label{eq:control-y-positive}
 y^i(t)=e^{-bt}y_0^i+
 b\int_0^t e^{-b(t-s)}q_i(x(s))\,\mathrm ds
 \ge\min\{y_0^i,\sigma^2\}>0.
\end{equation}
The formula defines a continuous $z$-path, which is $C^1$ on the
three time intervals. The control is piecewise continuous and belongs
to $\mathrm L^2$. The endpoint conditions for $x$ give $z(0)=z_0$,
$z(1)=z(1+S_0)=0$, and, since $y(T)=\eta$,
\[
 z^i(T)=-\frac{\sqrt{\eta^i}+\varepsilon}{\gamma}
                      v_\sharp^i(\zeta,\eta)=\zeta^i.
\]
Thus \eqref{eq:control-right-inverse} holds.

For the differentiability assertion, on each time interval
\[
 \dot z^i=-\frac1\gamma\left[
  (\sqrt{y^i}+\varepsilon)\ddot x^i
  +\frac{b(q_i(x)-y^i)}{2\sqrt{y^i}}\dot x^i\right].
\]
The $x$-path and its first two time derivatives depend $C^1$ on the
parameters on each fixed interval. The integral formula
\eqref{eq:control-y-positive} and the lower bound in that formula
show that $y$, $z$, and the control in \eqref{eq:control-recovery}
also depend $C^1$ on the parameters, with the control viewed in
$\mathrm L^2$. Only $q\in C^1$ and $\nabla f\in C^1$ are used here.
The interval endpoints are fixed for the given $T$, so possible
jumps of the control do not affect this assertion.

Compactness of the parameter sets and the explicit formulas give the
claimed compact path containment. Finally, the endpoint map is $C^1$
in $h$ near each selected control, by the variational equation for
the controlled ODE. Differentiating \eqref{eq:control-right-inverse}
in the $3m$ target variables gives
\[
 D_h\mathcal E_T\,D_{(\xi,\zeta,\eta)}h=I_{3m}.
\]
This proves \eqref{eq:control-surjectivity}.
\end{proof}

\subsection{From a regular control to local minorization}

The support theorem identifies neighborhoods which have positive
probability. To obtain a lower bound by Lebesgue measure we use the
surjective endpoint derivative. We give the argument for general
additive noise. The same regular-control mechanism appears in
\cite[Lemma 4.5]{HerzogMattingly2015}.

\begin{lem}\label{lem:regular-control-minorization}
Let $D\subset\mathbb R^d$ be open and consider a nonexplosive
diffusion
\[
 \mathrm dU_t=\mathfrak b(U_t)\,\mathrm dt+\Sigma\,\mathrm dB_t,
\]
where $\mathfrak b\in C^1(D;\mathbb R^d)$, $\Sigma$ is a
constant $d\times k$ matrix, and $B_t$ is standard Brownian motion
in $\mathbb R^k$. Let $\mathcal E_T(u,h)$ be the endpoint of
$\dot U=\mathfrak b(U)+\Sigma h$. Suppose a control
$h_0\in \mathrm L^2([0,T];\mathbb R^k)$ produces a path from $u_0$ to
$v_0$ which stays in a compact subset of $D$, and suppose
\[
 D_h\mathcal E_T(u_0,h_0)
 \colon \mathrm L^2([0,T];\mathbb R^k)\longrightarrow\mathbb R^d
 \quad\hbox{is onto}.
\]
Then there are neighborhoods $N$ of $u_0$ and $V$ of $v_0$ and $c>0$
such that
\[
 P_T(u,A)\ge c\operatorname{Leb}_d(A\cap V)
 \qquad (u\in N)
\]
for every Borel set $A\subset D$.
\end{lem}

\begin{proof}
Localize the drift in a neighborhood of the controlled path. A
cutoff permits us to work with a globally bounded $C^1$ drift with
bounded derivative. The original and localized equations agree on
a sufficiently small neighborhood of the controlled path.
For a continuous forcing path $w$, starting at zero, the localized
equation is the ordinary integral equation
\[
 U_t=u+\int_0^t\mathfrak b(U_s)\,\mathrm ds+\Sigma w(t).
\]
Its solution depends continuously on $(u,w)$ in the uniform path
topology. Its derivatives in any finitely many Cameron-Martin
directions also depend continuously on $(u,w)$. The latter assertion
follows by differentiating the integral equation and using the
variational equation and continuity of $D\mathfrak b$.

By surjectivity, density of smooth functions in $\mathrm L^2$, and
Gram-Schmidt orthonormalization, choose smooth orthonormal
$e_1,\ldots,e_d\in \mathrm L^2([0,T];\mathbb R^k)$ such that the
endpoint derivative restricted to their span is invertible. Put
$E_j(t)=\int_0^t e_j(s)\,\mathrm ds$ and decompose Brownian motion
as
\begin{equation}\label{eq:control-gaussian-decomposition}
 B_t=R_t+\sum_{j=1}^d Z_jE_j(t),
 \qquad Z_j=\int_0^T e_j(s)\cdot\mathrm dB_s.
\end{equation}
The vector $Z$ is standard Gaussian in $\mathbb R^d$ and is
independent of the continuous Gaussian process $R$. Indeed,
$\mathbb E[B_tZ_j]=E_j(t)$, so $R_t$ is orthogonal to every $Z_j$,
and independence follows from joint Gaussianity.

Write $H_0(t)=\int_0^t h_0(s)\,\mathrm ds$ and define
\[
 a_j=\int_0^T e_j(s)\cdot h_0(s)\,\mathrm ds,
 \qquad r_0=H_0-\sum_{j=1}^d a_jE_j.
\]
Every uniform neighborhood of $r_0$ has positive probability under
the law of $R$. To verify this explicitly, the functionals
\[
 \ell_j(w)=e_j(T)\cdot w(T)
              -\int_0^T e_j'(s)\cdot w(s)\,\mathrm ds
\]
are continuous on continuous paths starting at zero. They satisfy
$\ell_j(B)=Z_j$ almost surely and $\ell_j(H_0)=a_j$. Thus the
continuous linear map
$w\mapsto w-\sum_j\ell_j(w)E_j$ sends Brownian motion to $R$
and $H_0$ to $r_0$. Full support of Wiener measure proves the claim.

Let $F(u,r,z)$ denote the endpoint driven by
$r+\sum_jz_jE_j$. The derivative in $z$ at $(u_0,r_0,a)$ is
invertible. We spell out the uniform inverse argument because the
target neighborhood must be the same for all nearby initial states
and residual paths. Put $A_0=D_zF(u_0,r_0,a)$. By continuity of the
endpoint and its $z$-derivative, choose $\rho>0$, a neighborhood
$N$ of $u_0$, and a uniform-path neighborhood $\mathcal O$ of $r_0$
such that, with $I=B(a,\rho)$,
\[
 \sup_{u\in N,\ r\in\mathcal O,\ z\in\overline I}
   \|A_0^{-1}(D_zF(u,r,z)-A_0)\|\le\tfrac12,
 \qquad
 \sup_{u\in N,\ r\in\mathcal O}
   |A_0^{-1}(F(u,r,a)-v_0)|\le\rho/8.
\]
Shrink a ball $V$ about $v_0$ so that
$|A_0^{-1}(v-v_0)|<\rho/8$ for $v\in V$.
For each $u\in N$, $r\in\mathcal O$, and $v\in V$, the map
\[
 z\longmapsto z-A_0^{-1}(F(u,r,z)-v)
\]
is a contraction of $\overline I$ into $B(a,3\rho/4)$. Its fixed
point belongs to $I$ and satisfies $F(u,r,z)=v$. The derivative
bound also implies, for $z,z'\in I$,
\[
 |A_0^{-1}(F(u,r,z)-F(u,r,z'))|\ge\tfrac12|z-z'|.
\]
Thus $z\mapsto F(u,r,z)$ is one-to-one on $I$, has an invertible
derivative there, and its image contains the common ball $V$.
Moreover, the same derivative bound gives
\[
 |\det D_zF(u,r,z)|\le J<\infty,\qquad z\in I.
\]
All these paths stay in the localization region after shrinking
$\rho$, $N$, and $\mathcal O$ if necessary, by uniform continuity
of the solution map near the controlled forcing $H_0$.

Let $c_I>0$ be a lower bound for the standard Gaussian density on
$I$. By \eqref{eq:control-gaussian-decomposition}, independence, and
change of variables, for every $u\in N$,
\begin{align*}
 P_T(u,A)
 &\ge\int_{\mathcal O}\int_I
      \mathbf1_A(F(u,r,z))\varphi_d(z)\,\mathrm dz\,
      \mathbb P_R(\mathrm dr)\ge\mathbb P(R\in\mathcal O)\frac{c_I}{J}
                         \operatorname{Leb}_d(A\cap V).
\end{align*}
The first inequality uses paths on which the localized and original
equations agree. The probability factor is strictly positive by the
preceding support argument. This proves the lemma without requiring
absolute continuity of the entire transition measure.
\end{proof}

\subsection{Proof of Proposition~\ref{prop:compact-minorization}}

Fix $K$ and any $T\ge T_K$ as in
Lemma~\ref{lem:control-regular-family}. For each $u\in K$, the
control reaching $u_\ast$ has surjective endpoint derivative and
its path is compactly contained in $E$. Apply
Lemma~\ref{lem:regular-control-minorization} with $d=3m$, $k=m$,
and $\Sigma h=(0,a\sigma h,0)$. This gives an initial neighborhood
$N_u$, a target neighborhood $V_u$ of $u_\ast$, and $c_u>0$ such
that
\[
 P_T(v,A)\ge c_u\operatorname{Leb}_{3m}(A\cap V_u),
 \qquad v\in N_u.
\]
Choose a finite cover $N_{u_1},\ldots,N_{u_N}$ of $K$. Since
$u_\ast\in\operatorname{int}S$, there is an open ball $W_{K,T}$
about $u_\ast$ such that
\[
 \overline{W_{K,T}}
 \subset\operatorname{int}S\cap\bigcap_{j=1}^N V_{u_j}.
\]
Taking $\kappa_{K,T}=\min_jc_{u_j}>0$ proves
\eqref{eq:compact-minorization-repaired}. In particular, with
\[
 \nu_{K,T}(A)=
 \frac{\operatorname{Leb}_{3m}(A\cap W_{K,T})}
      {\operatorname{Leb}_{3m}(W_{K,T})},
\]
we have
\[
 P_T(u,\cdot)\ge
 \kappa_{K,T}\operatorname{Leb}_{3m}(W_{K,T})\nu_{K,T}(\cdot),
 \qquad u\in K.
\]
This establishes the small-set property at every $T\ge T_K$. The
constants may depend on $K$ and $T$, while the construction uses the
same terminal point $u_\ast$ for all initial compact sets.\qed

\section{Concluding remarks and future research}\label{Conc}

The stochastic model studied in this paper retains the momentum,
adaptive denominator, and bias-correction coefficients of Adam-type
dynamics. Its purpose is to isolate the interaction of these features
in a system whose long-time behavior can be analyzed. The modified
discrete recursion and its relation to exact Adam are specified
explicitly, while the entrance construction accounts for the
singularity of the bias-correction coefficients at zero. Under the
assumptions and parameter conditions of
Theorem~\ref{thm:main-revised}, the autonomous system has a unique
invariant probability measure and converges to it exponentially
under an exponential Lyapunov moment condition on the initial law.
The time-inhomogeneous process converges weakly to the same law.

We first discuss the scope of these conclusions and the questions
they leave open concerning the parameter range, convergence rates,
and structure of the invariant law. We then turn to the relation
between discrete and continuous long-time behavior and to alternative
stochastic models obtained from different scaling and approximation
choices.

\subsection{Long-time dynamics and invariant laws}

The two Lyapunov constructions describe different ways to control the
adaptive coupling. The cancellation energy is effective when the
relaxation of the tracker is not too fast relative to the momentum
relaxation, as expressed by $b<4a$. The momentum-error energy instead
uses a restriction on $\gamma L_f$ and allows arbitrary $b$. Both give
a drift inequality for the square root of the underlying energy and
a uniform bound for its quadratic-variation density. This combination explains
why exponentiating the square root produces the geometric weight.
Proposition~\ref{prop:polynomial-obstruction} gives a complementary
dynamical explanation. For a quadratic objective, positions of size
$R$ can persist for times proportional to $R$. A uniform exponential
convergence estimate therefore requires a prefactor growing at least
exponentially in $R$ along the family considered there. The two
Lyapunov weights have this growth in their respective parameter
regimes.

The control construction provides the other essential input. Radial dissipativity ensures that the squared-gradient map is locally invertible at some point, even though the objective may have flat regions elsewhere. The long intermediate control stage reduces the influence of the initial tracker, while the final short stage adjusts the endpoint. Solving the full terminal equation gives a regular family of controls, which permits a lower bound for transition probabilities. The distinction between accessibility, local smoothness, and minorization is useful in other degenerate systems as well.

There is a related distinction between mixing and global hypoellipticity. For example, Dou, Salort, and Smets \cite{DouSalortSmets2025} study the diffusion-transport equation
\[
 \partial_tu-\partial_{xx}u+V(x)\partial_yu=0
 \quad\hbox{on }\R_+\times\mathbb T\times\mathbb T.
\]
They obtain uniform minorization under either a two-plateau hypothesis with distinct velocities or a specified local nonflatness condition. These assumptions allow regions where local hypoellipticity fails. Their geometry and recurrence mechanism differ from the present noncompact setting, but the comparison emphasizes that access to a suitable region can suffice for mixing without a bracket condition at every point. The quantitative lower bounds obtained from our control family are not intended to be optimal. More precise estimates of the control cost and the endpoint derivative could give additional information about the dependence of the convergence rate on dimension and the parameters.

Several questions remain concerning the parameter range. The conditions in \eqref{eq:parameter-regimes} are sufficient and do not exhaust all positive choices of $a,b,\gamma,\sigma,\varepsilon$. The proof gives no conclusion about recurrence outside their union. It is therefore natural to ask whether other Lyapunov functions can extend this range, or whether instability occurs for some excluded choices. A related problem is to determine the strongest convergence estimates available under only polynomial assumptions on the initial law. Proposition~\ref{prop:polynomial-obstruction} rules out one particular form of uniform exponential estimate, but does not rule out slower rates or estimates with a different dependence on the initial state.

The invariant law itself also deserves further study. It is supported in $y_i\ge\sigma^2$ and has finite moments of every order. Proposition~\ref{prop:stationary-balances} gives several balance identities, including
\[
 \int y_i\,\mathrm d\pi_\infty
 =\int (\partial_i f)^2\,\mathrm d\pi_\infty+\sigma^2,
 \qquad
 \int\frac{z_i}{\sqrt{y_i}+\varepsilon}\,\mathrm d\pi_\infty=0.
\]
These identities describe constraints imposed by stationarity, but do not identify the law explicitly. For comparison, the overdamped equation
\[
 \mathrm d x_t=-\gamma\nabla f(x_t)\,\mathrm d t+\sigma\,\mathrm dB_t
\]
has the invariant density proportional to $\exp(-2\gamma f/\sigma^2)$. If the adaptive denominator is omitted and one considers
\[
 \mathrm d x_t=-\gamma z_t\,\mathrm d t,
 \qquad
 \mathrm d z_t=a(\nabla f(x_t)-z_t)\,\mathrm d t+a\sigma\,\mathrm dB_t,
\]
the invariant density is proportional to
\[
 \exp\left(-{2}f(x)/{(\gamma\sigma^2)}
                 -{|z|^2}/{(a\sigma^2)}\right).
\]
Both formulas follow by substitution into the corresponding stationary Fokker-Planck equation, and their normalizing constants are finite under \eqref{eq:initial-growth}. No analogous explicit formula for the adaptive system is derived here. Numerical approximation of its stationary Fokker-Planck equation, or expansions that use the linear momentum equation, may help describe the correlations between $x,z,y$.

Convergence to an invariant probability measure should also be distinguished from convergence of the optimizer to a minimizer. The fixed positive noise considered here prevents the invariant $x$-marginal from being a point mass, as explained in Section~\ref{Prelim}. Quantitative concentration near minimizers, small-noise limits, escape from saddles, and metastable transitions require additional estimates. In particular, the present convergence theorem does not imply a generalization statement for the finite-sum optimization problem.

\subsection{Discrete approximations and alternative stochastic models}

The invariant law of the continuous-time model raises the question
of its relation to the long-time behavior of the discrete recursion.
Theorem~\ref{thm:entrance-repair} establishes convergence of the
explicitly modified recursion to its diffusion on bounded time
intervals. Comparing their invariant laws would require stability
estimates uniform in the step size and control of the long-time
approximation error. A corresponding comparison for exact Adam
would require a separate approximation theorem, since its
squared-gradient update differs from the closure used here.

Other stochastic models arise by changing the scaling or the
approximation sought. With independent sampling, fixed gradient
noise having locally uniform finite fourth moments,
$\alpha_h=1-ah$, $\beta_h=1-bh$, and time $t=kh$, local moment
calculations suggest fluctuations of order $\sqrt h$ in both
trackers on fixed time intervals bounded away from zero.
A joint fluctuation theorem also requires suitable regularity
and tightness estimates. The restriction away from zero matters
because the first bias-corrected update retains fluctuations of
order one. The tracker fluctuations are generally correlated,
since the same stochastic gradient enters both updates. A joint
approximation must therefore retain their cross-covariance and
individual variances. Gaussian noise in the squared-gradient
tracker also raises a boundary question, since the approximation
need not preserve $y_i\ge0$. Reflection would be an additional
modeling choice requiring justification, rather than an automatic
consequence of positivity in the discrete recursion.

A related approach uses step-dependent stochastic modified equations.
Monzio Compagnoni et al.\ \cite{CompagnoniEtAl2025} derive weak
approximations with momentum diffusion coefficient
$\sqrt h\,\rho_1\Sigma^{1/2}$, where $\rho_1$ is the momentum
relaxation rate and $\Sigma$ the gradient-noise covariance.
Their parameter regime allows $\rho_1=O(h^{-\zeta})$ with
$0<\zeta<1$, so $\sqrt h$ alone does not describe the coefficient's
size. Their position drift includes a correction involving
$h\rho_1(\nabla f-M)$, where $M$ denotes the momentum, while the
squared-gradient tracker averages the squared gradient and the
diagonal of $\Sigma$. The finite-time weak approximation guarantee
does not identify the full joint fluctuation limit or compare
invariant laws.

Malladi et al.\ \cite{malladi2022sdes} use a different scaling,
letting the noise dominate the gradient and normalizing the
squared-gradient tracker. In their limiting equations, this
tracker follows the diagonal noise covariance, with no
squared-gradient term in its leading drift. The scaling and
approximation hypotheses in these works differ from ours.

These alternatives suggest further problems involving
state-dependent covariance, fluctuations in the adaptive
denominator, and the relation between finite-step dynamics and
stationary laws. Each changes the approximation problem or the
control and Lyapunov analysis and heavy-tailed stochastic gradients
require further changes. The present model retains nonlinear
squared-gradient feedback together with persistent momentum noise.
Its analysis identifies concrete estimates that would need to be
established anew for such extensions.

\medskip
\noindent
{\bf Conflict of interest.}
The author states that there is no conflict of interest.

\medskip
\noindent
{\bf Data availability.}
Data sharing is not applicable to this article as no datasets
were generated or analysed during the current study.

\medskip
\noindent
{\bf Declaration of use of AI.}
The author used GPT-6 Astra, developed by OpenAI, during the revision
of this manuscript. The tool was
used for organization and for assistance with checking mathematical
arguments, clarity, and internal consistency.  The mathematical arguments, results,
and conclusions were developed and verified by the author, who also
checked the references and takes full responsibility for the content of the manuscript.

\appendix

\section{The continuous-time limit. Proof of Theorem~\ref{thm:entrance-repair}}
\label{aaapp}

Throughout this appendix $g=\nabla f$ is globally Lipschitz with
constant $L_f$, and $a,b,\gamma,\sigma,\varepsilon>0$ are fixed.
We use the exact modified recursion \eqref{eq:closed-repair}, with
its uncorrected trackers initialized at zero. The proof works from
time zero and determines the entrance law without assuming
convergence of the discrete state at an intermediate positive time.
The positive lower bound \eqref{eq:closed-positive-repair} is used
throughout.

We first record the expansion of the bias-correction coefficients
away from the singular endpoint. This expansion explains the formal
limiting coefficients. The convergence proof itself uses the exact
averaging kernels and therefore also covers the initial layer.

\begin{lem}\label{lem:bc}
Let $0<\delta<T<\infty$ and set
\[
 \theta_k^{(a,h)}=\frac{ah}{1-(1-ah)^{k+1}},\qquad
 \theta_k^{(b,h)}=\frac{bh}{1-(1-bh)^{k+1}}.
\]
For sufficiently small $h$, uniformly over $t_{k+1}=(k+1)h\in[\delta,T]$,
\[
 \theta_k^{(a,h)}=h a_{t_{k+1}}+O_{\delta,T}(h^2),\qquad
 \theta_k^{(b,h)}=h b_{t_{k+1}}+O_{\delta,T}(h^2).
\]
\end{lem}

\begin{proof}
Write $t=(k+1)h$. Taylor expansion gives
\[
 (k+1)\log(1-ah)=-at-\frac{a^2}{2}th+O_T(h^2).
\]
Hence $(1-ah)^{k+1}=e^{-at}+O_T(h)$ uniformly over $t\le T$.
On $[\delta,T]$, both $1-(1-ah)^{k+1}$ and $1-e^{-at}$ are
bounded below by a positive constant for small $h$. Therefore
\[
 \left|\frac{ah}{1-(1-ah)^{k+1}}-\frac{ah}{1-e^{-at}}\right|
 \le C_{\delta,T}h^2.
\]
The calculation for $b$ is identical.
\end{proof}

\begin{rem}
The $O(h^2)$ coefficient expansion in Lemma~\ref{lem:bc} cannot be
extended uniformly to the first time step. Indeed,
\[
 \theta_0^{(a,h)}=1,\qquad
 h a_h=\frac{ah}{1-e^{-ah}}=1+\frac{ah}{2}+O(h^2),
\]
so the discrepancy at this step is of order $h$. This initial step
must therefore be treated separately.
\end{rem}

\begin{rem} Away from zero, Lemma~\ref{lem:bc} gives
\[
 \theta_k^{(a,h)}
 =h a_{t_{k+1}}+O_{\delta,T}(h^2),
 \qquad t_{k+1}\in[\delta,T],
\]
for every fixed $0<\delta<T$. When this expansion is used in the
stochastic increment $\theta_k^{(a,h)}\xi_k$, the remainder produces
an additional noise term that must also be estimated. Under the
Brownian coupling, write
\[
 \Delta B_k=B_{t_{k+1}}-B_{t_k},
 \qquad \xi_k=\frac{\sigma}{h}\Delta B_k.
\]
Then
\[
 \theta_k^{(a,h)}\xi_k
 =a_{t_{k+1}}\sigma\Delta B_k+r_k^h,
 \qquad
 r_k^h=\frac{\sigma}{h}
       \bigl(\theta_k^{(a,h)}-h a_{t_{k+1}}\bigr)\Delta B_k.
\]
The coefficient estimate and
$\mathbb E|\Delta B_k|^2=mh$ imply
\[
 \mathbb E|r_k^h|^2\le C_{\delta,T}h^3,
 \qquad t_{k+1}\in[\delta,T].
\]
Since the coefficients multiplying the Brownian increments are
deterministic, the remainders are independent centered Gaussian
vectors. Their partial sums therefore form a martingale, and
Doob's inequality gives
\[
 \mathbb E\max_{0\le\ell\le\lfloor T/h\rfloor}
 \left|\sum_{k=0}^{\ell-1}
    \mathbf1_{\{t_{k+1}\in[\delta,T]\}}r_k^h\right|^2
 \le C_{\delta,T}h^2.
\]
Thus the accumulated noise error converges to zero in the maximal
$L^2$ norm on intervals bounded away from zero. Near the initial
time, we retain the exact normalized Gaussian terms and establish
their convergence directly.
\end{rem}

In the following subsections we prove
Theorem~\ref{thm:entrance-repair}. We first establish estimates
for the Gaussian term, then construct the entrance solution,
and finally prove convergence of the modified recursion.

\subsection{The Gaussian term and the entrance equation}

We begin by estimating the normalized Gaussian term that appears
in the entrance equation. Fix $T<\infty$. Constants below may depend on $T$ and the fixed
coefficients. Since $N_t$ is Gaussian with covariance
$a\sigma^2(1-e^{-2at})I_m/2$, for every $p\ge1$,
\begin{equation}\label{eq:noise-integrability-repair}
 \big(\mathbb E|N_t/A(t)|^p\big)^{1/p}\le C_{p,T}t^{-1/2},
 \qquad 0<t\le T.
\end{equation}
Thus $H_T=\int_0^T|N_t|/A(t)\,\mathrm dt$ is finite almost surely
and has every finite moment. In fact, Minkowski's inequality gives
\[
 \|H_r\|_{\mathrm L^p}\le\int_0^r C_{p,T}t^{-1/2}\,\mathrm dt
                    \le C_{p,T}\sqrt r,\qquad 0<r\le T.
\]
The exact covariance of the normalized Gaussian vector is
\[
 \operatorname{Cov}(N_t/A(t))
 =\frac{a\sigma^2}{2}\frac{1+e^{-at}}{1-e^{-at}}I_m.
\]
Multiplication by $t$ makes this covariance tend to $\sigma^2I_m$.
This is the source of the small-time distributional statement about
$\sqrt t\,z_t$ following the theorem.

\subsection{Existence and uniqueness of the entrance process}

The normalized kernels in \eqref{eq:averaging-repair} are probability
densities. If $u,v$ take values in a ball of radius $R$, and
$e(t)=\sup_{r\le t}|u(r)-v(r)|$, then
\[
 |D[u](t)-D[v](t)|\le L_f e(t),\qquad
 |S[u](t)-S[v](t)|\le C_R e(t).
\]
Since $S_i\ge\sigma^2$, the function
$y\mapsto(\sqrt y+\varepsilon)^{-1}$ has a bounded derivative on
$[\sigma^2,\infty)$. The difference between the two integrands in
\eqref{eq:entrance-volterra-repair} is therefore bounded by
\[
 C_R\big(1+|N_t|/A(t)\big)e(t).
\]
Picard iteration on intervals where the integral of this random
coefficient is small gives a unique local continuous solution.
Equivalently, the usual successive-approximation estimates for a
Volterra integral equation apply, since this coefficient is integrable.
More explicitly, fix a Brownian path for which $H_T<\infty$ and
restrict the integral map to a ball in $C([0,t_1];\mathbb R^m)$
centered at the constant path $x_0$. For sufficiently small $t_1$,
the growth estimate for the integrand makes this map preserve the
ball, and
\[
 C_R\int_0^{t_1}(1+|N_t|/A(t))\,\mathrm dt<1
\]
makes it a contraction. On a subsequent interval, the already
constructed part of the path is kept fixed. The same estimate
applies to differences of the extensions, giving continuation and
uniqueness as long as the path remains bounded.
The construction is adapted, as all kernels are causal.
There is no finite-time explosion. Indeed, with
$F(t)=\sup_{r\le t}|x_r|$, linear growth of $g$ and positivity of $S$
give
\[
 F(t)\le |x_0|+C\int_0^t(1+F(r))\,\mathrm dr+C H_t.
\]
Gr\"onwall's inequality extends the solution to $[0,T]$ and gives
$\mathbb E\sup_{t\le T}|x_t|^p<\infty$ for every finite $p$.
In particular, this rules out a finite continuation endpoint.
The normalized averaging operators satisfy
\[
 |D[x](t)|\le C\left(1+\sup_{r\le t}|x_r|\right),\qquad
 |S[x](t)|\le C\left(1+\sup_{r\le t}|x_r|^2\right).
\]
In particular, the quadratic dependence of the $y$-drift on $x$ is
handled through this explicit formula. For every finite $p\ge1$,
\[
 \mathbb E\sup_{0\le t\le T}|y_t|^p
 \le C_{p,T}\left(1+\mathbb E\sup_{0\le t\le T}|x_t|^{2p}\right)
 <\infty.
\]
No global linear-growth bound for the complete three-component drift
is used.
The positivity of $S$, the continuity of $y$ at zero, and the
integrability of $z$ now follow from the formulas and
\eqref{eq:noise-integrability-repair}. In addition,
\[
 \big(\mathbb E|z_t|^p\big)^{1/p}
 \le C_{p,T}(1+t^{-1/2}),\qquad 0<t\le T.
\]
Since $D[x](t)\to g(x_0)$, the covariance calculation above also
proves $\sqrt t\,z_t\Rightarrow N(0,\sigma^2I_m)$.

The normalized integral representations of $z_t$ and $y_t$
show that the products $A(t)z_t$ and $B(t)y_t$ satisfy
\[
\begin{aligned}
 \mathrm d\bigl(A(t)z_t\bigr)
 &=a\bigl(g(x_t)-A(t)z_t\bigr)\,\mathrm d t
   +a\sigma\,\mathrm d B_t,\\
 \mathrm d\bigl(B(t)y_t\bigr)
 &=b\bigl(g(x_t)^{\odot2}+\sigma^2\mathbf1-B(t)y_t\bigr)
   \,\mathrm d t.
\end{aligned}
\]
These equations describe the uncorrected trackers, are nonsingular
at zero, and have zero initial values. For $t>0$, applying the
product rule to the quotients by $A(t)$ and $B(t)$ recovers
the tracker equations in \eqref{eq:entrance-sde-repair}.
Together with \eqref{eq:entrance-volterra-repair}, this establishes
the asserted SDE. The same identities give the stated equivalence of entrance
conditions. Indeed, suppose conversely that a solution of
\eqref{eq:entrance-sde-repair} satisfies
\[
 x_r\to x_0,\qquad A(r)z_r\to0,\qquad B(r)y_r\to0
 \quad\text{as }r\downarrow0.
\]
The product rule gives the two linear equations above.
Integrating them from $r$ to $t$ and letting $r\downarrow0$
removes the initial terms. Continuity of $g$ and the limit
$x_r\to x_0$ ensure that the ordinary integrals extend to zero,
while the stochastic convolution has its continuous extension
there. We thus recover
\[
 z_t=D[x](t)+{N_t}/{A(t)},\qquad y_t=S[x](t),
\]
with the averaging operators defined in
\eqref{eq:averaging-repair}. Substitution into the position
equation, together with $x_t\to x_0$ at zero, yields
\eqref{eq:entrance-volterra-repair}. Uniqueness for that equation
therefore gives uniqueness within the stated entrance class. The constructions on increasing finite time intervals agree
on their overlaps by uniqueness. They consequently define
a single consistent process for all $t>0$.

\subsection{The discrete averaging operators}

We prove convergence using the same normalized-kernel argument. For
$t\in[kh,(k+1)h)$ define
\begin{align*}
 D_h[v](t)&=\frac{ah}{A^h_{k+1}}
          \sum_{j=0}^k\alpha_h^{k-j}g(v(jh)),\quad S_h[v](t)=\frac{bh}{B^h_{k+1}}
          \sum_{j=0}^k\beta_h^{k-j}
                    \big(g(v(jh))^{\odot2}+\sigma^2\mathbf1\big),\\
 N^h_{k+1}&=a\sigma\sum_{j=0}^k\alpha_h^{k-j}
                  (B_{(j+1)h}-B_{jh}),\qquad
 \eta_h(t)=N^h_{k+1}/A^h_{k+1}.
\end{align*}
Both discrete averaging operators have weights summing to one.
Consequently their local Lipschitz bounds are the same as those of
$D,S$, uniformly in $h$. Specifically, for paths bounded by $R$,
\[
 |D_h[u](t)-D_h[v](t)|
       \le L_f\sup_{r\le t}|u(r)-v(r)|,
 \qquad
 |S_h[u](t)-S_h[v](t)|
       \le C_R\sup_{r\le t}|u(r)-v(r)|.
\]
The use of values only at times $jh\le t$ is important in these
causal estimates. For every fixed continuous path $v$,
\begin{equation}\label{eq:kernel-convergence-repair}
 \|D_h[v]-D[v]\|_{\infty,[0,T]}
 +\|S_h[v]-S[v]\|_{\infty,[0,T]}\longrightarrow0.
\end{equation}
To see this, use ordinary uniform Riemann-sum convergence on
$[r,T]$, where the denominators are bounded below. On $[0,r]$,
the fact that the weights sum to one bounds both errors by the
oscillation of $g(v)$ and $g(v)^{\odot2}$ on $[0,r+h]$.
First let $h\to0$ and then $r\to0$.

\subsection{Convergence of the normalized Gaussian terms}

The sequence $N_k^h$ is the Euler approximation to
$\mathrm dN_t=-aN_t\,\mathrm dt+a\sigma\,\mathrm dB_t$.
In this additive linear equation its piecewise constant interpolation
converges uniformly in probability to $N$. This can also be checked
directly from
\[
 N_k^h=a\sigma\left(B_{kh}
       -ah\sum_{j=1}^{k-1}\alpha_h^{k-1-j}B_{jh}\right),\qquad
 N_t=a\sigma\left(B_t-a\int_0^t e^{-a(t-r)}B_r\,\mathrm dr\right),
\]
using continuity of Brownian paths. Hence $\eta_h\to\eta=N/A$
uniformly in probability on each $[r,T]$. Near zero, the Gaussian
variance formula gives, for each coordinate,
\[
 \operatorname{Var}\bigl({N_j^{h,i}}/{A_j^h}\bigr)
 =\frac{a^2\sigma^2h}{1-\alpha_h^2}
                 \frac{1+\alpha_h^j}{1-\alpha_h^j}.
\]
For $jh\le T+1$, the denominator satisfies
$1-\alpha_h^j\ge c_Tjh$ for sufficiently small $h$. Therefore
\[
 \big(\mathbb E|N_j^h/A_j^h|^p\big)^{1/p}
 \le C_{p,T}(jh)^{-1/2},\qquad j\ge1,\quad jh\le T+1.
\]
Minkowski's inequality consequently yields
\[
 \left\|\int_0^r|\eta_h(t)|\,\mathrm dt\right\|_{\mathrm L^p}
 \le C_{p,T}\sqrt{r+h}.
\]
Together with \eqref{eq:noise-integrability-repair} this proves
\begin{equation}\label{eq:noise-convergence-repair}
 \|\eta_h-\eta\|_{\mathrm L^1(0,T)}\longrightarrow0
 \quad\hbox{in probability}.
\end{equation}

\subsection{Uniform bounds and stability of the parameter equation}

The linear interpolation satisfies exactly
\[
 \bar X^{h,i}(t)=x_0^i-\gamma\int_0^t
       \frac{D_{h,i}[\bar X^h](r)+\eta_h^i(r)}
            {\sqrt{S_{h,i}[\bar X^h](r)}+\varepsilon}\,\mathrm dr.
\]
Discrete Gr\"onwall's inequality, the linear growth of $g$, and the
preceding Gaussian estimate show that
\begin{equation}\label{eq:discrete-bound-repair}
 \sup_{0<h<h_0}\mathbb E\sup_{t\le T}|\bar X^h(t)|^p<\infty,
 \qquad p\ge1.
\end{equation}
Here $h_0>0$ is fixed sufficiently small that
$h_0<\min\{1,a^{-1},b^{-1}\}$.
For completeness, the grid-point bound follows from
\[
 \max_{j\le k}|x_j^h|
 \le |x_0|+Ch\sum_{j=0}^{k-1}
           \left(1+\max_{\ell\le j}|x_\ell^h|\right)
       +C\int_0^{kh}|\eta_h(r)|\,\mathrm dr.
\]
The interpolation introduces no larger maximum than its endpoints.

Compare the paths until either $|\bar X^h|$ or $|x|$ reaches $R$.
This is a pathwise localization: $\bar X^h$ uses the next Brownian
increment between grid points, so the resulting exit time need not
be a stopping time for the Brownian filtration. No stopped stochastic
integral is used below.
Let $e_h(t)=\sup_{r\le t}|\bar X^h(r)-x_r|$. The local Lipschitz
bounds above, \eqref{eq:kernel-convergence-repair} applied to the
random continuous path $x$, and the bounded derivative of the
preconditioner imply
\[
 e_h(t)\le C_R\int_0^t(1+|\eta_h(r)|)e_h(r)\,\mathrm dr
       +\rho_h\left(T+\|\eta_h\|_{\mathrm L^1(0,T)}\right)
       +C\|\eta_h-\eta\|_{\mathrm L^1(0,T)},
\]
where $\rho_h\to0$ almost surely. Integral Gr\"onwall and
\eqref{eq:noise-convergence-repair} show convergence in probability
up to this exit time. Indeed, the right-hand remainder is
multiplied by at most
\[
 \exp\left(C_R\left(T+\|\eta_h\|_{\mathrm L^1(0,T)}\right)\right),
\]
which is bounded in probability, while the remainder tends to zero
in probability. The moment bounds for $x$ and
\eqref{eq:discrete-bound-repair} allow $R\to\infty$. This proves
the first convergence in \eqref{eq:entrance-convergence-repair}.

\subsection{Convergence of the trackers and completion of the proof}

The identities defining $D_h,S_h,\eta_h$ now give uniform convergence
of the right-endpoint interpolations of $z_k^h,y_k^h$ on $[\delta,T]$.
For clarity, if $t\in[kh,(k+1)h)$, these interpolations are exactly
\[
 \widehat Z^h(t)=z_{k+1}^h=D_h[\bar X^h](t)+\eta_h(t),\qquad
 \widehat Y^h(t)=y_{k+1}^h=S_h[\bar X^h](t).
\]
Their convergence follows from the parameter convergence, the local
Lipschitz bounds for the averaging operators, and the Gaussian
convergence on intervals bounded away from zero. The left-endpoint
interpolations satisfy $Z^h(t)=\widehat Z^h(t-h)$ and
$Y^h(t)=\widehat Y^h(t-h)$ when $t\ge h$. Shifting by one grid
interval has no effect on the limit, because $z,y$ are uniformly
continuous on compact subintervals of $(0,\infty)$.
Finally,
\[
 \sup_{t\le T}|X^h(t)-x_t|
 \le\sup_{t\le T}|\bar X^h(t)-x_t|
       +\sup_{\substack{r,s\in[0,T]\\|r-s|\le h}}|x_r-x_s|,
\]
which tends to zero in probability. Uniform convergence in
probability on $[\delta,T]$ implies convergence in law in the
Skorokhod topology, and completes the proof.
\qed

\section{Long-term dynamics of the inhomogeneous system}\label{LLD}

We prove Theorem~\ref{thm:inhom-main}. Write $U_t=(x_t,z_t,y_t)$
for the complete state. The coefficients $a_t,b_t$
converge exponentially to $a,b$, but this observation alone does
not identify the limiting distribution. We combine three facts.
First, solutions started in a compact set have transition laws
that approach those of the autonomous process over a fixed time
interval. Second, the inhomogeneous laws are uniformly tight after
a waiting time. Third, the autonomous process mixes uniformly over
compact sets. Together, these properties establish weak convergence.

For every $\delta>0$, the well-posedness argument of
Proposition~\ref{prop:repair-wellposed} extends to the inhomogeneous
system on $[\delta,\infty)$. On each finite interval $[\delta,T]$,
the coefficients $a_t,b_t$ are bounded and positive. The $x,z$ drift
has linear growth, whereas the $y$ drift has quadratic growth in $x$.
The tracker is instead controlled by its explicit formula
\[
\begin{aligned}
 y_t^i={}&e^{-\int_\delta^t b_r\,\mathrm d r}y_\delta^i
       +\sigma^2\left(1-e^{-\int_\delta^t b_r\,\mathrm d r}\right)+\int_\delta^t b_s e^{-\int_s^t b_r\,\mathrm d r}
                         g_i(x_s)^2\,\mathrm d s.
\end{aligned}
\]
Thus the estimates for $x,z$ first give nonexplosion of those components,
and the displayed formula controls $y$. Its lower bound
$y_t^i\ge\sigma^2(1-e^{-b(t-\delta)})$ also gives the integrable
boundary Lipschitz estimate with coefficient
$C(1+(t-\delta)^{-1/2})$ when $y_\delta^i=0$.
Localization and the same difference estimates as in
Proposition~\ref{prop:repair-wellposed} therefore give a unique global
strong solution. Arbitrary initial probability laws are treated by
conditioning on the initial state.

\subsection{A time-dependent Lyapunov criterion for tightness}

\begin{lem}\label{lem:time-dependent-tightness-repair}
Let $\mathcal L_t$ be the generator of the inhomogeneous system, and
let $F\subset E$ be a closed forward invariant region for $t\ge t_0$.
Suppose that, for $t\ge t_0$, there are functions $\mathcal V_t\ge1$
on $F$ to which
the time-dependent It\^o formula applies, locally bounded at each
fixed time, and constants $c,C>0$ such that
\[
 (\partial_t+\mathcal L_t)\mathcal V_t\le-c\mathcal V_t+C.
\]
Suppose also that their sublevel sets are uniformly relatively compact
on $F$, meaning that
$\bigcup_{t\ge t_0}\{\xi\in F:\mathcal V_t(\xi)\le R\}$ is relatively
compact in $E$ for every $R<\infty$. Then the laws of the process are
uniformly tight for $t\ge t_0$, for every initial probability law
supported on $F$ at time $t_0$, without an initial higher-moment
assumption.
\end{lem}

\begin{proof}
Localization in It\^o's formula, followed by Fatou's lemma, gives
\[
 \mathbb E_\xi[\mathcal V_t(U_t)]
 \le e^{-c(t-t_0)}\mathcal V_{t_0}(\xi)+C/c.
\]
For an initial law supported on a compact set, local boundedness of
$\mathcal V_{t_0}$ makes this estimate uniform in $t$. Markov's
inequality and uniform compactness of the sublevel sets give tightness.
For an arbitrary initial law $\mu$, choose a compact $K$ with
$\mu(K)>1-\eta$. Condition on $U_{t_0}\in K$ and apply the preceding
argument. The contribution from initial states outside $K$ has total
mass at most $\eta$, uniformly at all later times. Letting
$\eta\downarrow0$ proves the claim.
\end{proof}

\begin{rem}
For the present system, after a fixed positive waiting time the
coordinates $y^i$ are bounded below by $\sigma^2/2$. Thus uniform
coercivity as $|(x,z,y)|\to\infty$ suffices for the compactness condition
in the preceding lemma on that forward state region. Verifying the
time-dependent drift inequality remains a separate model-specific step.
\end{rem}

\subsection{Comparison on finite intervals}

Let $E=\mathbb R^m\times\mathbb R^m\times(0,\infty)^m$.
Write $P_t$ for the transition semigroup of the homogeneous system,
and $Q_{s,s+t}$ for the transition kernels of the inhomogeneous system
started at $s>0$. Equip probability measures on $E$ with the
bounded-Lipschitz metric
\[
 d_{\mathrm{BL}}(\mu,\nu)
 =\sup_{\|\phi\|_\infty\le1,\ \operatorname{Lip}(\phi)\le1}
       \left|\int_E\phi\,\mathrm d\mu-\int_E\phi\,\mathrm d\nu\right|.
\]

\begin{lem}\label{lem:compact-comparison-repair}
Assume that $g$ is globally Lipschitz. For every compact $K\subset E$
and $T>0$,
\[
 \lim_{s\to\infty}\sup_{\xi\in K}
 d_{\mathrm{BL}}\big(Q_{s,s+T}(\xi,\cdot),P_T(\xi,\cdot)\big)=0.
\]
\end{lem}

\begin{proof}
For $\xi\in K$, let $U^{(s)}_t$ denote the inhomogeneous solution
at physical time $s+t$, started from $\xi$ at time $s$. Let
$\widetilde U_t$ denote the autonomous solution at elapsed time
$t$, with the same initial state. Couple the two solutions using
the same Brownian increments.
For $s\ge s_0>0$, the coefficients $a_{s+t},b_{s+t}$ are uniformly
bounded for $0\le t\le T$. Since $K\subset E$ is compact, its initial
$y$-coordinates are bounded below by some $\eta>0$. The explicit
$y$-equations give a positive deterministic lower bound on these
coordinates throughout the comparison interval, uniform in $s\ge s_0$
and $\xi\in K$. Linear growth of the $x,z$ equations, followed by the
explicit formula for $y$, gives, for every finite $p$,
\[
 \sup_{s\ge s_0,\ \xi\in K}
 \mathbb E_\xi\sup_{0\le t\le T}
       \big(|U^{(s)}_t|^p+|\widetilde U_t|^p\big)<\infty.
\]
Here the estimate for the $y$-component uses the $2p$-moment estimate
for $x$. These moments are available because the initial states range
over a compact set. No bound linear in the initial second moment is
being asserted.

The coefficient differences have the explicit bounds
\[
 0\le a_{s+t}-a
   =\frac{a e^{-a(s+t)}}{1-e^{-a(s+t)}}\le C_{s_0}e^{-as},
 \qquad
 0\le b_{s+t}-b\le C_{s_0}e^{-bs}.
\]
These bounds hold uniformly over $0\le t\le T$. They apply to
the diffusion coefficient as well, since its difference is
$(a_{s+t}-a)\sigma$.

Let $\tau_R$ be the minimum of the two exit times from a ball of
radius $R$, and stop both processes at this common time. On the stopped
region all coefficients are Lipschitz, and their time-dependent
differences are bounded by $C_R e^{-\kappa s}$, where
$\kappa=\min\{a,b\}$. The synchronous coupling, the
Burkholder--Davis--Gundy inequality and Gr\"onwall's inequality give
\[
 \sup_{\xi\in K}\mathbb E_\xi\sup_{0\le t\le T}
  |U^{(s)}_{t\wedge\tau_R}-\widetilde U_{t\wedge\tau_R}|^2
 \le C_{K,R,T,s_0}e^{-2\kappa s}.
\]
The positive lower bound on $y$ ensures that the stopped coefficients
are Lipschitz with a deterministic constant as the derivatives of the
preconditioner do not approach their singularity at $y=0$.
The preceding moment bound gives
\[
 \sup_{s\ge s_0,\,\xi\in K}\mathbb P_\xi(\tau_R\le T)
                         \le C_{K,T,p}R^{-p}.
\]
If $\|\phi\|_\infty\le1$ and $\operatorname{Lip}(\phi)\le1$,
comparison on $\{\tau_R>T\}$ and the bound $|\phi|\le1$ on its
complement imply
\[
 \sup_{\xi\in K}
 \left|\mathbb E_\xi\phi(U^{(s)}_T)
                  -\mathbb E_\xi\phi(\widetilde U_T)\right|
 \le C_{K,R,T,s_0}e^{-\kappa s}+C_{K,T,p}R^{-p}.
\]
First let $s\to\infty$ for fixed $R$, and then let $R\to\infty$.
The estimate is uniform over the defining test functions for
$d_{\mathrm{BL}}$, proving the lemma.
\end{proof}

\begin{thm}\label{thm:same-equilibrium-repair}
Let $\nu_t$ be the law of the inhomogeneous solution started at
$\delta>0$. Suppose that $\{\nu_t:t\ge\delta\}$ is tight in $E$ and
that a probability measure $\pi_\infty$ satisfies
\begin{equation}\label{eq:compact-mixing-repair}
 \lim_{T\to\infty}\sup_{\xi\in K}
       d_{\mathrm{BL}}\big(P_T(\xi,\cdot),\pi_\infty\big)=0
 \qquad\hbox{for every compact }K\subset E.
\end{equation}
Assume also the compact-set approximation in
Lemma~\ref{lem:compact-comparison-repair}. Then $\nu_t\Rightarrow\pi_\infty$
as $t\to\infty$.
\end{thm}

\begin{proof}
Given $\eta>0$, tightness supplies a compact $K\subset E$ with
$\sup_{s\ge\delta}\nu_s(K^c)<\eta$. The Markov property gives
$\nu_{s+T}=\nu_sQ_{s,s+T}$. Therefore
\begin{align*}
 d_{\mathrm{BL}}(\nu_{s+T},\pi_\infty)
 &\le\sup_{\xi\in K}
       d_{\mathrm{BL}}(Q_{s,s+T}(\xi,\cdot),P_T(\xi,\cdot))+\sup_{\xi\in K}
       d_{\mathrm{BL}}(P_T(\xi,\cdot),\pi_\infty)+4\eta.
\end{align*}
Choose $T$ sufficiently large using \eqref{eq:compact-mixing-repair},
and then let $s\to\infty$ using compact-set approximation.
For example, choose $T$ so that the second supremum is below $\eta$.
The first supremum then tends to zero as $s\to\infty$, and
\[
 \limsup_{t\to\infty}d_{\mathrm{BL}}(\nu_t,\pi_\infty)\le5\eta.
\]
The time $T$ is held fixed during this limit, and $t=s+T$ ranges
over all sufficiently large times. Finally let $\eta\downarrow0$.
\end{proof}

\begin{rem}
Uniform tightness and uniform-on-compacts mixing are hypotheses that
must be proved elsewhere for the model under consideration. A unique
invariant probability alone is not a substitute for
\eqref{eq:compact-mixing-repair}. The conclusion here is weak
convergence, and consequently convergence on $\pi_\infty$-continuity sets.
Convergence on every Borel set, or in total variation, requires an
additional argument. If both the compact approximation and compact
mixing hypotheses hold in total variation, the same proof establishes
total-variation convergence.
\end{rem}

\subsection{Uniform tightness under the stability hypotheses}

In this subsection assume the smoothness and dissipativity hypotheses
used in the autonomous Lyapunov estimates. Set
\[
 h_*=\sigma^2/2,\qquad d_0=\sqrt{h_*}+\varepsilon,
 \qquad a_t=\frac{a}{1-e^{-at}},\qquad
 b_t=\frac{b}{1-e^{-bt}}.
\]
The notation $h_*$ is a lower bound for a tracker coordinate and is
unrelated to the discretization step used earlier.

\begin{prop}\label{prop:inhomogeneous-tightness-repair}
Suppose that either
\[
 b<4a,\quad\mbox{or}\quad
 \gamma L_f\le{a d_0}/{4}.
\]
Let the inhomogeneous solution start at $\delta>0$ from any probability
law supported on $\mathbb R^m\times\mathbb R^m\times[0,\infty)^m$.
Its laws are uniformly tight in $E$ after a fixed positive waiting time.
If the autonomous semigroup satisfies the compact mixing conclusion
\eqref{eq:compact-mixing-repair} with invariant probability $\pi_\infty$, then
\[
 \operatorname{Law}(U_t)\Rightarrow\pi_\infty\qquad(t\to\infty).
\]
\end{prop}

\begin{proof}
Since $b_t\ge b$, the explicit second-moment equation gives
\[
 y_t^i\ge\sigma^2\big(1-e^{-b(t-\delta)}\big).
\]
Thus $y_t^i\ge h_*$ for $t\ge\delta+(\log2)/b$.
All the following drift calculations take place on this forward region.

First suppose that $b<4a$. Write
\[
 Y=\sum_i y_i,\qquad Q=\sqrt{1+|x|^2+Y},\qquad
 K=\sum_i\frac{z_i^2}{\sqrt{y_i}+\varepsilon}.
\]
Use the constants $\beta,\delta_0>0$ from the autonomous construction
in Proposition~\ref{repair:lyapunov}, chosen with strict margins,
and define
\[
 W_t=A_0+f(x)-f_*+\frac{\gamma}{2a_t}K
             -\beta\frac{x\cdot z}{Q}+\delta_0Q.
\]
Here $A_0$ is enlarged so that $W_t\ge1$. The autonomous estimates
hold uniformly when the frozen pair $(a_t,b_t)$ is sufficiently close
to $(a,b)$. Indeed, the strict inequalities used to choose the
correction parameters in \eqref{repair:full-dissipation} persist
on a sufficiently small closed neighborhood of $(a,b)$.
In particular, for all sufficiently large $t$,
\[
 \mathcal L_tW_t\le-c(K+Q)+C,
 \qquad W_t\le C(K+Q^2),
 \qquad \Gamma_t(W_t)\le C W_t,
\]
where $\Gamma_t(v)=a_t^2\sigma^2|\nabla_zv|^2$ is the quadratic
variation density of the martingale part of $v(U_t)$.
The only explicit time derivative is
\[
 \partial_tW_t=\frac{\gamma}{2}e^{-at}K.
\]
It is absorbed by the negative $K$-term after increasing the starting
time. Hence
\[
 (\partial_t+\mathcal L_t)W_t\le-c_1\sqrt{W_t}+C_1,
 \qquad \Gamma_t(W_t)\le C_2W_t.
\]
The same exponential square-root calculation as in the autonomous
proof, with a sufficiently small fixed $\theta>0$, gives
\begin{equation}\label{eq:inhomogeneous-geometric-repair}
 (\partial_t+\mathcal L_t)e^{\theta\sqrt{W_t}}
       \le-c_2e^{\theta\sqrt{W_t}}+C_3.
\end{equation}
To make the time-dependent calculation explicit, put
$\Psi_t=\sqrt{W_t}$. The chain rule gives
\[
 (\partial_t+\mathcal L_t)\Psi_t
 =\frac{(\partial_t+\mathcal L_t)W_t}{2\sqrt{W_t}}
       -\frac{\Gamma_t(W_t)}{8W_t^{3/2}}
 \le-\frac{c_1}{2}+\frac{C_1}{2\sqrt{W_t}},
 \qquad
 \Gamma_t(\Psi_t)\le C_2/4.
\]
Consequently
\[
 (\partial_t+\mathcal L_t)e^{\theta\Psi_t}
 \le e^{\theta\Psi_t}
      \left(-\frac{\theta c_1}{2}
             +\frac{\theta C_1}{2\sqrt{W_t}}
             +\frac{\theta^2C_2}{8}\right).
\]
Choose $\theta$ small and then choose a sufficiently large fixed
sublevel threshold for $W_t$. The right-hand side is a negative
multiple of $e^{\theta\Psi_t}$ outside that sublevel set and is
bounded uniformly inside it. This proves
\eqref{eq:inhomogeneous-geometric-repair}.
For all sufficiently large $t$, the coercivity estimate is uniform,
\[
 c(1+|x|^2+K+Q)\le W_t\le C(1+|x|^2+K+Q),
\]
with $c,C>0$ independent of $t$. Thus the sublevel sets are uniformly
compact on $y_i\ge h_*$.

Next, suppose $\gamma L_f\le{a d_0}/{4}$. In this case we use Proposition~\ref{prop:repair-sg-lyapunov}.
Put $e=z-g(x)$, $R=\sum_i\sqrt{y_i}$, and
\[
 \lambda_t=\frac{\gamma}{a_td_0},\qquad
 \eta_t=\frac{\gamma\sqrt{h_*}}{4b_td_0},\qquad
 W_t=1+f(x)-f_*+\lambda_t|e|^2+\eta_tR.
\]
Since $a_t\ge a$, we have
$\gamma L_f\le ad_0/4\le a_td_0/4$, including when
$\gamma L_f=ad_0/4$. The frozen-coefficient estimate
\eqref{eq:repair-sg-pre-drift} therefore gives
\[
 \mathcal L_tW_t
 \le-\frac\gamma4G-\frac\gamma{2d_0}|e|^2
       -\frac{\gamma\sqrt{h_*}}{8d_0}R+C,
 \qquad G=\sum_i\frac{g_i(x)^2}{\sqrt{y_i}+\varepsilon},
\]
where $C$ is uniform for all sufficiently large $t$. In particular,
the negative coefficients remain bounded away from zero at equality.
Since $\lambda_t$ and $\eta_t$ converge to positive constants, the
coercivity estimate is also uniform for all sufficiently large $t$,
\[
 c(1+|x|^2+|e|^2+R)\le W_t\le C(1+|x|^2+|e|^2+R).
\]
The additional derivatives are
\[
 \partial_tW_t
  =\frac{\gamma}{d_0}e^{-at}|e|^2
    +\frac{\gamma\sqrt{h_*}}{4d_0}e^{-bt}R.
\]
For all sufficiently large $t$, these terms are at most half the
corresponding negative terms in the frozen-coefficient estimate.
All remaining estimates in that proposition are uniform for large
$t$, and again give
\eqref{eq:inhomogeneous-geometric-repair}, with possibly different
constants and $\theta$.

In either case, $b<4a$ or $\gamma L_f\le{a d_0}/{4}$, Lemma~\ref{lem:time-dependent-tightness-repair}, applied
at a sufficiently large fixed time, gives eventual uniform tightness
for every initial probability law. No integrability of
$e^{\theta\sqrt{W_t}}$ under that law is needed, because the lemma
uses compact truncation. The laws on any intervening compact time
interval contained in $(\delta,\infty)$ are also tight, by continuity
of sample paths and the positive lower bound on $y$ there.
The convergence conclusion now follows from
Theorem~\ref{thm:same-equilibrium-repair} and
Lemma~\ref{lem:compact-comparison-repair}, restarting the process
at this sufficiently large time, when its law is supported in $E$.
\end{proof}

\begin{rem}
For deterministic or compactly supported initial states, the
exponential weight in \eqref{eq:inhomogeneous-geometric-repair} is
integrable at the sufficiently large starting time used in the proof.
Indeed, on a fixed interval $[\delta,T]$, put
$\widehat B_t=B_t-B_\delta$. Integration by parts gives
\[
 \int_\delta^t a_s\,\mathrm dB_s
   =a_t\widehat B_t-\int_\delta^t a_s'\widehat B_s\,\mathrm d s.
\]
Since $a_s$ is positive and decreasing,
\[
 \sup_{\delta\le t\le T}
       \left|\int_\delta^t a_s\,\mathrm dB_s\right|
 \le a_\delta\sup_{\delta\le t\le T}|\widehat B_t|.
\]
The integral equation for $z$, the bound
$|\dot x|\le\gamma|z|/\varepsilon$, and Gr\"onwall's inequality now give
\[
 \sup(|x|+|z|)\le C\big(1+|x_\delta|+|z_\delta|
                      +\sup|B_t-B_\delta|\big),
 \qquad
 \sup Y\le C\big(Y_\delta+1+\sup|x|^2\big).
\]
In either construction,
$W_t\le C(1+|x|^2+|z|^2+\sqrt Y)$ on the positive forward region.
Consequently $\sqrt{W_t}$ is bounded by a constant times
$1+|x|+|z|+Y^{1/4}$. Gaussian maximal tails give the required
exponential integrability. This observation permits direct moment
estimates for such initial states, but does not create missing moments
of an arbitrary initial distribution.
\end{rem}

\begin{proof}[Proof of Theorem~\ref{thm:inhom-main}]
The compact mixing property is given by
Proposition~\ref{prop:repair-compact-mixing}.
Proposition~\ref{prop:inhomogeneous-tightness-repair} therefore gives
the asserted weak convergence. The entrance solution can be restarted
at any fixed $\delta>0$, where it has a probability law on $E$.
\end{proof}

\section{Higher moments and integrability}\label{HMI}

Let $U_t=(x_t,z_t,y_t)$ denote the solution of the homogeneous system
\eqref{eq:main-system}, with initial law $\mu_0$ on $\overline E$.
In this appendix we record the moment estimates used in the ergodic analysis
and make their dependence on the initial law explicit. We distinguish
finite-time propagation of polynomial moments, uniform estimates under
the exponential Lyapunov weight, and integrability of the invariant
law. The final subsection explains why an exponential convergence
estimate cannot have a fixed polynomial initial-state prefactor.

\subsection{Finite-time estimates and the initial law}

Under global Lipschitz continuity of $\nabla f$, Proposition~
\ref{prop:repair-wellposed} gives, for every $q\ge1$ and $T<\infty$,
\begin{equation}\label{eq:appendix-finite-moments}
 \begin{aligned}
 &\mathbb E_{\mu_0}\!\left[\sup_{0\le s\le T}
       \bigl(|x_s|^{2q}+|z_s|^{2q}+|y_s|^q\bigr)\right]\le C_{T,q}\left(1+
       \mathbb E_{\mu_0}
       \bigl[|x_0|^{2q}+|z_0|^{2q}+|y_0|^q\bigr]\right).
 \end{aligned}
\end{equation}
The constant is independent of $\mu_0$, but the right-hand side contains
the corresponding initial moment. In particular, a finite second
moment in $x_0,z_0$ and a finite first moment in $y_0$ give the same
moments on each bounded time interval. Taking $q=2$ gives a sufficient
condition for finite second moments of the complete state, namely
finite fourth moments in $x_0,z_0$ and a finite second moment in $y_0$.
These are sufficient conditions from the estimate, rather than an
assertion that each of them is necessary. The variation-of-constants formula also shows directly why missing
moments of the tracker cannot be supplied by the dynamics. For every
coordinate,
\begin{equation}\label{eq:appendix-y-lower}
 y_t^i=e^{-bt}y_0^i+\sigma^2(1-e^{-bt})
       +b\int_0^t e^{-b(t-s)}g_i(x_s)^2\,\mathrm d s
       \ge e^{-bt}y_0^i.
\end{equation}

\begin{rem}\label{rem:no-moment-creation}
For any $q>0$, if $\mathbb E_{\mu_0}[(y_0^i)^q]=\infty$ for some
coordinate $i$, then $\mathbb E_{\mu_0}[(y_t^i)^q]=\infty$ at every
finite time. For example, let $x_0=z_0=0$ and choose a positive scalar
$y_0$ with finite first moment and infinite second moment. The resulting
law has infinite second moment in $y$ at every finite time. Its
$W_2$ distance to a probability measure with finite second moment is
therefore infinite. In particular, an assumption involving only
$\mathbb E(|x_0|^2+|z_0|^2+|y_0|)$ is insufficient for a general
finite $W_2$ estimate on the complete state space.
\end{rem}

\subsection{The exponential weight and uniform moments}

Assume from now on \textnormal{(A1)} and \textnormal{(A2)}, together
with at least one of the parameter conditions
\eqref{eq:parameter-regimes}. Let $\Omega$ be the corresponding weight
constructed in Section~\ref{sec:lya}. Thus, on the invariant strip $S$,
\begin{equation}\label{eq:appendix-Omega-drift}
 \mathcal L\Omega\le-c\Omega+B\mathbf1_C,
 \qquad \Omega\ge1,
\end{equation}
for fixed $c,B>0$ and a compact set $C\subset S$.

\begin{lem}\label{lem:Omega-polynomial-domination}
For every $p>0$ there is a constant $C_p<\infty$ such that
\begin{equation}\label{eq:repair-sg-moment-domination}
 1+|u|^p\le C_p\Omega(u),\qquad u=(x,z,y)\in S.
\end{equation}
\end{lem}

\begin{proof}
For the weight in Proposition~\ref{repair:lyapunov}, use the notation
$W,Q,K$ from that proposition. Its coercivity estimate implies
\[
 |x|^2\le CW,\qquad
 |z|^2\le c_dQK\le CW^2,\qquad
 |y|^2\le\biggl(\sum_i y_i\biggr)^2\le Q^4\le CW^4.
\]
Here $W\ge1$, so $1+|u|^p\le C_pW^{2p}$. Every fixed power of $W$
is bounded by a constant multiple of $\exp(\alpha\sqrt W)$. For the second weight, let $V$ be the energy in
\eqref{eq:repair-sg-weight} and write $e=z-g(x)$ and
$Q_y=\sum_i\sqrt{y_i}$. Since $f-f_*\ge0$ and $V\ge1$, the definition
of $V$ and \eqref{repair:f-growth} give
\[
 |x|^2\le CV,\qquad |e|^2\le CV,\qquad Q_y\le CV.
\]
It follows that $|z|^2\le C(|e|^2+1+|x|^2)\le CV$ and
$|y|^2\le Q_y^4\le CV^4$. Consequently,
$1+|u|^p\le C_pV^{2p}\le C'_p\exp(\theta\sqrt V)$.
\end{proof}

\begin{prop}\label{prop:Omega-moment}
If $\mu$ is a probability measure on $S$ with $\mu\Omega<\infty$, then
\begin{equation}\label{eq:appendix-weighted-moment}
 \mathbb E_\mu[\Omega(U_t)]
 \le e^{-ct}\mu\Omega+\frac Bc(1-e^{-ct}),\qquad t\ge0.
\end{equation}
In particular, for every $p>0$,
\begin{equation}\label{eq:appendix-uniform-polynomial}
 \sup_{t\ge0}\mathbb E_\mu[|U_t|^p]
 \le C_p\left(\mu\Omega+\frac Bc\right)<\infty.
\end{equation}
The invariant probability measure $\pi_\infty$ in
Theorem~\ref{thm:main-revised} satisfies $\pi_\infty\Omega<\infty$ and has
finite moments of every positive order.
\end{prop}

\begin{proof}
For a deterministic starting state $u\in S$, stop the process when
$\Omega(U_t)$ first reaches level $N$. It\^o's formula applied to
$e^{ct}\Omega(U_t)$, together with
\eqref{eq:appendix-Omega-drift}, gives
\[
 \mathbb E_u\!\left[e^{c(t\wedge\tau_N)}
             \Omega(U_{t\wedge\tau_N})\right]
 \le\Omega(u)+B\int_0^t e^{cs}\,\mathrm d s.
\]
Nonexplosion and properness of $\Omega$ imply $\tau_N\to\infty$
almost surely. Fatou's lemma therefore yields
\eqref{eq:appendix-weighted-moment} for deterministic initial data,
and integration against $\mu$ proves the general case. The polynomial
bound follows from Lemma~\ref{lem:Omega-polynomial-domination}. Corollary~\ref{cor:lyapunov-invariant} constructs an invariant probability
with finite $\Omega$ moment. Uniqueness in
Theorem~\ref{thm:main-revised} identifies that probability with $\pi_\infty$.
Applying \eqref{eq:repair-sg-moment-domination} proves the claimed
stationary moments.
\end{proof}

\begin{rem} Every trajectory starting in $\overline E$ belongs to $S$ at
time $t_*=(\log2)/b$. Hence the preceding Lyapunov moment estimate
applies with initial law $\mu_0P_{t_*}$. If
$\mu_0P_{t_*}\Omega<\infty$, the Markov property gives
\[
 \mathbb E_{\mu_0}[\Omega(U_t)]
 =(\mu_0P_{t_*})P_{t-t_*}\Omega
 \le e^{-c(t-t_*)}\mu_0P_{t_*}\Omega+\frac Bc,
 \qquad t\ge t_*.
\]
The bound retains its dependence on the Lyapunov moment at
time $t_*$. In particular, it does not give a bound uniform
over initial laws with arbitrarily large
$\mu_0P_{t_*}\Omega$ on an interval containing $t_*$.
\end{rem}

\subsection{Entrance integrability and the quantitative prefactor}

For completeness, the exponential moment required at $t_*$ can be
checked from the initial law by a direct finite-time bound. In the
following lemma the constants may depend on the fixed parameters and
on $t_*$.

\begin{lem}\label{lem:Omega-entrance-bound}
There are constants $C_1,C_2<\infty$ such that, for every deterministic
$u=(x,z,y)\in\overline E$,
\begin{equation}\label{eq:Omega-entrance-bound}
 P_{t_*}\Omega(u)
 \le C_1\exp\!\left(C_2\bigl(|x|+|z|+|y|^{1/4}\bigr)\right).
\end{equation}
In particular, $P_{t_*}\Omega$ is bounded on compact subsets of
$\overline E$. An initial law with finite expectation of the right-hand
side in \eqref{eq:Omega-entrance-bound} satisfies
$\mu_0P_{t_*}\Omega<\infty$.
\end{lem}

\begin{proof}
Fix $T=t_*$ and write $M_T=\sup_{0\le s\le T}|B_s|$.
The pathwise bounds \eqref{eq:prelim-driven-xz-bound} and
\eqref{eq:prelim-driven-y-bound}, with $w=B$, give
\[
 \sup_{s\le T}(|x_s|+|z_s|)
 \le C_T(1+|x|+|z|+M_T)
\]
and
\[
 \sup_{s\le T}\sum_i y_s^i
 \le\sum_i y_i+C_T\left(1+|x|^2+|z|^2+M_T^2\right).
\]
At time $T$ the process belongs to $S$. The definitions of both base
energies imply, respectively,
\[
 \sqrt{W(U_T)},\ \sqrt{V(U_T)}
 \le C\left(1+|x_T|+|z_T|
                      +\left(\sum_i y_T^i\right)^{1/4}\right).
\]
For $W$ this follows from its upper bound in
\eqref{repair:W-equivalence}, using $K\le|z_T|^2/d_0$. For $V$ use
$|g(x_T)|\le C(1+|x_T|)$ in \eqref{eq:repair-sg-weight}.
Consequently, for either choice of $\Omega$,
\[
 \Omega(U_T)
 \le\exp\!\left(C\bigl(1+|x|+|z|+|y|^{1/4}+M_T\bigr)\right).
\]
The Brownian maximal inequality gives Gaussian tails for $M_T$, and
hence $\mathbb E[e^{CM_T}]<\infty$ for every fixed $C<\infty$.
Taking expectations proves the result.
\end{proof}

We next recall why weighted total variation is the appropriate estimate
for the Wasserstein conclusion. If $\mu_1,\mu_2$ are probability measures
on $S$ with finite second moments of the complete state, couple their
common part identically and couple their residual
parts independently. The inequality $|u-v|^2\le2|u|^2+2|v|^2$ gives
\begin{equation}\label{eq:appendix-W2-weighted-TV}
 (W_2(\mu_1,\mu_2))^2
 \le2\int |u|^2\,\mathrm d|\mu_1-\mu_2|(u)
 \le C\|\mu_1-\mu_2\|_\Omega.
\end{equation}
The last inequality follows from
\eqref{eq:repair-sg-moment-domination} with $p=2$, including the
$|y|^2$ contribution. If $\mu_0P_{t_*}\Omega<\infty$, both
$\mu_0P_t$ and the invariant law have finite second moments for
$t\ge t_*$. Applying the weighted total variation estimate in
Theorem~\ref{thm:main-revised} therefore yields
\[
 \|\mu_0P_t-\pi_\infty\|_{\mathrm{TV}}+(W_2(\mu_0P_t,\pi_\infty))^2
 \le C e^{-\rho(t-t_*)}\mu_0P_{t_*}\Omega,
 \qquad t\ge t_*.
\]
The exponent $\rho$ applies to the squared Wasserstein distance. Taking
square roots gives exponent $\rho/2$ for $W_2$ itself. This argument
uses the weighted estimate directly and does not require a tail
truncation or a higher-moment bootstrap.

\subsection{Stationary balance identities}

The moment estimates also justify several identities for the invariant
law. They follow directly from stationarity and do not require a
formula for its density.

\begin{prop}\label{prop:stationary-balances}
Under the hypotheses of Theorem~\ref{thm:main-revised}, the invariant
probability $\pi_\infty$ satisfies, for $i=1,\ldots,m$,
\begin{equation}\label{eq:stationary-balances}
 \begin{aligned}
  \mathbb E_{\pi_\infty}[z_i]&=\mathbb E_{\pi_\infty}[g_i(x)],\quad \mathbb E_{\pi_\infty}[y_i]=\mathbb E_{\pi_\infty}[g_i(x)^2]+\sigma^2,\\
  \mathbb E_{\pi_\infty}\!\bigl[{z_i}/{(\sqrt{y_i}+\varepsilon)}\bigr]&=0,\quad \mathbb E_{\pi_\infty}[z_i^2]=\mathbb E_{\pi_\infty}[z_i g_i(x)]
                          +{a\sigma^2}/{2}.
 \end{aligned}
\end{equation}
\end{prop}

\begin{proof}
Start the process with law $\pi_\infty$. All its one-time laws are then $\pi_\infty$,
and Proposition~\ref{prop:Omega-moment} supplies every polynomial
moment. Since $g$ has linear growth and
$\sqrt{y_i}+\varepsilon\ge d_0$ on $S$, the drift integrands in the
three equations are integrable. Taking expectations of the integral
$z_i$, $y_i$, and $x_i$ equations on $[0,t]$ and using stationarity
gives the first three identities. It\^o's formula gives
\[
 z_i(t)^2-z_i(0)^2
 =2a\int_0^t\bigl(z_i(s)g_i(x_s)-z_i(s)^2\bigr)\,\mathrm d s
  +a^2\sigma^2t+2a\sigma\int_0^t z_i(s)\,\mathrm d B_s^i.
\]
The stochastic integral has zero expectation, because stationarity
and the finite second moment imply
$\mathbb E_{\pi_\infty}\int_0^t z_i(s)^2\,\mathrm d s<\infty$.
Taking expectations and dividing by $2at$ proves the final identity.
\end{proof}

\subsection{Limitations of polynomial initial-state weights}
\label{sec:obstructions}

The exponential factor above reflects a genuine feature of the
dynamics. Large tracker values can slow the return of $x$ to a bounded
region even for the quadratic objective $f(x)=x^2/2$. The following
example applies for arbitrary positive $a,b,\gamma,\sigma,\varepsilon$
and gives a lower bound on the initial-state prefactor in any uniform
exponential convergence estimate.

\begin{prop}\label{prop:polynomial-obstruction}
Let $m=1$, $f(x)=x^2/2$, and
$u_R=(R,R,R^2+\sigma^2)$ for $R\ge1$. For every fixed
$\kappa\in(0,1/(8\gamma))$,
\[
 P_{\kappa R}(u_R,\{x\ge R/2\})\longrightarrow1
 \qquad\text{as }R\to\infty.
\]
Suppose that a fixed probability measure $\nu$ on $\overline E$,
constants $C,\lambda>0$, and a function $\Phi\colon S\to[1,\infty)$ satisfy
\[
 \|P_t(u,\cdot)-\nu\|_{\mathrm{TV}}
 \le C e^{-\lambda t}\Phi(u)
 \qquad(u\in S,\ t\ge0).
\]
Then, for every such $\kappa$ and all sufficiently large $R$,
\[
 \Phi(u_R)\ge\frac1{2C}e^{\lambda\kappa R}.
\]
In particular, no such estimate holds with
$\Phi(u)=1+|u|^p$ for any fixed $p>0$.
\end{prop}

\begin{proof}
Let $\tau_R$ be the first exit time of $x$ from $(R/2,3R/2)$ and set
$T_R=\kappa R$. Variation of constants gives
\[
 z_t=e^{-at}R+a\int_0^t e^{-a(t-s)}x_s\,\mathrm d s+N_t,
 \qquad N_t=a\sigma\int_0^t e^{-a(t-s)}\,\mathrm d B_s.
\]
Integration by parts gives
\[
 N_t=a\sigma\left(B_t-a\int_0^t e^{-a(t-s)}B_s\,\mathrm d s\right),
 \qquad
 \sup_{t\le T}|N_t|\le2a\sigma\sup_{t\le T}|B_t|.
\]
Thus, for
$F_R=\{\sup_{t\le T_R}|N_t|\le R/2\}$, the Brownian maximal
inequality yields
\[
 \mathbb P(F_R^c)\le C\frac{T_R}{R^2}
 =\frac{C\kappa}{R}\longrightarrow0.
\]
Before $\tau_R$, the deterministic part of $z$ lies in
$[R/2,3R/2]$. Hence $|z_t|\le2R$ on $F_R$ for
$t\le T_R\wedge\tau_R$. The formula for $y$ gives, on the same
interval,
\[
 y_t\ge e^{-bt}R^2+(1-e^{-bt})\frac{R^2}{4}
       \ge\frac{R^2}{4}.
\]
Consequently $|\dot x_t|\le4\gamma$. Exiting the interval requires
a displacement of $R/2$, whereas
$4\gamma T_R=4\gamma\kappa R<R/2$. Thus $\tau_R>T_R$ on $F_R$,
which proves the first assertion. In fact, $x_t\in(R/2,3R/2)$
throughout $[0,T_R]$ on this event.

Since $\nu$ is a fixed probability measure,
$\nu(\{x\ge R/2\})\to0$. The first assertion therefore implies
\[
 \|P_{\kappa R}(u_R,\cdot)-\nu\|_{\mathrm{TV}}
 \ge P_{\kappa R}(u_R,\{x\ge R/2\})
       -\nu(\{x\ge R/2\})\longrightarrow1.
\]
For sufficiently large $R$, the proposed convergence estimate gives
\[
 \frac12\le C e^{-\lambda\kappa R}\Phi(u_R),
\]
which proves the lower bound on $\Phi$. Since $|u_R|$ has order
$R^2$, a fixed polynomial prefactor cannot satisfy this bound.
\end{proof}

\begin{rem}
Under the corresponding parameter condition, either Lyapunov weight
constructed in Section~\ref{sec:lya} satisfies
\[
 cR\le\log\Omega(u_R)\le CR
 \qquad\text{for all sufficiently large }R.
\]
Indeed, for the first construction,
$K(u_R)$ and $Q(u_R)$ have order $R$, so
\eqref{repair:W-equivalence} gives $W(u_R)$ of order $R^2$.
For the second construction, $z-g(x)=0$ at $u_R$, and
\eqref{eq:repair-sg-weight} gives $V(u_R)$ of order $R^2$.
Thus exponential growth in $R$ is both necessary for a uniform
exponential estimate and attained by the constructed weights along
this family. This comparison does not identify an optimal
convergence rate.

The same argument also rules out a fixed polynomial prefactor
evaluated after a deterministic waiting time. For every fixed
$s>0$ and $p>0$, \eqref{eq:appendix-finite-moments} bounds
$\mathbb E_{u_R}[1+|U_s|^p]$ by a fixed power of $R$. At
$t=\kappa R$, a proposed bound of the form
\[
 \|P_t(u_R,\cdot)-\nu\|_{\mathrm{TV}}
 \le C e^{-\lambda(t-s)}
       \mathbb E_{u_R}[1+|U_s|^p],\qquad t\ge s,
\]
would therefore tend to zero, whereas the left-hand side tends to
one. The quantitative theorem retains the exponential factor
$\mu_0P_{t_*}\Omega$, while its qualitative total variation
conclusion applies to arbitrary initial probability laws.
\end{rem}

\end{document}